          \def\version{01 February, 2026}            %

 \documentclass[reqno,11pt]{amsart}
 \usepackage{amsmath, amsthm, a4, latexsym, amssymb,comment}
 
\usepackage[unicode]{hyperref}
\hypersetup{
    colorlinks=true, 
    linktoc=all,     
    linkcolor=blue,  
}
\usepackage{srcltx}

\usepackage{xcolor} 

\usepackage{amstext,bbm,nicefrac,physics, tikz}
\usepackage{subcaption}

\usetikzlibrary{calc}
\usetikzlibrary{decorations.pathmorphing}

\def\@rmrk#1#2{\refstepcounter
  {#1}\@ifnextchar[{\@yrmrk{#1}{#2}}{\@xrmrk{#1}{#2}}}

\makeatletter\@addtoreset{equation}{section}\makeatother

 \newfont{\bfit}{cmbxti10 scaled 1200}

\renewcommand{\d}{{\rm{d}}}
 \newcommand{\e}{{\rm e} }
 \newcommand{\bd}{{\rm bd} }

 \newcommand{\eps}{\varepsilon}

 \newcommand{\R}{\mathbb{R}}
\renewcommand{\S}{\mathbb{S}}
 \newcommand{\N}{\mathbb{N}}
 \newcommand{\Z}{\mathbb{Z}}
 \newcommand{\Q}{\mathbb{Q}}

 \newcommand{\Leb}{{\rm Leb\,}}

 \newcommand{\E}{\mathbb{E}}
 \renewcommand{\H}{\mathbb{H}}
 \renewcommand{\P}{\mathbb{P}}
 \def\1{{\mathchoice {1\mskip-4mu\mathrm l} 
{1\mskip-4mu\mathrm l}
{1\mskip-4.5mu\mathrm l} {1\mskip-5mu\mathrm l}}}

 \newcommand{\skrid}{{\mathcal D}}

 \newcommand{\skrif}{{\mathcal F}}
 \newcommand{\Fcal}{{\mathcal F}}
 \newcommand{\skrig}{{\mathcal G}}

 \newcommand{\Ncal}{{\mathcal N}}

  \newcommand{\skrit}{{\mathcal T}}

\newcommand{\natls}{{\mathbb N}}

\newcommand\FF{{\mathcal F}}
\newcommand\GG{{\mathcal G}}

\newcommand\LL{{\mathcal L}}
\newcommand\MM{{\mathcal M}}

\newcommand\PP{{\mathcal P}}

\newcommand\PMF{{\PP\kern-2pt\MM\FF}}

\newcommand\PML{{\PP\kern-2pt\MM\LL}}

\newcommand\ep{\epsilon}

\newcommand\til{\widetilde}
 
 \newcommand{\Ga}{{\Gamma}}

\newenvironment{example}{\refstepcounter{theorem}
{\bf Example \thetheorem\ }\nopagebreak }%
{\nopagebreak {\hfill\rule{2mm}{2mm}}\\ }

\newtheorem{theorem}{Theorem}[section]
\newtheorem{lemma}[theorem]{Lemma}
\newtheorem{cor}[theorem]{Corollary}
\newtheorem{prop}[theorem]{Proposition}

\newtheorem{defn}[theorem]{Definition}

\newtheoremstyle{theorem}{1.5ex}{1.5ex}{\itshape\rmfamily}{}
{\bfseries\rmfamily}{}{2ex}{}

\newcounter{remark}
\newenvironment{remark}{
  \refstepcounter{remark}
  \noindent {\bf Remark \theremark\ }\nopagebreak
}{
  \nopagebreak {\hfill\rule{2mm}{2mm}}\\
}

\def\thebibliography#1{\section*{References}
 \list%
 {\arabic{enumi}.}
  {\settowidth\labelwidth{[#1]}\leftmargin\labelwidth
  \advance\leftmargin\labelsep
  \parsep0pt\itemsep0pt
  \usecounter{enumi}}
  \def\newblock{\hskip .11em plus .33em minus .07em}
  \sloppy          
  \sfcode`\.=1000\relax}

\let\oldtocsection=\tocsection

\let\oldtocsubsection=\tocsubsection

\let\oldtocsubsubsection=\tocsubsubsection

\renewcommand{\tocsection}[2]{\hspace{0em}\oldtocsection{#1}{#2}}
\renewcommand{\tocsubsection}[2]{\hspace{1em}\oldtocsubsection{#1}{#2}}
\renewcommand{\tocsubsubsection}[2]{\hspace{2em}\oldtocsubsubsection{#1}{#2}}

\begin{document}

\title[Randomized Geodesic Flow on Hyperbolic Groups]
{\large Randomized Geodesic Flow on Hyperbolic Groups}
\author{Luzie Kupffer}
\address{Luzie Kupffer, Fachbereich Mathematik und Informatik, Universit\"at M\"unster, Einsteinstra{\ss}e 62, M\"unster 48149, Germany}

\email{lkupffer@uni-muenster.de}

\author{Mahan Mj}
\address{Mahan Mj,  School
of Mathematics, Tata Institute of Fundamental Research. 1, Homi Bhabha Road, Mumbai-400005, India}

\email{mahan@math.tifr.res.in}

\author{Chiranjib Mukherjee}
\address{Chiranjib Mukherjee, Fachbereich Mathematik und Informatik, Universit\"at M\"unster, Einsteinstra{\ss}e 62, M\"unster 48149, Germany}

\email{chiranjib.mukherjee@uni-muenster.de}

\thispagestyle{empty}
\vspace{-0.5cm}

\subjclass[2010]{60K35,  20F67 (20F65, 51F99, 60J50) } 

\keywords{random walk, hyperbolic group,  harmonic measure, Bowen-Margulis-Sullivan measure}

\date{\today}

\begin{abstract}
Motivated by  Gromov's geodesic flow problem on hyperbolic groups $G$, we develop in this paper an analog using random walks.
This leads to a notion of a harmonic analog  $\Theta$ of the Bowen-Margulis-Sullivan measure on $\partial^2 G$. We provide three different but related constructions of $\Theta$:
1) by moving the base-point along a quasigeodesic ray
2) by moving the base-point along random walk trajectories
3) directly as a push-forward under the boundary map to $\partial^2 G$ of a  measure inherited from studying all bi-infinite random walk trajectories (with no restriction on base-point) on $G^\Z$. 

Of these, the third construction is the most involved and needs  new techniques.  It  relies on developing a framework where we can treat bi-infinite random walk trajectories as analogs of bi-infinite geodesics  on complete simply connected negatively curved manifolds. Geodesic flow on a hyperbolic group is typically not well-defined due to non-uniqueness of geodesics. We circumvent this problem in the random walk setup by considering \emph{all} trajectories. 

We thus get a well-defined
discrete flow that we call the \emph{randomized geodesic flow}, given by the $\Z-$shift on bi-infinite random walk trajectories. The $\Z-$shift is the random analog of the time one map of the geodesic flow. 

As an analog of ergodicity of the geodesic flow on a closed negatively curved manifold, we establish ergodicity of the $G$-action on $(\partial^2G, \Theta)$. As a consequence of our construction, we prove that the randomized geodesic flow is exponentially mixing of all orders and establish a functional CLT.

\end{abstract}

\maketitle

\tableofcontents

\section{Background and Introduction}\label{Sec Background}

{The goal of this paper is to develop a new framework of randomized geodesic flow in the setting of Gromov-hyperbolic groups by leveraging random walks rather than classical geodesic paths, addressing a long-standing challenge arising from the non-uniqueness of geodesics in this context. In Cayley graphs of hyperbolic groups, unlike negatively curved manifolds,  classical (deterministic) geodesic flow runs into the problem of multiple geodesics connecting boundary points. We circumvent this here by considering all bi-infinite random walk trajectories, yielding a fully defined discrete flow via the natural $\Z$-shift on the space of these trajectories. This $\Z$-shift serves as a probabilistic analog of the time-one map of geodesic flow.

One of the central contributions (Theorem~\ref{theorem 1}) is to provide three distinct yet equivalent constructions of the {\it harmonic Bowen-Margulis-Sullivan measure} $\Theta$ on the double boundary $\partial^2 G$: (a) following quasigeodesic rays, (b) along random walk paths, and (c) as a push-forward from a measure defined on bi-infinite random walks. 
The third construction, which is technically the most involved, is necessary in order to interpret the system as a genuine random-walk analog of geodesic flow, in that it provides the appropriate global dynamical framework rather than a purely boundary-level description. At the same time, it is not a priori clear that this framework should capture the boundary behavior. In particular, one surprising aspect is that the push-forward of the natural measure on bi-infinite random walk paths turns out to have the desired density with respect to the product of harmonic measures, or equivalently that it yields the harmonic Bowen–Margulis–Sullivan measure on $\partial^2G$. Establishing this identification requires new probabilistic and geometric tools and constitutes one of the central conceptual components of the paper. 

One first consequence (Theorem~\ref{big ergodicity intro}) of the above constructions is the ergodicity of the $G$-action on $(\partial^2G, \Theta)$, generalizing classical ergodic properties of geodesic flow on negatively curved manifolds to the random walk context on hyperbolic groups. 

Another consequence of the above constructions that distinguishes our work from previous approaches is the following quantitative aspect: We show that our random walk flow exhibits {\it strong statistical properties}. Concretely, we prove {\it exponential mixing of all orders} (Theorem~\ref{thm mixing}) and a central limit theorem 
(Theorem~\ref{thm CLT}) for the randomized geodesic flow. These results are in sharp contrast with the deterministic hyperbolic-group setting, where strong mixing properties remain largely inaccessible. Crucially, these limit theorems rely on the aforementioned structural viewpoint developed using the random-walk perspective. 

We will now review the relevant literature and outline the main results of the paper, placing our work in context.}

\subsection{Motivation and Outline of Results.} 
The study of  geodesic flows on closed negatively curved manifolds side-by-side with Brownian motion or random walks on their universal cover has a rich history. See \cite{kaimanovich,ledrappier,hamenstadt,roblin,cdst} for a small, but representative sample. Gromov-hyperbolic groups $G$ \cite{gro-hyp} provide a far-reaching generalization of fundamental groups of
closed negatively curved manifolds. However, extending the notion of geodesic flow to hyperbolic groups $G$ is problematic due to non-uniqueness of geodesics between points. In \cite[Theorem 8.3C]{gro-hyp}, Gromov gives a construction for geodesic flows on hyperbolic groups. However, the associated group action on the resulting flow space is not by isometries in this construction. A related bi-combing was constructed by Mineyev-Monod-Shalom in \cite{mms}, and
a topological analog of the geodesic flow was developed by Mineyev in \cite{mineyev}. A more measure-theoretic version of a geodesic flow on hyperbolic groups was developed by Bader-Furman in 
 \cite{BF} building on earlier work of Furman \cite{furman}: this  construction  is philosophically related to the present paper.
 Roblin's work \cite{roblin} on $CAT(-1)$ spaces deserves special mention. A key technical problem addressed in many of these papers  is the fact that geodesics between pairs of points in a hyperbolic group are only \emph{coarsely unique, but not unique in general}.

 We adopt a completely opposite point of view in this paper, by considering
 \emph{all possible paths}, thus doing away entirely with the problem of selecting a preferred path. The natural framework that we need to set up thus becomes one of bi-infinite random walks. Once we construct such a framework, geodesic flow (more precisely, the time one map associated with the geodesic flow)  gets replaced by a natural shift on the space of parametrized bi-infinite random walk paths.
 What we   achieve by this construction may be summarized by the following:
 \begin{enumerate}
 \item The construction of a flow space $\FF(\Ga)$ associated to the Cayley graph $\Ga
     =\Ga(G,S)$ of a hyperbolic group $G$ with respect to a finite generating set. The metric on $\Ga$ is the standard word metric with respect to $S$, and each trajectory maps into $\Gamma$.
     \item There is a $\Z$-shift  on $\FF(\Ga)$  preserving trajectories.
 \end{enumerate}
As pointed out above, the flow space in \cite[Section 8]{gro-hyp} does not admit an isometric $G-$action. In Mineyev's construction in \cite{mineyev}, an auxiliary metric needs to be constructed to force geodesics converging to the same point on
the boundary $\partial G$ to be asymptotic in a strong sense, thus changing the background simplicial metric.
In \cite{BF}, the trajectories are not maps to $\Ga$.
Thus, item (1) above fails in some way or another in these constructions.
We address these issues in the present paper by requiring that each flow line/trajectory in $\FF(\Ga)$ is a bi-infinite random walk path. Thus, the topological nature of the constructions in 
\cite{gro-hyp,mineyev} is replaced by a random construction, and the unique auxiliary/abstract trajectory of \cite{BF} is replaced by all trajectories that are intrinsic, i.e.\ $G^\Z$.

{
We mention some consequences of the construction in this paper that underscores the difference of the approach with the earlier ones. There is a natural
left action of $\Z$ (the shift) on $G^\Z/G$. This action turns out to be mixing once the setup is properly established. We thus deduce a  CLT easily. (See Theorems \ref{thm mixing} and  \ref{thm CLT} below, and Section \ref{sec mixing} for proofs). This is much unlike the situation in \cite{BF}, where only weaker mixing results are known. Weak mixing for the Mineyev topological flow was only recently established in \cite{cantrell}. The conceptual reason behind this difference is that the mathematical source of mixing is vastly different in the present setup. A key ingredient behind mixing of the geodesic flow in the earlier approaches is strong asymptoticity of geodesics converging to a point on the boundary, an approach going back to Hopf, Hedlund and Anosov. This property fails for a typical Cayley graph $\Ga$ of a hyperbolic group. In the present paper, mixing is guaranteed on the other hand by randomness. Morally, it goes back to the fact that Brownian motion is mixing on any compact manifold (not necessarily of negative curvature). We thus need an appropriate analog of a finite measure fundamental domain in our setup. This is furnished by $G^\Z/G$, which takes the place of the flow space $\FF(\Ga)$. Via hyperbolicity, and earlier works on random walks on hyperbolic groups, notably \cite{kaimanovich}, $G^\Z$ gets connected to the boundary of $G$.}

 This brings us to the intimately related boundary perspective.
A natural boundary measure in the context of
geodesic flow on non-elementary hyperbolic groups is the Patterson-Sullivan measure \cite{patterson,sullivan,coornaert-pjm}.
The Patterson-Sullivan measure $\mu$ \cite{coornaert-pjm} on the boundary $\partial G$ of a Gromov-hyperbolic group $G$ may be regarded as a limit of the uniform probability distributions $\mu_n$ on $n-$balls $B_n(o)$ in a Cayley graph $\Gamma$ of $G$ with respect to a finite generating set. It
is an example of a conformal density: the group $G$ acts on $\partial G$ by uniformly quasiconformal automorphisms. Let $v$ denote the Hausdorff dimension of $(\partial G, \mu)$, and $\rho$ a visual metric on $\partial G$. Then $v$ can be identified with the volume entropy of $\Ga$.
Then there exists a natural measure 
$$
\d\mu_{\mathrm{BMS}}= \frac{\d\mu \otimes \d\mu}{\rho(x,y)^{2v}}
\quad\mbox{on}\quad \partial^2 G := \big((\partial G \times \partial G) \setminus \Delta \big), \,\, \Delta=\{(\xi,\xi): \xi \in \partial G\}, 
$$
called the
\emph{Bowen-Margulis-Sullivan} measure. The key feature of $\mu_{\mathrm{BMS}}$ is that it is \emph{quasi-invariant} under the $G$-action with uniformly bounded Radon-Nikodym derivatives  -- that is, there exists $C \geq 1$ such that
$$
\frac{\d g_*\, \mu_{\mathrm{BMS}}}{\d\mu_{\mathrm{BMS}}} \in [1/C,C]\qquad\mbox{for all $g \in G$}.
$$

The {Bowen-Margulis-Sullivan} measure may be regarded as the analog in the context of hyperbolic groups
of the Bowen-Margulis measure on the space of geodesics in a closed negatively curved manifold \cite{margulis,bowen,hamenstadt89,kaimanovich}.

In this paper's random walk framework, we replace the Patterson-Sullivan measures with harmonic measures. Instead of using the Patterson-Sullivan measure $\mu$, we start with the harmonic measures $\nu_g$ on the boundary $\partial G$, also known as the hitting measure for a random walk starting at $g \in G$. While harmonic measures correspond to the Poisson boundary of $G$, Kaimanovich's fundamental work \cite[Thm. 7.4]{Kai00} demonstrates that the Gromov boundary $\partial G$ can be identified with the Poisson boundary for a broad class of random walks on $\Gamma$. Let $d_{\skrig}$ represent the Green metric on $\Gamma$, $v_{\skrig}$ the volume entropy with respect to this metric, and $\rho_{\skrig}$ the visual metric on $\partial G$ associated with $d_{\skrig}$ (see Section \ref{green metric} for more details). The harmonic measures are a conformal density of dimension $v_{\skrig}$. We define the \emph{harmonic Bowen-Margulis-Sullivan measure $\mu_{\mathrm{hBMS}}$} as follows:
\[
d\mu_{\mathrm{hBMS}} = \frac{d\nu_o \otimes d\nu_o}{\rho_{\skrig}(x,y)^{2v_{\skrig}}}.
\]

In this paper, we explore three different ways of constructing $\mu_{\mathrm{hBMS}}$ directly in terms of \emph{bi-infinite random walk trajectories}. In Section~\ref{sec bi-inf RW} we set up the main player in the game: bi-infinite random walk trajectories obtained by starting with a random walk trajectory $(X_n)_{n\in\N}$ with finitely supported and symmetric step distribution, shifting the origin to $X_m$, i.e.\ $(X_m^{-1}X_n)_{n\in\N}$, and letting $m$ tend to infinity. We then construct three different measure-equivalent versions of $\mu_{\mathrm{hBMS}}$ on $\partial^2 G$ in the following ways:
\begin{enumerate}
   \item[(1)] We construct $\Theta_1$ as a limit of occupation measures of bi-infinite random walk trajectories, after fixing a geodesic ray in each direction $\xi \in \partial G$ -- see Theorem \ref{theorem 1}, Part (i) and Section~\ref{sec Theta bd}. 
   \item[(2)] We construct a more intrinsic \emph{disintegrated version} $\Theta_2$ of the above measure: instead of choosing a geodesic ray in each direction $\xi \in \partial G$, we capture the behavior obtained while moving along \emph{all} random walk paths starting at the identity -- see Theorem \ref{theorem 1}, Part (ii) and Section \ref{sec Theta disintegration}. 
\item [(3)] We first construct a natural measure $\Q$ on the space  $G^\Z$ of all bi-infinite random walk trajectories (with no restriction on base-point). The measure $\Q$ is 
\begin{itemize}
  \item invariant under the natural $\Z-$shift on trajectories, and admits a quotient $\Z\backslash G^\Z$ equipped with a natural measure $\widehat \Q$ inherited from $\Q$. Note that
  $\Z$ acts on the left on $G^\Z$.
  \item invariant under the natural $G-$action that moves the base-point, and admits a quotient $G^\Z/G$ equipped with a natural measure $\bar\P_{o}$ inherited from $\Q$.
   Note that $G$ acts on the right on $G^\Z$.
\end{itemize} 
We then push forward the measure $\widehat \Q$ under the \emph{boundary map} to obtain a measure $\Theta_3$ that sends each bi-infinite random walk trajectory to the pair of points in $\partial^2 G$ that it is forward and backward asymptotic to -- see Theorem \ref{theorem 1}, Part (3) and Section \ref{sec Theta push-forward}. This construction of $\Theta_3$ allows us to show that for this measure, the $G$-action on $\partial^2 G$ is ergodic -- see Theorem \ref{theorem 1}, Theorem \ref{big ergodicity intro} and Section \ref{sec ergodicity}. 
 \end{enumerate}

These constructions, and in particular (3) above, lie at the heart of the present paper. Indeed, $\Theta_3$ is what justifies the terminology harmonic Bowen-Margulis-Sullivan measure as it reconciles
the internal ($G$) and  boundary ($\partial G$) points of view:
\begin{enumerate}    
\item Internally, it provides an interpretation of the measure's construction as a random walk analogue of the geodesic flow (see below). 

\item On the boundary $\partial^2 G$, it leads to the desired density with respect to $\nu_o\otimes\nu_o$. Thus $\Theta_3$ and the harmonic Bowen-Margulis measure  $\mu_{\mathrm{hBMS}}$ are really the same measure on $\partial^2 G$ (up to a uniformly bounded multiplicative constant).
\end{enumerate}
Note that $\mu_{\mathrm{hBMS}}$ can be defined directly on
$\partial^2 G$ as soon as one has a conformal density given by the harmonic measure $\nu_0$
on $\partial G$. It is  as a consequence of the above two features of $\Theta_3$, that $\mu_{\mathrm{hBMS}}$ can indeed be treated as a genuine discrete flow along bi-infinite random walk paths. It is worth noting that the first two methods (1) and (2) also yield constructions of $\mu_{\mathrm{hBMS}}$, despite the fundamentally different approach adopted there. Showing that these differing constructions achieve the same result is the main thrust of Theorem~\ref{theorem 1} below.
These approaches provide  different perspectives on the interpretation of  $\mu_{\mathrm{hBMS}}$, enriching its theoretical framework.
Demonstrating that these approaches all lead to  versions of $\mu_{\mathrm{hBMS}}$  is nontrivial and requires new techniques that we outline in Section \ref{sec proof ideas}.

We now elaborate on  the point (i) above. Indeed, the description in (3) is the random trajectories analog of the geodesic flow on the unit tangent bundle of the universal cover of a closed negatively curved manifold viewed from the boundary. We should remind the reader here that geodesic flow on a hyperbolic group is typically not well-defined due to non-uniqueness of geodesics. In the random walk setup however, we do get a well-defined 
\emph{discrete flow} given by the $\Z-$shift. This may be regarded as the random walks analog of the cyclic group generated by the time one map of the geodesic flow. Further, since we consider \emph{all}
random walk trajectories, we circumvent the problem 
of non-uniqueness of geodesics and we have a well-defined 
{\emph{random walk flow}}. Thus, we would like to emphasize on the following points: 

\begin{enumerate}
  \item[(a)] The geodesic flow on the (unit tangent bundle to the) universal cover $\widetilde{N}$ of a closed negatively curved manifold
  $N$ is well-defined 
  from the boundary perspective due to uniqueness of geodesics connecting $\xi,\eta \in \partial \widetilde{N}$.
  \item[(b)] The {\emph{random walk flow}} is well-defined 
  from the boundary perspective for an opposite reason: all random walk trajectories are considered.
\end{enumerate}
Furthermore, while any pair of points 
$\xi,\eta \in \partial \widetilde{N}$ determines a unique bi-infinite geodesic, $\xi,\eta \in \partial G$ determines a whole family of random walk trajectories.

Finally, we note that the bi-infinite random walks studied in the present work differ from the classical view of bi-infinite random walks. Indeed, as will be explained in Section \ref{bi-inf RW}, our bi-infinite random walks are constructed by iteratively shifting back a random walk $(X_n)_{n\in\N}$. Hence, the distribution $\bar\P_o$ of our bi-infinite random walk (see \eqref{def X} below and Section \ref{bi-inf RW}) is equivalent to $\P_o\otimes\P_o$, the product measure of two independent random walks started at $o$. The classical view on the other hand regards bi-infinite random walks as loops through the point at infinity in the one-point compactification of the space. We refer to Appendix \ref{appendix} for a more detailed discussion on this difference, both conceptual and technical, between these two notions of bi-infinite random walk paths and the associated distributions and to Example \ref{example RI} for an explicit computation illustrating this difference in the case of a free group.

\section{Main  Results.}\label{sec main results}
We now fix a Gromov-hyperbolic group $G$ with a symmetric finite generating set $S$. Let $\partial G$ denote its Gromov (geometric) boundary, while $(\cdot,\cdot)^{\skrig}_{\cdot}$ being the Gromov product with respect to the Green metric $d_\skrig$ (see Section \ref{sec background} for definitions and background). In this paper,  the set of pairs of {\it distinct boundary points} will be written as 
\begin{equation}\label{def partial2}
\partial^2 G:= \big(\partial G \times \partial G\big)\setminus \Delta, \quad 
\text{with}\quad \Delta:=\big\{(\xi,\xi)\colon \xi \in \partial G\big\}.
\end{equation}
 
To state our results we will need the following notation:

\begin{itemize}
\item Let $\mu$ be a probability measure on $G$ whose support is the finite symmetric generating set $S$. For any $g\in G$, $\P_g$ will denote the distribution of a random walk with step distribution $\mu$ and started at $g\in G$, while $\nu_g$ stands for its hitting distribution on the boundary $\partial G$ (see \eqref{def nug} in Theorem \ref{divergence RW}).
\vspace{1mm}

\item $\bar\P_{g}$ denotes the distribution of a bi-infinite random walk $(X_z)_{z\in\Z}$ defined as
\begin{equation}\label{def X}
X_{z} = \begin{cases} Z_z&\text{ if } z\geq 0,\\ \bar Z_{-z}&\text{ if } z<0. \end{cases}
\end{equation}
where $(\bar Z_n)_{n\in \N}$ and $(Z_n)_{n\in \N}$ are two independent random walks started at $g$ and with step distribution $\mu$. $G^\Z$ will denote the space of all bi-infinite random walk trajectories, see Section \ref{bi-inf RW} for details.

\vspace{1mm} 

\item With $(\bar\P_g)_{g\in G}$ defined on $G^\Z$ above, we define
\begin{equation}\label{def QZ int}
\Q(B):= \sum_{g \in G}\bar\P_g(B)\qquad\mbox{and}\qquad \widehat \Q(B):= \Q(B \cap \skrid) \qquad \forall \, B\subset G^\Z
\end{equation}
where
\begin{equation}\label{def Dcal}
\begin{aligned}
    \mathcal D:= \bigcup_{g \in G} \bigg \{(x_z)_{z} : x_0 = g ,\,  &d_{\skrig}(x_{z},o)> d_{\skrig}(x_0,o) \text{ for }z<0,\\
    &\text{and}\,\, d_{\skrig}(x_{z},o)\geq d_{\skrig}(x_0,o) \text{ for }z>0, \bigg\}.
    \end{aligned}
\end{equation}
We refer to Section \ref{def Q} and Section \ref{sec hatQ Theta3} for details regarding $\Q$ and $\widehat\Q$, respectively.  
\end{itemize}
 Here, and throughout the article, for any two functions $f(\cdot)$ and $g(\cdot)$, we write $f(\cdot)\asymp g(\cdot)$ to mean that there exists a constant $C\in (0,\infty)$ such that $\frac{f(\cdot)}{g(\cdot)}\in [1/C, C]$. 
 
 A measure $\mu$ on a measure space $X$ is called {\bf quasi-invariant} under an action $G$ on $X$, if 
 $g_\star\mu$ is mutually absolutely continuous with respect to $\mu$ for all $g\in G$.
And, two measures $\mu_1,\mu_2$ on the same measure space $(\Omega,\Fcal)$ are called {\bf measure-equivalent}, if they are mutually absolutely continuous.

\subsection{Random Walk Constructions of $\Theta$ and Their Properties} 

We are now ready to state the first main result of our paper:

\begin{theorem}\label{theorem 1}
There exists a measure $\Theta$ on $\partial^2 G$ satisfying 
\begin{equation}\label{density-pair}
\frac{\d\Theta}{\d (\nu_o\otimes\nu_o)} (\xi,\eta)\asymp \e^{2(\xi,\eta)_o^{\skrig}} \qquad\forall (\xi,\eta)\in \partial^2 G, 
\end{equation}
which can be constructed explicitly in the following three measure-equivalent ways:
\begin{enumerate}
\item As the measure $\Theta_1$ given by
\begin{equation}\label{bd meas}
\Theta_1(A) = \int_{A} \lim_{n\to\infty} \sum_{z=-n}^n\bigg[\frac{1}{2n+1} \, \frac{\d(\nu_{p_z}\otimes \nu_{p_z})}{\d(\nu_{o}\otimes \nu_{o})}(\xi,\eta)\bigg]\;\d(\nu_{o}\otimes \nu_{o})(\xi,\eta)
\end{equation}
for $A\subset \partial^2 G$, where $(p_z)_{z\in\Z}=(p_z(\xi,\eta))_{z\in\Z}$ is a fixed quasi-geodesic between $\xi$ and $\eta$. 
\vspace{2mm}
\item As a disintegrated version $\Theta_2$ of $\Theta_1$, defined by
\begin{equation}\label{Theta disintegration}
\Theta_2(A) = \int_{A} \lim_{n\to\infty} \sum_{z=-n}^n \bigg[\frac{1}{2n+1} \,\frac{\d(\nu_{x_z}\otimes \nu_{x_z})}{\d(\nu_{o}\otimes \nu_{o})}(\xi,\eta)\bigg]\;\d\bar\P_o\big((x_n)_{n\in\N}\big),
\end{equation}
for all $A\subset\partial^2 G$, where $\xi=\bd(x_n)_{n\in\N}$ and $\eta=\bd(x_{-n})_{n\in\N}$.
\vspace{2mm}
\item \label{push-forward} As the push forward $\Theta_3(\cdot)= \widehat\Q(\bd^{-1}(\cdot))$ of the measure $\widehat\Q$ (defined in \eqref{def QZ int}) under the boundary map $\bd: G^{\Z}\to \partial^2 G$. The measure $\Theta_3$ is invariant under the action of $G$ on $\partial^2G$. 
\end{enumerate}
Finally, the measure $\Theta$ on $\partial^2 G$ is 
\begin{itemize}
    \item quasi-invariant under the $G$-action with uniformly bounded Radon-Nikodym derivatives $\frac{\d g_{\ast}\Theta}{d\Theta}$, and
    \item ergodic under the action of $G$ on $\partial^2 G$. 
\end{itemize} 
 \end{theorem}

  As a result of the different constructions, the above three measure-equivalent constructions of $\Theta$ all provide different information about its structure. The measures $\Theta_1$ and $\Theta_2$, constructed in Section \ref{sec Theta bd} and Section \ref{sec Theta disintegration} respectively, reflect the dynamics of the hitting measures when moved along bi-infinite quasi-geodesics (resp. bi-infinite random walk paths). 
  
  In contrast, the measure $\Theta_3$, constructed in Section \ref{sec Theta push-forward}, is the random walk equivalent of the Bowen-Margulis-Sullivan measure $\mu_{\mathrm{BMS}}$ (see also Remark \ref{comparison versions Theta}). As mentioned in Section \ref{Sec Background} and we will see below, it is here that the connection to the {\it random walk flow} turns up in a similar way to how $\mu_{\mathrm{BMS}}$ is connected to the geodesic flow in the deterministic setting. This connection then enables us to show the ergodicity of $\Theta_3$ for the $G$-action, and therefore that of $\Theta$. Indeed, on the way to showing this ergodicity, we also show the ergodicity, exponential mixing of all orders and a central limit theorem of the random walk flow. We turn to a precise descriptions of these results now.

\subsection{Randomized Geodesic Flow, Exponential Mixing and Functional CLT}

{

\begin{defn}\label{def RWFlow}
Let $(\tau_z)_{z \in \Z}$ denote the action of $\Z$ on $G^{\Z}_o= G^\Z/G$ defined by
\[\tau_z(x_n)_{n\in \Z}:= x_z^{-1}(x_{n-z})_{n\in \Z}.\]
We call this action the {\bf randomized geodesic flow}, or {\bf the random walk flow}.
\end{defn}

Here is our next main result. 

\begin{theorem}[Double Ergodicity]\label{big ergodicity intro}
The $G$-action on $\Z \backslash G^{\Z}$, and the $\Z$-action on $G^{\Z}/G$ are ergodic for $\bar\P_{o}$ and $\widehat \Q$.
\end{theorem}

\noindent The above theorem is proved in Section \ref{sec proof ergodicity}; see Theorem~\ref{big ergodicity} there. Theorem~\ref{big ergodicity intro} is an analog of the ergodicity of geodesic flow on closed negatively curved manifolds $N$:
\begin{enumerate}
  \item $G^{\Z}$ is the analog of the unit tangent bundle $T_1\til N$ of $\til N$,
  \item $\Z \backslash G^{\Z}$ is the analog of $\partial^2 \til N$,
  \item $G^{\Z}/G$ is the analog of the unit tangent bundle $T_1 N$ of $N$.
\end{enumerate}
A standard duality principle in the theory of geodesic flows on negatively curved manifolds states that the ergodicity of the $G-$action on $\partial^2 \til N$ is equivalent to the 
ergodicity of the $\R-$action on $T_1 N$.
In the context of hyperbolic groups, the analogous construction is due to Bader-Furman 
\cite{BF}. We refer the reader to the discussion preceding Theorem~\ref{big ergodicity} for the relevance and differences with \cite{BF}. 

We now turn to the results concerning quantitative exponential mixing estimates of our random walk flow. This requires setting up the following definitions (see e.g. \cite{expmixing}). 
\begin{defn}\label{def Mixing}
Let $G$ be a group acting on a probability space $(X,\mu)$ and let $r\geq 2$ be an integer. 
\begin{enumerate}
\item [(1)] We say that the action of $G$ on $(X,\mu)$ is {\bf mixing of order $r\geq 2 $}, if for all $f_1,...,f_r \in L^{\infty}(X,\mu)$, and for all $g_1,g_2,\dots, g_r \in G$ we have 
$$
\bigg| \int_{X}\prod_{i=1}^r ( f_i\circ g_i)\d \mu- \prod_{i=1}^r  \int_{X}f_i\d \mu\bigg| \to 0  \qquad\mbox{as $\min_{\ell\neq k} \, d(g_\ell,g_k)\to \infty$}.
$$

\item [(2)] We fix an $G$-invariant sub-algebra $\mathcal A$ of $L^{\infty}(X,\mu)$, and a family $N = (N_s)_{s\in \N}$ of semi-norms on $\mathcal A$. Then we say that the action of the group $G$ on  $(X,\mu)$ is {\bf exponentially mixing of order $r$} with respect to $d$ and $(\mathcal A,N)$, if there exist $\delta_r > 0$ and an integer $s_r\in \N$ such that for all $s > s_r$, $f_1,...,f_r \in  \mathcal A$ and $g_1,\dots,g_r\in G$, 
$$
 \bigg| \int_{X}\prod_{i=1}^r ( f_i\circ g_i)\d \mu- \prod_{i=1}^r  \int_{X}f_i\d \mu\bigg|  \leq C \e^{-\delta_r d_r(g_1,..,g_r)} \prod_{i=1}^r N_s(f_i) 
$$
for some finite constant $C=C(s,r)$. In this case, $\delta_r>0$ is called {\bf the rate of mixing}. 
\end{enumerate}
\end{defn}

\begin{defn}\label{def Holder}
A function $f: X \to \R$ on an infinite product space $X = \prod_{n\in \N} X_n$ of metric spaces $(X_n,d_n)$ is H\"older continuous with respect to the standard metric on $X$ given by
\[d(x,y) = \sum_{n\in \N} 2^{-n} \min\{1, d_n(x_n,y_n)\},\]
if there exists $\alpha >0$ such that
\[\abs{f(x)-f(y)}\leq d(x,y)^{\alpha} \quad \text{for all } x,y\in X. 
\]
\end{defn}

Here is our next main result showing  exponential mixing and a functional central limit theorem (CLT) for our random walk flow from Definition \ref{def RWFlow}. 

\begin{theorem}[Exponential Mixing of all orders]\label{thm mixing}
The $\Z$-action on $(G_o^{\Z},\bar\P_o)$, i.e. the randomized geodesic flow from Definition \ref{def RWFlow}, is mixing of all orders. Moreover, for any $\alpha>0$, this action is also exponentially mixing of all orders with respect to the class $\mathcal A_\alpha$ of H\"older continuous functions of exponent $\alpha>0$. 
\end{theorem}

\begin{theorem}[Functional CLT]\label{thm CLT}
For any H\"older continuous function $f:G^{\Z}_o\to \R$, there exists $\sigma^2_f >0$ such that 
\[\frac{\sum_{z=1}^{m} f\big(\tau_z(X_n)_{n\in \Z}\big)- m \E^{\bar\P_o}[f\big((X_n)_{n\in\Z})\big]}{\sqrt{m}} \to \Ncal(0, \sigma^2_f) \qquad\mbox{as $m\to\infty$,}\]
in distribution under $\bar\P_0$. Here $\Ncal(0,\sigma^2_f)$ denotes the centered Gaussian law with variance $\sigma^2_f$. 
\end{theorem}

Theorem \ref{thm mixing} and Theorem \ref{thm CLT} are proved in Section \ref{sec mixing}.

}

\subsection{Context}\label{sec-lit} 
As mentioned before, the aim of this paper is to develop  a harmonic analog of the Bowen-Margulis measure on $\partial^2 G$ from an \emph{internal} (as opposed to boundary) point of view, i.e\ from the point of view of the collection of bi-infinite random walk trajectories on a Cayley graph $\Ga$ of $G$. Since the Bowen-Margulis measure has been
a long-studied object in  ergodic theory, we provide an eclectic survey of the literature under a few heads below to provide context for the work in this paper.\\

\noindent {\bf Geodesic flow on hyperbolic groups:} We have already said that the key motivation for this paper was to develop a bi-infinite random-walks analog of geodesic flow on hyperbolic groups in the spirit of Gromov's topological construction \cite[Section 8]{gro-hyp}. The papers \cite{mineyev,mms,ct24} amongst others have  pursued and considerably developed
the topological construction suggested by Gromov. Our emphasis in this paper is probabilistic and involves, by contrast, random trajectories.\\

\noindent {\bf The Hopf-Tsuji-Sullivan theorem:} A historical context in which the Bowen-Margulis measure frequently comes up is the Hopf-Tsuji-Sullivan  dichotomy that asserts that the (deterministic) geodesic flow is either conservative and ergodic or dissipative. We refer the reader to 
\cite{hamenstadt89, Kai94,cp94,BM96,furman,roblin,ricks,link,BF,cdst,coulon} for some of the relevant work. Of these papers, \cite{hamenstadt89,Kai94,cp94,BM96,roblin,ricks,link} deal with
Hadamard manifolds, or more generally CAT(0) spaces. In particular, a standing assumption in each of these  is the uniqueness of geodesic segments connecting any pair of points. The only Cayley graph per se that satisfies the CAT(0) condition is that of a free group with standard generating set. The present paper, on the other hand, deals exclusively with Cayley graphs of non-elementary hyperbolic groups and not with CAT(0) spaces.

The papers \cite{BF,cdst} do away with the CAT(0) assumption and prove the Hopf-Tsuji-Sullivan theorem for hyperbolic spaces and \cite{coulon} proves it for actions of acylindrically hyperbolic groups. Kaimanovich's papers \cite{kaimanovich,Kai94} go beyond
the classical Bowen-Margulis measure associated with the geodesic flow:
for certain Markov operators $P$, 
an invariant measure of the  geodesic 
flow called the
invariant harmonic measure associated with $P$ is defined in \cite{Kai94}, and the 
Hopf-Tsuji-Sullivan  dichotomy for these is established. However, a crucial assumption in \cite{Kai94} is uniqueness of geodesics joining a pair of boundary points (denoted Property $U$ in that paper). Similarly, in \cite{kaimanovich} it is assumed that the underlying manifold has negative curvature and hence 
there is a unique geodesic joining a pair of boundary points.
This property breaks down fairly drastically in the context of Cayley graphs of most hyperbolic groups, where typically there are infinitely many infinite geodesic bi-gons between pairs of points on the boundary.

We reiterate that  all these papers, unlike the present one, deal with the (deterministic) geodesic 
flow. \\

\noindent {\bf Action on the double boundary:} We now deal with precursors to the  last statement of Theorem~\ref{theorem 1}, i.e.\ double ergodicity. The papers of Bader and Furman 
\cite{furman,BF} cited above proves that the action of $G$ on $\partial^2 G$ equipped with the Bowen-Margulis measure is ergodic for $G$ hyperbolic.
Perhaps more relevant to the present paper is \cite{Kai03}.
In \cite{Kai03}, Kaimanovich proves an extremely general double ergodicity result
on the square of the Poisson boundary  of a locally compact
second countable group. The papers \cite{furman,BF,Kai03} adopt a  point of view intrinsic to the boundary, the former for the (deterministic) geodesic flow and the latter
for random walks.  In the discussion following \cite[Proposition 5.4]{BHM11} it is pointed out that once the harmonic measure has been constructed a similar double ergodicity result may be obtained by weighting the product measure in a way similar to \eqref{density-pair}. These papers thus deal directly with $\partial^2 G$ without
considering actual trajectories (geodesics or bi-infinite random walks)
in the Cayley graph $\Ga$.
The constructions of $\Theta$ in Theorem~\ref{theorem 1} follow a 
complementary  approach by going to $\partial^2 G$ from inside the space. This is the main difference in the approach
in the present paper.

Our results, too, have a subtly different flavor. We briefly delineate the main technical difference between 
the main result of \cite{Kai03} and Theorem~\ref{theorem 1}.
Consider a symmetric driving measure $\mu$. Double ergodicity is established in  \cite{Kai03} with respect to the product of the $G$-invariant measures $\nu_m = \sum_{g\in G}\nu_g$. In the language of the present work, this aligns with looking at the ergodicity of the push-forward of the measure $\Q$ (defined in \eqref{def QZ int}) under the boundary map to $\partial^2 G$. In Theorem~\ref{theorem 1} we consider instead the push-forward of $\widehat\Q$ (also defined in \eqref{def QZ int}). It is the push-forward of $\widehat\Q$ that is used to construct $\Theta_3$. The ergodicity of the measure $\nu_m \otimes \nu_m$ under the $G$-action on  $\partial^2 G$ follows from the ergodicity of $\Theta_3$. The latter is shown in Section \ref{sec ergodicity}. Note that the $G$- and $Z$-action on $G^{\Z}$ commute (see \eqref{actions}). Now look at the pre-image of a $G$-invariant subset of $\partial^2G$ in $G^{\Z}$.  This gives a $G$-invariant set under the induced action on the fundamental domain. By ergodicity of $\widehat\Q$ (Theorem \ref{big ergodicity intro}), this set is either of measure $0$ or of full measure. Hence the same holds  for $\Q= \sum_{z\in \Z}\tau_{z\ast}\widehat\Q$ and for $\nu_m\otimes\nu_m$. (Recall that the latter is the push forward of $\Q$ under the boundary map.)

Note however that $\nu_m$ is infinite on any set $B \subset \partial G$ with $\nu_o(B)>0$ and $0$ on all other sets. The product $\nu_m \otimes \nu_m$ is thus in a different measure class from $\Theta_3$. We can of course restrict $\Q$ to a fundamental domain of the $\Z$-action 
and push forward the restricted measure to the boundary (as done in Section \ref{sec Theta push-forward}). This does give a measure that is ergodic thanks to Kaimanovich \cite{Kai03}.
However, this does not provide information about the existence of its density with respect to $\nu_o\otimes \nu_o$. The relation to $\Theta_1$ and $\Theta_2$
is also unavailable from this approach.
To obtain these, we need the ideas developed in  Sections \ref{sec-pfofbehavior Theta 3} and \ref{bounds on Green combinatorics}. Proving the existence of $\Theta_3$, establishing the correct density with respect to $\nu_o\otimes\nu_o$, and finding its connection to a randomized version of the geodesic flow, form the conceptual core  of the present work.

\noindent \emph{Harmonic functions on groups:}
In the related literature on harmonic functions on groups, the action on the double boundary makes a somewhat tangentially related appearance in the context of bi-harmonic functions \cite{raugi,kaimanovich-cras}. \\

\subsection{Key Technical Input in the Proofs}\label{sec proof ideas}
For the convenience of the reader, we now summarize the main technical constituents of the proofs. 

To study a flow along bi-infinite random walk paths, we need to define what these paths are. 
In Section \ref{sec bi-inf RW} we will construct these bi-infinite random walk paths, and deduce some important properties of these. Concretely, in Section \ref{bi-inf RW} the bi-infinite random walk will be constructed by taking a random walk $(X_n)_{n\in\N}$ started at $o$ and successively shifting it back by $X_n$ (the position at time $n$) to obtain a bi-infinite path. The reason to consider this notion of bi-infinite random walk paths is that it preserves the distribution of geometrical properties of the original random walk. These properties will be crucially used subsequently.

In particular, in Section \ref{deviation from geodesic} we study the long term behavior of these bi-infinite random walks $(X_z)_{z\in\Z}$ and prove estimates on its distance to the geodesics. Studying the deviation of a (standard/semi-infinite) random walk trajectory $(Y_n)_{n\in\N}$ from geodesics goes back to Kaimanovich, who showed that the distance of $Y_n$ to $[o,Y_{\infty})$ is almost surely $o(n)$ in \cite[Thm. 7.2]{Kai00}. Later, Blachère, Haïssinsky and Mathieu refined this result and showed that the distance is of order $\log(n)$ in \cite[Cor. 3.9]{BHM11}. 

In this regard,  in Theorem \ref{deviation bds} we will first study the case of standard/semi-infinite random walk paths $(Y_n)_{n\geq 0}$ and show deviation bounds for fixed points $Y_n$. The upper bound on the probability in Theorem \ref{deviation bds}, Part (1), was already shown in \cite[Lem. 3.8]{BHM11}, while the corresponding lower bound is new. Note that due to the very minimal assumptions on the driving measure $\mu$ in \cite{BHM11}, they need to add the assumption that $d_{\skrig}$ is hyperbolic. In the case that $\mu$ is symmetric and has finite support, the latter is always true.

These estimates are then used to study the joint behavior of two independent random walks -- see Lemma \ref{gp independent}. We can use this result to show that the two hitting points $X_{-\infty}$ and $X_{\infty}$ are almost surely distinct (Corollary \ref{non-sing}). Consequently, the behavior of our bi-infinite random walk $(X_z)_{z\in\Z}$ can be studied and  the asymptotic behavior of the distance of $X_n$ to a geodesic $(X_{-\infty}, X_{\infty})$ can be made precise -- see Corollary \ref{almost geodesic}. This  also justifies our point of view of considering bi-infinite random walk trajectories as generalized geodesics.

In Theorem \ref{disintegrated deviation bds} we show a more refined version of the above bounds: we prove the validity of the bounds when conditioning on hitting a certain point on the boundary. Here it is imperative that we have both the upper and lower bound obtained in Theorem \ref{deviation bds}. These refined estimates are later used for proving the statements about the measure $\Theta_2$ appearing in Theorem \ref{theorem 1}.

In Section \ref{sec Theta constructions}, we turn to the first two constructions of the measures mentioned in Theorem \ref{theorem 1}. In Section \ref{sec Theta bd} we define and study the measure $\Theta_1$ on
$\partial^2 G$ as defined in \eqref{bd meas} and show that the sum there converges to a measure which is 
\begin{enumerate}
    \item quasi-invariant under the $G$-action on $\partial^2 G$ induced from the $G$-action on $\partial G$ (cf. Lemma \ref{density Theta1}) and
\item has the desired density with respect to $\nu_o\otimes \nu_o$ (cf. Lemma \ref{qi Theta1})
\end{enumerate}
This measure depends however on a choice of quasi-geodesic rays in 
(a Cayley graph of) $G$. We would like to remove this choice. Hence, in Section \ref{sec Theta disintegration} we move on to the construction of $\Theta_2$ defined in \eqref{Theta disintegration}. The measure $\Theta_2$ is defined by considering bi-infinite random walk paths and does not depend on a specific choice of quasi-geodesic rays. In order to show that $\Theta_2$ has the desired density with respect to $(\nu_o\otimes\nu_o)$ (in particular that it satisfies the lower bound), we need to apply the aforementioned refined estimates from Theorem \ref{disintegrated deviation bds}. We refer to Lemma \ref{density Theta2} and Lemma \ref{qi Theta2} for details.

Section \ref{sec Theta push-forward} is devoted to constructing and deducing properties of the measure $\Theta_3(\cdot)=\widehat\Q(\bd^{-1}(\cdot))$ with $\widehat\Q$ from \eqref{def QZ int}. In Section \ref{def Q} we show that the measure $\Q$ is invariant both under the action of $\Z$ and $G$ on $G^\Z$. This invariance will then carry over to the same property for $\Theta_3$ on $\partial^2 G$ under the action of $G$.

The main technical difficulty then involves showing that the measure $\Theta_3$ has the correct density, as stated in Theorem \ref{density Theta}. Due to the poorly connected nature of free groups, the proof is much easier in the case that $G$ is a free group. In Section \ref{example density tree} we will outline the guiding philosophy for the proof of Theorem \ref{density Theta} for this case. Despite its simplicity, this example serves as an important illustration in two key aspects. First, it demonstrates why the measure $\Theta_3$ should possess the desired density relative to $\nu_o \otimes \nu_o$. Second, it  provides insight as to why the method below for bounding the density should yield the correct result.

To make the second point point clearer, we note that, in the tree case, the only contribution to the density $\frac{d\Theta_3}{d(\nu_o \otimes \nu_o)}(\xi, \eta)$ comes from trajectories of the random walk with origins that lie along the path where the geodesics $[o,\xi)$ and $[o,\eta)$ run parallel. The exact value of the density then comes from the conformal nature of the hitting measures $(\nu_g)_{g\in G}$.

\noindent Section \ref{bounds on Green combinatorics} and Section \ref{sec-pfofbehavior Theta 3} constitute the proof of Theorem \ref{density Theta} for general hyperbolic groups. In this setting, we can no longer assume that  removing a ball of a certain size  disconnects points. Due to the geometric nature of hyperbolic groups, removing a ball significantly increases the distance between points on opposite sides of the ball. So a random walk started on one side of a ball is much less likely to hit a point on the opposite side of the ball when forced to walk around the ball. This heuristic will be used in a crucial way in the sequel. We will elaborate this point now.

This fact is in particular important in showing the upper bound for the density in Proposition \ref{upper bound density}. The main challenge in this proof is that in $\widehat \Q$ we are potentially summing over something which is positive for all group elements. So, in order to show the upper bound for the density $\frac{\d\Theta_3}{\d \nu_o\otimes\nu_o}(\xi,\eta)$ in the case of a general hyperbolic group, we need to show that the contributions to $\widehat \Q$ of starting points away from the region where $[o,\xi)$ and $[o,\eta)$ lie very close to each other, decay sufficiently fast.
A key estimate in showing this deals with how much forcing the random walk to walk around a ball reduces the probability.
Estimates similar in spirit have been considered for the weighted Green function by Gou\"ezel and Lalley for surface groups in \cite[Lem. 4.4]{GL13} and by Gou\"ezel for general hyperbolic groups in \cite[Lem. 2.6]{LLT}. They showed that the probability of walking from $x$ to $y$ while staying outside a ball of radius $n$ centered on $[x,y]$ decays as $\exp(c\e^{-\varepsilon n})$ for some $c,\varepsilon>0$. It is however by no means obvious how to apply this estimate to show that the probabilities in $\widehat \Q$ decay fast enough to end up with the claim.
Indeed, we need to refine their statements in two main ways in Lemma \ref{decay greens function}.

\begin{itemize}
\item[(1)] We generalize to a ball not necessarily centered on $[x,y]$. In other words, our study will extend the previous findings by considering cases where the center $o$ of the ball  that gets removed does not lie on the geodesic between $x$ and $y$. 

\medskip 

\item[(2)] Further, we are also able get a bound dependent on the exact distance of $y$ to the center $o$ of the ball, not just the fact that is more than $n$ away from $o$. This part will require an application of the Ancona inequality. 
\end{itemize}

These arguments, developed in Lemma \ref{decay greens function}, Lemma \ref{lemma def barriers} and Lemma \ref{number hitting points}, effectively provide us with estimates for the hitting distribution on a small neighborhood of a geodesic $[o,\xi)$ of the random walk started at $x$ on the boundary of $B_n(o)$ when forced to walk around this ball. This in turn enables us to compare the probability of hitting $x$ when restricting the random walk to staying outside $B_n(o)$ with the unrestricted case.

We note that the bounds mentioned above depend on the angle the random walk has to walk around the ball. When this angle gets very small, we may lose the aforementioned double exponential decay. So even with the  refined estimates developed in the above lemmas in Section \ref{bounds on Green combinatorics},  we then still have to be very careful in applying them in the Proof of Proposition \ref{upper bound density}.  More concretely, we will need to 

\begin{itemize}
\item[(a)] split the group apart into several domains judiciously, and 

\medskip 

\item[(b)] carefully estimate how many points can lie in each of these domains.

\end{itemize}

Throughout these estimates it is crucial that, via the Bonk-Schramm embedding, we may embed the Cayley graph of the group into $\mathbb{H}^d$. This allows us to invoke more explicit geometric knowledge about $\mathbb{H}^d$ than the abstract geometric knowledge we have about the group. We refer to the proof of Proposition \ref{upper bound density} in Section \ref{sec-pfofbehavior Theta 3} for details. The lower bound of Theorem \ref{density Theta}, shown in Proposition \ref{lower bound density}, is easier. The proof is similar to the case of the free group in Section \ref{example density tree}. It proceeds by summing over origins that roughly lie on the shortest geodesic between $o$ and the geodesic $(\xi,\eta)$. This will complete the proof of Theorem \ref{density Theta}.

Section \ref{sec ergodicity} constitutes the proof of ergodicity of $\Theta$ as well as the exponential mixing and the CLT of our random walk flow. For this purpose, we will first prove Theorem \ref{big ergodicity intro}, namely that the measures $\bar\P_{o}$ and $\widehat \Q$ are ergodic for the flow along the random walk paths and the $G$-shift on the paths respectively. The ergodicity of $\Theta_3$, and hence $\Theta$, will then follow.

For proving Theorem \ref{big ergodicity intro} our approach builds on a philosophy related to Bader and Furman  \cite{furman,BF}. However, while they regard $\partial^2 G\cross \R$ as the ``space of all geodesics with all parametrizations", we consider $G^{\Z}$, the space of all random walk trajectories with all parametrizations. This way we avoid having to choose geodesics for each pair of boundary points. Also, since we are in a discrete setup, it is
slightly misleading to think of geodesics as parametrized by $\R$. 

This difference also means that in our case the fundamental domain for the $G$-action as well as the induced $\Z$-action is a more concrete set than in the abstract almost sure fundamental domain they construct. This, in turn, makes the proof of ergodicity somewhat simpler in our case. 

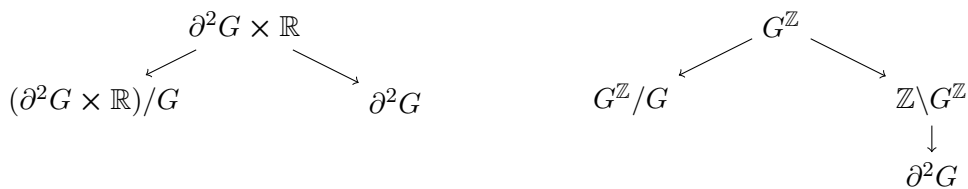
\begin{figure}[h]
  \centering
  \begin{subfigure}[T]{0.45\textwidth}
      \begin{tikzpicture}
 \node (A) at (0,1) {$\partial^2 G \cross \R$};
 \node (B) at (-2,0) {$(\partial^2 G \cross \R)/G$};
 \node (C) at (2,0) {$\partial^{2} G$};
 
 \draw[->] (A) -- (B);
 \draw[->] (A) -- (C);
\end{tikzpicture}
  \end{subfigure}
  \begin{subfigure}[T]{0.45\textwidth}
    \begin{tikzpicture}
 \node (A) at (0,1) {$G^{\Z}$};
 \node (B) at (-2,0) {$G^{\Z}/G$};
 \node (C) at (2,0) {$\Z\backslash G^{\Z}$};
 \node (D) at (2,-1) {$\partial^2 G$};
 
 \draw[->] (A) -- (B);
 \draw[->] (A) -- (C);
 \draw[->] (C) -- (D);
\end{tikzpicture}
 \end{subfigure}
  \caption{Comparing the setup in \cite{BF} (on the left) to our setup (on the right).}
\end{figure}

In the end, it is exactly due to this connection between the $G$-action on $\partial^2 G$ and the random walk flow that we are able to show the ergodicity of $\Theta$. This is why we had to move on from the construction of the measure $\Theta$ as seen in Section \ref{sec Theta bd} and \ref{sec Theta disintegration} to seeing $\Theta$ as the push forward of a measure on the quotient, as seen in \ref{sec Theta push-forward}. Indeed, as argued in Remark \ref{comparison versions Theta}, in spirit $\Theta_1$ and $\Theta_2$ align more with the behavior on the $G$-quotient than on the $\Z$ quotient. Hence we cannot hope to gain further information about the behavior of $\Theta$ under the $G$-action on $\partial^2 G$ from these measures. 

Finally, the proofs of exponential mixing of all orders and that of the CLT (Theorem \ref{thm mixing} and Theorem \ref{thm CLT}) rely on quantitative estimates coming from the aforementioned constructions, we refer to Section \ref{sec mixing} for details.

\subsection{Organization of the Rest of the Article}
The remaining part of the paper is organized as follows: In Section \ref{sec background}, we will provide the background on hyperbolic groups needed throughout the paper. In Section \ref{sec bi-inf RW} we introduce the main object used in our constructions, the bi-infinite random walks and deduce some important parts of the long-term behavior of these paths. In Section \ref{sec Theta constructions}, we will construct and derive properties of $\Theta_1$ and $\Theta_2$, while in Section \ref{sec Theta push-forward}, we move on to defining and studying properties of $\Theta_3$ (as outlined above in Section \ref{sec proof ideas}).
In Section \ref{sec ergodicity}, we establish Theorem \ref{big ergodicity intro}. A corollary of this theorem is that $\Theta_3$ is ergodic for the action of $G$ on $\partial^2 G$. Hence,  the final point in Theorem \ref{theorem 1} will also be proved. Finally, Theorem \ref{thm mixing}- \ref{thm CLT} are proved in Section \ref{sec mixing}.

\section{Background and Preliminaries}\label{sec background}

\subsection{Measure Theoretic Terminology} 

We first start with some measure theoretic terminology which will be used in the sequel.

\begin{defn}
A measure $\mu$ on a measure space $(X,\Fcal)$ is said to be \textbf{invariant} for a group action $G$ on $X$, if for all $A\in \Fcal$ and all $g\in G, \,\mu(A) = \mu(gA)$.

The measure $\mu$ is said to be \textbf{quasi-invariant}, if $\mu$ and $g_{\ast}\mu$ are in the same measure class for all $g\in G$, i.e. if they are mutually absolutely continuous.

A quasi-invariant measure is called \textbf{ergodic} with respect to the action of $G$, if for all invariant sets $A\in \Fcal$, i.e.\ sets $A$ with $$\mu(A \backslash gA) + \mu(gA \backslash A) =0 \,\,\  \forall \, g\in G,$$  either $\mu(A)=0$ or $\mu(A^c)=0$.
\end{defn}

\begin{remark}\label{rem ergodicity}
Note that ergodicity is in fact just a property of the measure class of $\mu$, since all measures in the same measure class have the same zero-sets. Hence ergodicity of any particular quasi-invariant measure determines the
ergodicity of any quasi-invariant measure  in the same measure class. In particular,  ergodicity of an invariant measure determines the
ergodicity of any quasi-invariant measure  in the same measure class.
\end{remark}

\begin{theorem}\label{disintegration}(Disintegration of measures)
Let $(\Omega, \P)$ and $(\Omega',\bf{P})$ be two probability spaces and assume that there exists a map $\pi : \Omega \to \Omega'$ such that $\bf{P}= \pi_{\ast}\P$. Then there exist a family of probability measures $(\P_{x})_{x\in \Omega'}$ such that 
\[\P_x(\Omega\backslash \pi^{-1}(x))=0\]
and for any measurable function $g:\Omega \to \R$ we have
\[\int_{\Omega}g(y)\d\bar\P_{o}(y) = \int_{\Omega'}\int_{\pi^{-1}(x)\cross \bd^{-1}(\eta)}g(y)\d\P_x(y)\d\bf{P}(x).\]
\end{theorem}

\subsection{Gromov-hyperbolic groups} 

In this section we will now recall some elementary facts about Gromov-hyperbolic groups and refer the reader to \cite{gro-hyp,GhH,bh-book} for a more detailed treatment of Gromov-hyperbolicity.

\begin{defn}\label{def-gi}\cite{gro-hyp}\cite[p. 410]{bh-book} 
Let $(X,d)$ be a metric space.
For  $x, y, w$ in $(X,d)$, the Gromov inner product $( x, y )_w$ of $x, y$ with respect to the base-point $w \in X$ is given by $$( x, y )_w = \frac{1}{2} (d(x,w) + d(y,w) - d(x,y)).$$
\end{defn}
Hyperbolicity can be defined in terms of the Gromov inner product as follows:
 $(X,d)$ is  \emph{$\delta$-hyperbolic}, if for all $x,y,z,w \in X$,
$$(x,y)_w \geq \min\{(x,z)_w,(z,y)_w\}-\delta.$$ This is equivalent to the usual thin-triangles definition given in Lemma~\ref{delta slim} below (see \cite{bh-book} for details).
For $x,y \in X$, the geodesic between $x$ and $y$ is denoted by $[x,y]$.
 
\begin{lemma}\label{delta slim}
Let $(X,d)$ be a geodesic metric space which is $\delta$-hyperbolic, then for all $x,y,z \in X$, we have
\[[x,y] \subset N_\delta ([x,z] \cup [y,z])\]
where $[x,y]$ denotes a geodesic from $x$ to $y$ and $N_{\delta}(A)$ is the $\delta$-neighborhood of a set $A\subset X$.
\end{lemma}

A finitely generated group $G$ is said to be hyperbolic with respect to some finite symmetric generating set $S$ if the Cayley graph $\Ga = \Ga (G,S)$ is hyperbolic with respect to the word-metric. A standard fact is that if $G$ is  hyperbolic with respect to some finite symmetric generating set $S$, it is hyperbolic with respect to any finite symmetric generating set.
The group $G$ is 
\emph{non-elementary hyperbolic}, if it is infinite, and does not contain a subgroup of finite index which is isomorphic to $\Z$.

Any hyperbolic space $(X,d)$ comes naturally equipped with a boundary at infinity $\partial X$, which is given by 
\[\partial X = \{[\xi] : \xi \text{ quasi-geodesic ray in } X\}\]
where the set of equivalence classes $[\xi]$ is given by the equivalence relation $\xi \sim \eta$ if
$(\xi(t),\eta(s))\to \infty$, as $s,t \to \infty$. 

The Gromov product extends naturally to the boundary $\partial X$ (see \cite{gro-hyp} for details), by setting
\[(\xi,x)_p= \limsup_{\xi_n \to \xi} (\xi_n,x)_p \;, \text{ and } (\xi,\eta)_p= \limsup_{\xi_n \to \xi, \eta_n \to \eta} (\xi_n,\eta_n)_p\]
for $x,p\in X, \xi,\eta \in \partial X$.

Note that, if we chose $\liminf$ instead of $\limsup$ the value of the Gromov product would only change by at most $2\delta$ on $X \cup \partial X$. Indeed, this follows directly from the definition of hyperbolicity.

\begin{lemma}
Let $\xi, \eta\in \partial X$ and $(\xi_n)_{n\in\N}, (\xi_n')_{n\in\N}$ and $(\eta_n)_{n\in\N}, (\eta_n')_{n\in\N}$ be sequences converging to $\xi$ and $\eta$ respectively. Then for $w,z \in X$
\[\limsup_{n,m \to \infty}(\xi_n,\eta_m)_w \leq \limsup_{n,m \to \infty}(\xi'_n,\eta'_m)_w +2\delta,\]
\[\limsup_{n,m \to \infty}(\xi_n,z)_w \leq \limsup_{n,m \to \infty}(\xi'_n,z)_w +\delta.\]
\end{lemma}

The hyperbolic boundary is not just an abstract space. Indeed, it can be seen as a metric space in which the geodesic ray $\xi$ converges to $[\xi]$. 

\begin{defn}\label{def-visualmetric}
	A metric $d_v$ on $\partial X$ is said to be {\bf a visual metric with parameter $a>1$} with respect to the base-point $o$ if there exist $k_1, k_2 > 0$ such that $$k_1 a^{-( \xi, \xi' )_o} \leq d_v (\xi_1, \xi_2) \leq k_2 a^{-( \xi, \xi' )_o}.$$
\end{defn}

\begin{prop}\label{prop-vmexist}\cite[p. 435]{bh-book} Given $\delta \geq 0$, there exists $a > 1$ such that if
	$(X,d)$ is $\delta-$hyperbolic, then for any base-point $o$,
	a visual metric $d_v $
	with parameter $a>1$ exists on $\partial X$ with respect to the base-point $o$.
	Further, for any $o' \in X$, let $d_v'$ be a visual metric on $\partial X$
 with respect to the base-point $o'$ and
	with parameter $a>1$. Then $d_v'$
	is equivalent (as a metric) to $d_v$.
	
	There exists a natural topology on $\hat{X} = X \cup \partial X$ such that
	\begin{enumerate}
		\item $X$ is open and dense in $\hat X$,
		\item $\partial X$ and $\hat{X}$ are compact if $X$ is proper,
		\item the subspace topology on $\partial X$ agrees with that given by $d_v$.
	\end{enumerate}
\end{prop}
We call $\hat{X}$ the {\bf Gromov compactification} of $X$.

For $\xi \in \partial X$ and $o \in X$, a geodesic ray from $o$ and converging to $\xi \in \partial X$ will be denoted by $[o, \xi)$. For $\xi_1 \neq \xi_2 \in \partial X$, a bi-infinite geodesic $f: \R \to X$ converging to $\xi_1, \xi_2$ as $s \in \R$ tends to $\pm \infty$ will be denoted by $(\xi_1, \xi_2)$. 

\subsection{Busemann Functions and Shadows}

On a hyperbolic space $(X,d)$ the {\bf Busemann function} is defined as
\begin{equation}\label{eq:buse}
\beta_\xi (p,q) :=\limsup_{z\to \xi} \, (d(p,z)-d(q,z)).
\end{equation}
for $p,q,z\in X, \xi \in \partial X$. 

\begin{remark}
Note that Busemann functions are also studied for more general metric spaces $(X,d)$. In this case they are defined as
\[\beta_\gamma(x):=\lim_{t \to \infty} \big(d(x, \gamma(t))-t\big)\]
for (parameterized) geodesic rays $\gamma \subset (X,d)$, and $x\in X$. 
In the case of a hyperbolic space, one can then declare as equivalent the set of all Busemann functions for which the defining geodesic rays $\gamma$ have the same limit at $\infty$. This defines an equivalence class $\beta_{\xi}$ of Busemann functions in the horofunction-compactification of $X$, given by the quotient of $C(X)$ under $f \sim g $ if $f-g$ is constant. If we further quotient the boundary of this space by bounded functions, we obtain the Gromov boundary of this space. See for example \cite[Section 2.4]{calegari-notes} for more details. In the context of the present work it suffices however, to use the definition of the Busemann function as given in \eqref{eq:buse}.
\end{remark}

\begin{defn}\cite[Definition 4.1]{coornaert-pjm} \label{qcdens}
For $(X,d)$ Gromov-hyperbolic, let $M(\partial_g X)$ denote the collection of positive finite Borel measures on $\partial_g X $. A $G$-equivariant map $X \to M(\partial_g X)$ sending $x \to \mu_x$ is said to be a $C$ {\bf quasi-conformal density} of dimension $v$ ($v\geq 0$), for some $C\geq 1$, if
\begin{equation}\label{eq-qc}
	\frac{1}{C}\exp\big(-v\beta_\xi(y,x)\big)
	\leq
	\frac{\d\mu_{x}}{\d\mu_y} (\xi)
	\leq C \exp\big(-v\beta_\xi(y,x)\big)
\end{equation}
for all $x, y \in X$, $\xi\in \partial X$. If $C=1$, it is a {\bf conformal density}. Equivalently, we also call the family of measures $(\mu_x)_{x\in G}$ a (quasi)-conformal density.
\end{defn}

Let $R>0$ and $x,y\in X$. The {\bf shadow} $\mho_y(x,R)$ is given by
\[\mho_y(x,R) := \{a\in\partial X: (a,x)_w\geq d(w,x)-R\}.\]
Equivalently, $\mho_x(w,R)$ is the collection of all limits on the boundary of geodesic rays starting at $x$ and passing through $B_R(x)$. 
Shadows come into play when dealing with conformal measures due to the following Shadow lemma, originally due to Sullivan in (\cite{sullivan}, Proposition 3), and rendered  in the context of hyperbolic groups in (\cite{coornaert-pjm}, Proposition 6.1). 

We refer the reader to Remark~\ref{rmk-qr} following Theorem~\ref{thm-bhm} for the notion of quasi-ruled spaces. A metric space is said to be \emph{cocompact} if it admits a cocompact isometric action 
by a discrete subgroup of the group of its isometries.

\begin{lemma}[Shadow Lemma]\label{shadow lemma} 
For a hyperbolic group, or a cocompact hyperbolic quasiruled space $X$, and a family of quasi-conformal measures $(\mu_x)_{x\in X}$ of dimension $v$
 there exists $R_0>0$, such that for $R>R_0$, and for any $w,x\in X$, 
 \[\mu_w(\mho_w(x,R))\asymp \e^{-v d(w,x)}\]
 where the implicit constants do not depend on $x$.
\end{lemma}
This comes from the observation that for any $\tau>0$ there exist $C,R_0>0$ such that for $R> R_0$ and $\xi \in \partial X, x \in X$ with $(w,\xi)_x \leq \tau$, we have
\[
B_{\eps}(a,(1/C) \e^{-\eps d(w,x)}) \subset \mho_{\GG} (x,R) \subset B_{\ep}(a,C \e^{-\eps d(w,x)}).
\]

\subsection{The Bonk-Schramm Embedding}
Let $(X,d_X)$ and $(Y, d_Y)$ be metric spaces.
We say, following \cite{BS} that $X$ is \emph{roughly similar} to $Y$ if there exists a map
$f : X \to Y$ and constants $k, A> 0$, such that
for all $x, y \in X$
$$|Ad_X(x, y) - d_Y(f(x), f(y)) | \leq k,$$ and $Y$ lies in a $k-$neighborhood of $f(X)$. This notion is considerably stronger than just a quasi-isometry.

One very helpful tool in dealing with hyperbolic groups is the following main theorem of \cite[Theorem 10.2]{BS}.

\begin{theorem}\label{thm-bs}
Let $(X,d)$ be a Gromov-hyperbolic geodesic space. Then there exists an
$n \in \natls$ such that $X$ is roughly similar to a convex subset of hyperbolic
$n$-space $\H^n$.
\end{theorem}
(see also Remark \ref{rmk-qr} below). 

\subsection{Green Metric}\label{green metric}

Let $\mu$ be a symmetric measure on $G$ whose  $S$ generates $G$ and $(X_n)_{n\in \N}$ a random walk on $\Gamma$ with transition probabilities given by $p(x,y)= \mu(x^{-1}y)$. If $Y_1,Y_2,...$ are independent and identically distributed random variables on $S$ with distribution $\mu$, then the random walk started at $x\in G$ is given by
\[X_0:=x, \; X_n:= xY_1\cdot Y_2\cdot...\cdot Y_n.\]
We denote by $\P_x$ the law of $(X_n)_{n\in\N}$ when started at $x$. 

\begin{theorem}\cite[Cor. 1]{Woe93}\label{divergence RW}

A random walk $(X_n)_{n\in\N}$ on a non-elementary hyperbolic group $G$ with transition probabilities given by $p(x,y)= \mu(x^{-1}y)$, where $\mu$ is a measure whose support is a finite generating set of $G$, converges to a random point $X_{\infty}\in \partial G$ almost surely. The hitting distribution on $\partial G$ is denoted by 
\begin{equation}\label{def nug}
\nu_g(\cdot)= \mathbb P_g[X_\infty \in \cdot].
\end{equation}
\end{theorem}
See also \cite[Thm. 7.6]{Kai00} for a version of this statement with a more general step distribution.

We define \emph{the Green’s function} $\GG : \Ga \times \Ga \to [0, \infty]$ by $\GG(x, y) = \sum_{0}^{\infty}
\P_x [X_n = y]$. $\GG(x, y)$ is the expected number of times a random walk starting at $x$ visits $y$. The \emph{hitting probability} $F : \Ga \times \Ga \to [0, \infty]$ is given as follows: $F(x, y)$ is the probability that a random walk starting at $x$ visits $y$.
Thus, $$\GG(x, y) = F (x, y)\GG(e, e).$$
The Green metric on $\Gamma$, first introduced by Blachère and Brofferio in \cite{BB07},  is then given by $$d_{\skrig}(x,y) = -\log F (x, y).$$
Note that it was shown in \cite[Lemma 3.2]{BHM11} that $(\Gamma, d_{\skrig})$ is a proper metric space. Further, by symmetry of the random walk, $\Gamma$ acts on $(\Gamma, d_{\skrig})$ by isometries. 
The volume growth of $(\Gamma,d_{\skrig})$ is exponential with volume growth constant $\nu_G = 1$, as was observed for example in \cite[Sec. 3.4]{BHM11} and \cite{BHM08}.

Further,
\begin{theorem} \cite[Cor. 1.2]{BHM11}\label{thm-bhm}
\begin{enumerate}
	\item The identity map on the vertex set of $\Gamma$ gives a quasi-isometry between the Cayley graph and $(\Gamma, d_{\skrig})$.
	\item $(\Gamma, d_{\skrig})$ is hyperbolic.
\end{enumerate}

\end{theorem}

\begin{remark}\label{rmk-qr}
    We point out that $(\Gamma, d_{\skrig})$ is not quite a geodesic metric space. 
However, by the quasi-isometry in the above theorem, it is quasi-geodesic. Furthermore, there exists $\tau>0$, such that it is $\tau$-\emph{quasi-ruled} in the sense of \cite{BHM11}, i.e.\ for all $x,y \in G$ and $z \in [x,y]$, we have
\[d_{\skrig}(x,z)+d_{\skrig}(y,z) \leq d_{\skrig}(x,y) +\tau.\]

This follows since $(G,d_{\skrig})$ may be embedded isometrically into a $\delta$-hyperbolic geodesic space $(G^{\ast},d^{\ast})$ (see \cite{BS}, Theorem 4.1). Denote the embedding by $\iota$. 
Since quasi-geodesics in a hyperbolic space are within bounded distance of real geodesics, there exists a constant $M>0$ which is independent of $x,y$, such that $[\iota(x),\iota(y)]$ and $\iota([x,y])$ are within distance $M$ of each other. Hence for $z \in [x,y]$
\begin{align*}
d_{\skrig}(z,x)+d_{\skrig}(z,y)&= d^{\ast}(\iota(x),\iota(z))+d^{\ast}(\iota(z),\iota(y))\leq d^{\ast}(\iota(x),z^{\ast})+d^{\ast}(z^{\ast},\iota(y))+2M \\
&= d^{\ast}(\iota(x),\iota(y))+2M = d_{\skrig}(x,y)+2M
\end{align*}
where $z^{\ast}$ is a point on $[\iota(x),\iota(y)]$ at distance smaller than $M$ to $\iota(z)$. 
See \cite[Appendix A]{BHM11} for a more detailed study of the relation between quasi-ruled and geodesic spaces. 

Note that via the embedding $\iota$, we may also use the Bonk-Schramm embedding to embed $(G,d_{\skrig})$ into $\H^d$ for some $d>0$ by a rough similarity.
\end{remark}
	
\subsection{Ancona Inequality.} One important tool in showing that $(\Gamma,d_{\skrig})$ is hyperbolic is the Ancona inequality. (We refer the reader to \cite{gouezel} for a significant extension of the Ancona inequality.)
\begin{lemma}\cite[Thm. 5]{ancona}\label{ancona} Let $G$ be a non-elementary hyperbolic group, and $\Gamma$ a Cayley graph of $G$ with graph metric $d$.
  Let $\rho$ be a finitely supported symmetric probability measure
 the support of which generates $G$.
 For any $r\geq 0$, there is a constant $C(r)\geq 1$ such that $$F(x,v)F(v,y)\leq F(x,y)\leq C(r)F(x,v)F(v,y)$$
 whenever $x,y\in \Gamma$ and $v$ is at distance at most $r$ from a geodesic segment between $x$ and $y$.
\end{lemma}

Another interesting value in connection to the Green function is the Martin kernel defined by
\[K(x,y):=\frac{\skrig(x,y)}{\skrig(o,y)} \quad \forall x,y \in G.\]
The Martin boundary $\partial_M G$ is then the smallest compactification of $G$ to which the kernels $K(x,\cdot)$ extend continuously. We set
\[K(x,\xi):= \limsup_{\xi_n \to \xi}K(x,\xi_n)\quad \forall x\in G, \xi \in \partial_M G.\]

\begin{theorem}\cite[Thm. 8]{ancona}
Let $G$ be a non-elementary hyperbolic group and $\mu$ be a symmetric probability measure on $G$ supported on a finite generating set of $G$. Then $\partial_M G$ is homeomorphic to the hyperbolic boundary of $G$.
\end{theorem}

Further, through this connection of the two boundaries, it can be shown that the hitting measures $(\nu_x)_{x \in G}$ are mutually absolutely continuous and their density is given by the Martin kernel (cf. \cite{woess}, Theorem 24.10). This in particular also means that the hitting measures are a family of conformal measures of dimension $1$. See also \cite[Sec. 3]{BHM11} for a more detailed discussion of the Martin boundary and its connection to the hyperbolic boundary and the hitting distributions.

\begin{theorem}\label{density hitting measures}
The hitting distributions $(\nu_x)_{x\in G}$ a the random walk on $G$ satisfy
\[\frac{\d \nu_y}{\d \nu_o}(\xi) = K(y^{-1}x,\xi)= \exp\big\{\limsup_{\xi_n \to\xi}d_{\skrig}(o,\xi_n)-d_{\skrig}(y,\xi_n)\big\}.\]

\end{theorem}


\section{Bi-infinite Random Walk Paths and Deviations}\label{sec bi-inf RW}

In Section \ref{bi-inf RW} we will define the bi-infinite random walk paths by taking a random walk $(X_n)_{n\in\N}$ started at $o$ and successively shifting it back by $X_n$, the position at time $n$. As remarked before, this will preserve the distribution of the geometric properties of the original random walk and will allow us to transfer the estimates on the distance of the random walk to the geodesics to the current bi-infinite setting. These estimates will be shown in Section \ref{deviation from geodesic}, while in Section \ref{deviation from geodesic disintegrated} we will develop a refined version of these estimates as a disintegrated version w.r.t. the measures $\P_o(\cdot | \xi)$ conditioned on the random walk trajectory hitting $\xi\in \partial G$.

\subsection{The Bi-infinite Random Walk}\label{bi-inf RW}

\begin{lemma}\label{shifting RW}
Let $(X_n)_{n\in\N}$ be a random walk on a hyperbolic group $G$ with symmetric step distribution $\mu$ whose support is finite and generates $G$ and starting point $X_0=o$. Fix $m\in \N$ and set
\begin{align*}
\bar Z^m_n &:= X_m^{-1}X_{m-n} &&\text{ for } n=0,...,m;\\
Z^m_n &:= X_m^{-1}X_{m+n} &&\text{ for } n\geq m\\
\end{align*}
Then the following hold:
\begin{itemize}
\item $\big(\bar Z^m_n\big)_{n=0}^m$ and $\big(Z^m_n\big)_{i=0}^{\infty}$ are independent for any $m\in\N$. 
\item Furthermore, they converge in distribution to independent random walks $\big(\bar Z_n\big)_{n\in \N}$ and $\big(Z_n\big)_{n\in\N}$ with step distribution $\mu$ and starting point $o$.
\item Finally, $\big(\bar Z_n\big)_{n\in \N}$ and $\big(Z_n\big)_{n\in\N}$ converge a.s to points $Z_{\infty}$ and $Z_{-\infty}$ on $\partial G$.
\end{itemize}
\end{lemma}

\begin{proof}
Let $W_1,W_2,...\in S$ be the increments of the original random walk $(X_n)_{n\in\N}$, i.e.\  $X_n = W_1\cdot...\cdot W_n$. Then for $n=1,...,m$

\begin{align*}
\bar Z^m_n = X_{m}^{-1}X_{m-n} &= (W_1\cdot...\cdot W_n)^{-1} Y_{1}\cdot...\cdot W_{m-n} \cdot \\
&= W_{n}^{-1}\cdot...\cdot W_1^{-1} \cdot W_1\cdot...\cdot W_{n-m} = W_{n}^{-1}\cdot...\cdot W_{n-m+1}^{-1}
\end{align*}
and similarly 
\begin{align*}
Z^m_n = X_{m}^{-1}X_{m+n} &= (W_1\cdot...\cdot W_n)^{-1}\cdot W_{1}\cdot...\cdot W_{m+n} = W_{m+1}\cdot...\cdot W_{m+n}.
\end{align*}
This means that $(\bar Z^m_n)_{n=0}^m$ only depends on $W_1,...,W_m$ and $(Z^m_n)_{n=0}^{\infty}$ on $(W_{n})_{n\geq m}$. Since the increments $W_1,W_2,...$ are independent, $(\bar Z^m_n)_{n\in \N}$ and $(Z^m_n)_{n\in\N}$ are independent for every fixed $m\in\N$, and hence this also holds when we take the limit in distribution.

\noindent Furthermore, since the increments are identically distributed, for any $m\in\N$ the distribution of $(Z^m_n)_{n\in\N}$ is the one of a random walk with step distribution $\mu$ and origin $o$. Since this holds independently from $m$, the convergence to $(Z_n)_{n\in\N}$ holds trivially.

\noindent Finally, since $\mu$ is symmetric, $W_i$ has the same distribution as $W_i^{-1}$, i.e.\  $(\bar Z^m_n)_{n=0}^{m}$ has the same distribution as the first $m$ steps of a random walk. 
We extend its distribution to $G^{\N}$ by taking the product with $\mathbbm{1}_G$ for each of the missing steps. Then on any cylinder set this extended distribution converges to the one of a random walk with step distribution $\mu$ and origin $o$. So, the extended measure converges to $(\bar Z_n)_{n\in\N}$.
\end{proof}

\begin{remark}
Note that in the same way we can also construct a bi-infinite random walk with any other fixed origin $X_0=g \in G$. In this case we obtain two independent copies of the random walk started at $X_0=g$ by successively shifting back $(X_n)_{n\in\N}$.
\end{remark}

\begin{defn}\label{def bi-inf RW} $ $
\begin{itemize}
\item Given two sequences $\bar z:= (\bar z_n)_{n\in \N}$ and $z:=(z_n)_{n\in \N}$ with $\bar z_0=z_0$, we denote by $(\bar z\oplus z)$ the bi-infinite sequence $(x_n)_{n\in\N}$ given by
\[ \bar z_n \oplus z_n:= x_{n} := \begin{cases} \bar z_{-n} &\text{if } n \leq 0,\\ z_{n} &\text{if } n > 0.
\end{cases}
\]
\item Given two independent random walks $\big(Z_n\big)_{n\in\N}$ and $\big(\bar Z_n\big)_{n\in\N}$ starting at $g\in G$ and with step distribution $\mu$, define the \textbf{bi-infinite random walk with fixed origin $g$ and step distribution $\mu$} as 
\[\big(X_z\big)_{z\in\Z}:= \bigg(\big(\bar Z_n\big)_{n\in\N} \oplus \big( Z_n\big)_{n\in\N}\bigg)\] 
\item We will denote by {\bf $\bar\P_{g}$ the distribution of $(X_z)_{z\in\Z}$, equivalently, the joint distribution of $\big(Z_n\big)_{n\in\N}$ and $\big(\bar Z_n\big)_{n\in\N}$}. We will also write 
\begin{equation}
G^{\Z}:= \bigg\{ \text{all bi-infinite random walk paths } (X_{z})_{z\in\Z} \text{ in } G\bigg\} 
\end{equation}
and for $g \in G$
\begin{equation}
G_{g}^{\Z}:= \bigg \{ \text{all bi-infinite random walk paths } (X_{z})_{z\in\Z} \text{ in } G \text{ with } X_{0}=g \bigg \}.
\end{equation}
We then have that $G^{\Z} = \bigcup_{g\in G}G^{\Z}_g$ is a disjoint union.
\end{itemize}

\end{defn}

\begin{remark}
We note that Lemma \ref{shifting RW} shows that by successively shifting back the random walk $(X_n)_{n\in\N}$ by $X_m$ and reversing time in one half, we end up with two independent copies of $(X_n)_{n\in\N}$. Hence, this lemma justifies the construction of the bi-infinite random walk with fixed origin by gluing together two independent random walks and running time backwards on one half as given in Definition \ref{def bi-inf RW}.

\noindent We also note that Definition \ref{def bi-inf RW} provides us with two different points of view on the bi-infinite random walk with fixed origin $X_0=g$: 
\begin{enumerate}
\item It can be seen as a bi-infinite sequence $(X_z)_{z\in\Z}$, whose distribution we denote by $\bar\P_g$. 
\item Alternatively, it can be understood as a pair of independent random walk paths, whose joint distribution is given by $(\P_g\otimes \P_g)$. 
\end{enumerate}
These two different points of view will be helpful when studying the bi-infinite random walk in the sequel.
 
\end{remark}

\subsection{Deviation from the Geodesic}\label{deviation from geodesic}

The goal of this section is to study the long term behavior of the bi-infinite random walks $(X_z)_{z\in\Z}$ defined in the previous section. First in Theorem \ref{deviation bds} we first study the case of a usual/ non-bi-infinite random walk $(Y_n)_{n\geq 0}$ and show deviation bounds for fixed points $Y_n$. These bounds can then be used to study the joint behavior of two independent random walks -- see Lemma \ref{gp independent}. We can use this result to show that the two hitting points $X_{-\infty}$ and $X_{\infty}$ are almost surely distinct -- see Corollary \ref{non-sing}. Consequently, the behavior of our bi-infinite random walk can be studied and in particular the asymptotic behavior of the distance of $X_n$ to a geodesic $(X_{-\infty}, X_{\infty})$ can be quantified -- see Corollary \ref{almost geodesic}, which also justifies our point of view of considering bi-infinite random walk trajectories as generalized geodesics.

\begin{theorem}\label{deviation bds}
Let $(Y_n)_{n\in\N}$ be a usual random walk started at $o$ with step distribution $\mu$.
\begin{enumerate}
\item[(1)]\label{deviation distance} For any $D\geq 0$ and $n\geq 0$ sufficiently large, there exists $C_D>1$ such that 
\[ C_D^{-1} \e^{-D} \leq \P_{o}\big[d_{\skrig}\big(Y_n, [o,Y_{\infty})\big)\geq D\big]\leq C_D \e^{-D}.\]
The constants $C_D$ are bounded in $D$. 

\item[(2)]\label{deviation gp} This means in particular that for $n>0$ big enough also
\[ \P_{o}\big[\big(Y_n,Y_{\infty}\big)^{\skrig}_{o}< D\big]\ \lesssim \e^{-D}\]
\end{enumerate}
\end{theorem}

\begin{proof}
(1):  Conditioning on the position of the random walk at time $n$, we get
\begin{equation}\label{devbd eq}
\begin{aligned}
\P_{o}\Big[d_{\skrig}\big(Y_n, [o,Y_{\infty})\big)\geq D\Big]& = \sum_{y\in G} \P_{o}\Big[d_{\skrig}\big(Y_n, [o,Y_{\infty})\big)\geq D, Y_n=y\Big]\\
&= \sum_{y\in G} \P_o\Big[d_{\skrig}\big(y, [0,yY_n^{-1}Y_{\infty})\big)\geq D\Big\vert Y_n=y\Big]\P_{o}[Y_n=y]\\
&= \sum_{y\in G} \P_y\Big[d_{\skrig}\big(y, [o,Y_{\infty})\big)\geq D\Big]\P_o[Y_n=y]
\end{aligned}
\end{equation}
where we shifted $Y_{\infty}$ back by $Y_n^{-1}$, to get independence from $Y_n$ and then used that $yY_n^{-1}Y_{\infty}$ has the same distribution as the hitting distribution of a random walk started at $y$.

\noindent In order to prove the bounds in the claim we hence want to show that $\P_x[d(x, [o,Y_{\infty}))\geq D]$ fulfills the same exponential bounds. This will be done by showing that there exist $p_l,p_u \in G$ and $R_l,R_u>0$ such that
\begin{equation}\label{relation to shadow}
\{Y_{\infty} \in \mho_y(p_l,R_l)\}\subseteq \{d(y, [o,Y_{\infty}))\geq D\}\subseteq \{Y_{\infty} \in \mho_y(p_u,R_u)\}.
\end{equation} 
We can then use the Shadow Lemma \ref{shadow lemma}, which says that for $y,p\in G$ and $R>R_0$ we have 
\[\P_y[Y_{\infty}\in \mho_y(p,R)]\asymp \e^{-d_{\skrig}(p,y)} = \e^{-D- \tilde c}\]
to get the desired bounds.

\noindent So let us first prove the second inclusion in \eqref{relation to shadow}. On the event $\{d_{\skrig}(y, [o,Y_{\infty}))\geq D\}$ we have in particular $d_{\skrig}(y,o)>D$, so we can choose $p \in [o,y]$ such that $d_{\skrig}(p,x)=D +\tilde c$, where $\tilde c$ is a small constant stemming from the fact that by the coarse geometric nature of $(G,d_{\skrig})$ there might not exist a point at distance exactly $D$ from $x$. Then we consider the triangle $\Delta(o,y,Y_{\infty})$. By hyperbolicity of $G$, we know that we may map the triangle to a comparison tripod by a $(1,c)-$quasi-isometry $T$. Let $d_T$ be the metric on the tripod, $m$ the centroid of the tripod and denote by $m_{0}\in [y,Y_{\infty})$ the point whose image under $T$ is closest to $m$. Further we denote by $p_1,p_2$ and $k_1,k_2$ the two possible positions of the images of points $p\in[o,y]$ and $k\in[o,Y_{\infty}]$ under $T$ in relation to the centroid $m$.
\begin{figure}
\begin{tikzpicture}[scale=0.7]
\coordinate[label = left : $y$] (A) at (0,0);
\coordinate[label = right : $o$] (B) at (5,0);
\coordinate[label = above : $Y_{\infty}$] (C) at (3,7) ;
\coordinate[label = below : $p$] (D) at (4,.375) ;
\coordinate[label = right : $k$] (E) at (3.75,2.5) ;
\coordinate[label = left : $m_{0}$] (F) at (1.56,2) ;

\coordinate[label = left : $y$] (a) at (8,0);
\coordinate[label = right : $o$] (b) at (13,0);
\coordinate[label = above : $Y_{\infty}$] (c) at (11,7);
\coordinate[label = right : $m$] (d) at (11,2);
\coordinate[label = below : $p_1$] (e) at ($(a)!0.6!(d)$);
\coordinate[label = below : $p_2$] (f) at ($(b)!0.7!(d)$);
\coordinate[label = right : $k_1$] (g) at ($(c)!0.6!(d)$);
\coordinate[label = right : $k_2$] (h) at ($(b)!0.5!(d)$);

\fill (A) circle (1pt);
\fill (B) circle (1pt);
\fill (C) circle (1pt);
\fill (D) circle (1pt);
\fill (E) circle (1pt);
\fill (F) circle (1pt);
\fill (a) circle (1pt);
\fill (b) circle (1pt);
\fill (c) circle (1pt);
\fill (d) circle (1pt);
\fill (e) circle (1pt);
\fill (f) circle (1pt);
\fill (g) circle (1pt);
\fill (h) circle (1pt);

\draw (0,0) .. controls (2.5,1) .. (5, 0);
\draw (5,0) .. controls (3.7,2) .. (3, 7);
\draw (3,7) .. controls (2,2.3) .. (0, 0);

\draw [->, thick] (5.5,3.5) to node[above]{$(1,c)$-q.i.} (8,3.5);
\draw (8,0) --(11, 2);
\draw (13,0) -- (11, 2);
\draw (11,7) -- (11, 2);
\end{tikzpicture}
\caption{Possible positions of the images of $p$ and $k$ in the tripod under the $(1,c)$-quasi isometry.}
\label{Figure 1}
\end{figure}
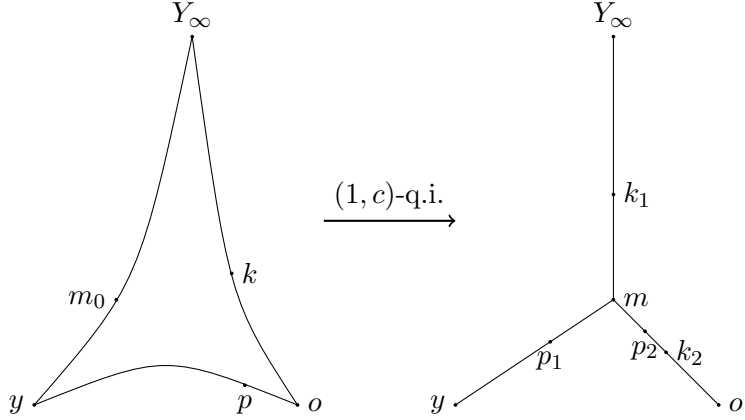
We will now turn to proving the upper bound. We want to show, that there exists $R>0$ independent of $y,p$ and $D$, such that $[y,Y_{\infty})$ passes through $B_R(y)$, i.e.\ $Y_{\infty}\in \mho_y(p,R)$.

\noindent Now, we have either $d_T(T(y),T(p))\leq d_T(T(y),m)$ or $d_T(T(y),T(p))) > d_T(T(y),m)$. In the first case, corresponding to $p_1$ in \ref{Figure 1}, it follows directly that $d_{\skrig}(p,[y,Y_{\infty}))<2 \delta$, since by construction the sides that get mapped onto the same arm of the tripod, are less than $2\delta$ apart.
In the second case we have
\[D \leq d_{\skrig}(y,[o,Y_{\infty})) \leq d_T(T(y),m)+c\]
meaning in particular that by the definition of the metric $d_T$
\[
D-c \leq d_T(T(y),m)= d_T(T(y),T(p))-d_T(m,T(p)) \leq d_{\skrig}(y,p)-d_T(m,T(p)).
\]
Hence, plugging in that $d_{\skrig}(y,p) = D+\tilde c$
\[d_{\skrig}(m_{0},p)\leq d_T(m,p) +c \leq \tilde c + 2c.\]
Hence, in both cases, $[y,Y_{\infty})$ passes through $B_{R}(p)$ for $R = \tilde c + 2c$. So $\{d_{\skrig}(y,[o,Y_\infty])\geq D\}\subset \{X_{\infty}(x)\in \mho_y(p,R)\}$ and hence
\[\P_y[d_{\skrig}(y,[o,Y_\infty])\geq D]\leq \P_{y}[X_{\infty}\in \mho_y(p,R)] \asymp e^{-d_{\skrig}(y,p)}=e^{-D-\tilde c}.\]
Plugging this into \eqref{devbd eq}, we get the upper bound in the claim. 

\noindent For the lower bound, let $Y_{\infty} \in \mho_y(p,R)$ for points $y,o,p\in G$ such that $p \in [y,o], d_{\skrig}(o,y)>D$, and $d_{\skrig}(y,p)= D+\tilde c$ where $\tilde c$ again comes from the fact that there might not be a point exactly at distance $D$. We again regard the triangle $\Delta(y,o,Y_{\infty})$ and switch to the corresponding comparison tripod, and let $m$ be the centroid of the tripod. We again regard the two cases $d_T(T(y),T(p))< d_T(T(y),m)$ and $d_T(T(y),T(p))< d_T(T(y),m)$.

\noindent In the first case, we get that for any point $k\in [o,Y_{\infty})$ that by quasi-isometry of $d_T$ and $d_{\skrig}$
\[d_{\skrig}(y,k)\geq d_T(T(y),m)-c\geq d_{\skrig}(y,m)-2c.\]
In the second case we use the fact that $Y_{\infty}$ is in the shadow $\mho_y(p,R)$, i.e.\ there exists $p_i\in [x,Y_{\infty})$ such that $d_{\skrig}(p_i,p)\leq R$. Then
\[d_T(T(p),m)\leq d_T(T(p),T(p_i))\leq d_{\skrig}(p_i,p)+c\leq R+c\]
and hence again for any point $k\in [o,Y_{\infty})$
\begin{align*}
d_{\skrig}(x,k)&\geq d_T(T(y),T(k))-c = d_T(T(y),m)+d_T(m,T(k))-c = \\
&= d_T(T(y),T(p))- d_T(T(p),m) +d_T(m,T(k))-c \geq d_{\skrig}(x,p)-(R+2c)\end{align*}
and hence $d(x, [o,Y_{\infty}(x)))\geq D-D_0$ where $D_0=R+2c+\tilde c$. Thus we have shown that 
\[\{d(x, [o,Y_{\infty}(x)))\geq D-D_0\}\subset \{Y_{\infty} \in \mho_y(p,R)\}\]
So in particular
\[\e^{-D-c} \asymp \P[X_\infty(x)\in \mho_x(p,R)]\leq \P_x[d_{\skrig}(x,[o,X_{\infty}))\geq D-D_0]\]
and thus
\begin{align*}\P_{o}[d(X_n,[o,X_{\infty}))\geq D] &=\sum_{x\in G} \P_x[d_{\skrig}(x, [0,X_{\infty}(x)))\geq D]\P_{o}[X_n=x]\\
&\geq \sum_{x\in G, d_{\skrig}(x,0)\geq D+D_0} \P_x[d_{\skrig}(x, [0,X_{\infty}(x)))\geq D]\P_{o}[X_n=x]\\
&\gtrsim \P_{o}[d_{\skrig}(X_n,0) \geq D+D_0] \e^{-D-c}.
\end{align*}
To get rid of the dependence on $n$, we observe that by transience $\P_{o}[d(X_n,0) \geq D+D_0] \to 1$ and hence for
$\varepsilon>0$ there exists $N\in \N$, such that $\P_{o}[d(X_n,[o,X_{\infty}))\geq D]\gtrsim (1-\varepsilon)e^{-D-c}$ for $n\geq N$.

\noindent (2): To prove the second claim, we note that 
\[(Y_n,Y_{\infty})^{\skrig}_o = d_{\skrig}(Y_n,o)- (Y_{\infty},o)^{\skrig}_{Y_n}\]
and since there exists $c>0$ dependent on $M$ and $\delta$ such that
\[d_{\skrig}\big(Y_n,[o,Y_\infty)\big)- c \leq (Y_{\infty},o)^{\skrig}_{Y_n}\leq d_{\skrig}\big(Y_n,[o,Y_\infty)\big)+ c\]
by part (1) in particular also
\[\P_o[(Y_{\infty},o)_{Y_n}\geq D]\asymp \e^{-D}.\]
Hence we have
\begin{align*}
\P_o\big[(Y_n,Y_{\infty})_o^{\skrig}\geq D\big] &= \P_o\big[d_{\skrig}(Y_n,o)- (Y_{\infty},o)^{\skrig}_{Y_n}\geq D\big]\\
&\geq \P_o\big[d_{\skrig}(Y_n,o)\geq 2D\big]-\P_o\big[ (Y_{\infty},o)^{\skrig}_{Y_n}\geq D\big]\\
&\geq \P_o\big[d_{\skrig}(Y_n,o)\geq 2D\big]- c_l \e^{-D}.
\end{align*}
And since $\P_o[d_{\skrig}(Y_n,o)\geq 2D] \to 1$, this is in particular bounded below by $1- (c_{l}+1)\e^{-D}$ for n big enough, so
\begin{equation}\label{deviation gp detailed}
\P_o\big[(Y_n,Y_{\infty})_o^{\skrig}< D\big]\leq \Big(1-\P_o\big[d_{\skrig}(Y_n,o)\geq 2D\big]\Big) + c_l \e^{-D}\lesssim \e^{-D}.
\end{equation}
\end{proof}

\begin{lemma}\label{gp independent}
Let $(Y_n)_{n\in\N}$ and $(Z_n)_{n\in\N}$ be two independent usual random walks on $G$ with step distribution $\mu$ and starting at $o$ and $Y_{\infty}, Z_{\infty}$ the respective hitting points on $\partial G$. Then for $n>D$ big enough
\begin{enumerate}
\item \label{gp on bd} $ \P_o\Big[\big(Y_n,Z_{\infty})^{\skrig}_o\geq D\Big]\asymp \e^{-D}$, and 
\item \label{gp in path} $ \P_o\Big[\big(Y_n,Z_n)^{\skrig}_o\geq D\Big]\asymp \e^{-D}$.
\end{enumerate}
\end{lemma}

\begin{proof}
By the independence of $(Y_n)_{n\in\N}$ and $(Z_n)_{n\in\N}$, $Y_n$ is also independent of $Z_{\infty}$ and hence 

\begin{equation}\label{gp independent eq1}
\P_o\Big[\big(Y_n,Z_{\infty}\big)^{\skrig}_{o}\geq D\Big] = \sum_{y \in G}\P_o\Big[Y_n=y\big]\P_o\Big[\big(y,Z_{\infty}\big)^{\skrig}_{o}\geq D\Big].
\end{equation}
We now use the Shadow Lemma \ref{shadow lemma} to approximate this sum. So let $R>0$ be as in Lemma \ref{shadow lemma} and $y\in G$ such that $d_{\skrig}(y,o)\geq D+R$. Denote by $x_y$ the point in $[o,y]$ which has distance to $o$ closest to $D+R$. Then
\[\P_o\big[\big(y,Z_{\infty}\big)^{\skrig}_{o}\geq D\big] = \P_o\big[Z_{\infty}\in \mho_{o}(x,R)\big]\asymp \e^{-d_{\skrig}(o,x)}.\]
Putting this into \ref{gp independent eq1}, we get
\begin{equation}\label{gp on bd upper detailed}
\P_o\big[\big(Y_n,Z_{\infty}\big)^{\skrig}_{o}\geq D\big] \lesssim \P_o\big[d_{\skrig}(Y_n,o) \geq D+R\big] \e^{-D}+ \P_o\big[d_{\skrig}(Y_n,o) \leq D+R\big]
\end{equation}
and
\begin{equation}\label{gp on bd lower detailed}
\P_o\big[\big(Y_n,Z_{\infty}\big)^{\skrig}_{o}\geq D\big] \gtrsim \P_o\big[d_{\skrig}(Y_n,o) \geq D+R\big] \e^{-D}.
\end{equation}
Together with the fact that $\P[d_{\skrig}(Y_n,o) \leq D+R]\to 0$, the claim in point \ref{gp on bd} follows.

\noindent For point \ref{gp in path}, we again use the independence of $Y_n$ and $Z_n$ to get
\begin{equation}\label{gp independent eq2}
\P_o\Big[\big(Y_n,Z_n\big)^{\skrig}_{o}\geq D\Big] = \sum_{y \in G}\P[Y_n=y]\P_o\Big[\big(y,Z_n\big)^{\skrig}_{o}\geq D\Big].
\end{equation}
In order to apply what we have already shown, we now want to replace $\big(y,Z_n\big)^{\skrig}_{o}$ by $\big(y,Z_{\infty}\big)^{\skrig}_{o}$. By the $\delta$-hyperbolicity of $G$, we have
\[\big(y,Z_n\big)^{\skrig}_{o} \geq \min\bigg\{\big(y,Z_{\infty}\big)^{\skrig}_{o}, \big(Z_n,Z_{\infty}\big)^{\skrig}_{o}\bigg\}-\delta\]
as well as
\[\big(y,Z_{\infty}\big)^{\skrig}_{o} \geq \min\bigg\{\big(y,Z_{n}\big)^{\skrig}_{o}, \big(Z_n,Z_{\infty}\big)^{\skrig}_{o}\bigg\}-\delta.\]
And since by definition $Z_n \to Z_{\infty}$ almost surely, we know that also $(Z_n,Z_{\infty})_{o}\to \infty$ almost surely. Hence with high probability the minima above will be given by $\big(y,Z_{\infty}\big)^{\skrig}_{o}$ and $\big(y,Z_{n}\big)^{\skrig}_{o}$ respectively. This gives us

\begin{align*}
\P_o\Big[\big(y,Z_n\big)^{\skrig}_{o}\geq D\Big]&\leq \P_o\Big[\big(y,Z_n\big)^{\skrig}_{o}\geq D, \big(Z_n,Z_{\infty}\big)^{\skrig}_{o}\geq D\Big] + \P_o\Big[\big(Z_n,Z_{\infty}\big)^{\skrig}_{o}< D\Big]\\
&\leq \P_o\Big[\big(y,Z_{\infty}\big)^{\skrig}_{o}\geq D-\delta \Big] + \P_o\Big[\big(Z_n,Z_{\infty}\big)^{\skrig}_{o}< D\Big]
\end{align*}
and 
\begin{align*}
\P_o\Big[\big(y,Z_n\big)^{\skrig}_{o}\geq D\Big]&\geq \P_o\Big[\big(y,Z_n\big)^{\skrig}_{o}\geq D, \big(Z_n,Z_{\infty}\big)^{\skrig}_{o}\geq D\Big] \geq \P_o\Big[\big(y,Z_{\infty}\big)^{\skrig}_{o}\geq D +\delta \Big].
\end{align*}
Plugging the bounds shown for $(Y_n,Z_{\infty})_o^{\skrig}$ in point \ref{gp on bd} into Equation \eqref{gp independent eq2}, we get the upper bound
\begin{align*}
\P_o\big[\big(Y_n,Z_n\big)^{\skrig}_{o}\geq D\big]\leq& \sum_{y\in G} \P_o[Y_n=y]\P_{o}\Big[\big(y,Z_{\infty}\big)^{\skrig}_{o}\geq D-\delta \Big] + \P_o\Big[\big(Z_n,Z_{\infty}\big)^{\skrig}_{o}< D\Big]\\
= \,& \P_o\Big[\big(Y_n,Z_{\infty}\big)^{\skrig}_{o}\geq D-\delta \Big] + \P_o\Big[\big(Z_n,Z_{\infty}\big)^{\skrig}_{o}< D\Big]\\
\leq \,& \P_o\big[d_{\skrig}(Y_n,o) \geq D+R-\delta\big] \e^{-D}+ \P_o\big[d_{\skrig}(Y_n,o) \leq D+R-\delta\big] \\
& +\Big(1-\P_o\big[d_{\skrig}(Y_n,o)\geq 2D\big]\Big) + c_l \e^{-D}\\
\leq \,& (c_l+1) \e^{-D}+\Big(1-\P_o\big[d_{\skrig}(Y_n,o)\geq 2D\big]\Big)
\end{align*}
by putting in the detailed upper bounds from equations and \eqref{deviation gp detailed}, and \eqref{gp on bd upper detailed}.

And for the lower bound, using the detailed lower bound in Equation \eqref{gp on bd lower detailed}, that
\begin{equation}
\P_o\big[\big(Y_n,Z_n\big)^{\skrig}_{o}\geq D\big]\geq \P_o\big[\big(Y_n,Z_{\infty}\big)^{\skrig}_{o}\geq D+\delta \big] \gtrsim \P_o\big[d_{\skrig}(Y_n,o)> D+R+\delta \big] \e^{-(D+c_2)} 
\end{equation}
And since $\P_o[d_{\skrig}(Y_n,0)\leq K] \to 0$ as $n\to \infty$ for any constant $K\geq 0$, the claim follows.
\end{proof}

\begin{cor}\label{non-sing}
Let $(Y_n)_{n\in\N}$ and $(Z_n)_{n\in\N}$ be two independent usual random walks on $G$ with step distribution $\mu$ and starting at $o$ and denote their hitting points on $\partial G$ by $Y_{\infty}$ and $Z_{\infty}$ respectively. Then
\[Y_{\infty} \ne Z_{\infty} \quad \text{almost surely.}\]
\end{cor}

\begin{remark}
Note that the non-singularity of the joint hitting distribution is a well known fact and can also be deduced from the non-singularity of the hitting measures. It is included here, since it follows nicely from the more detailed joint behavior of the two paths (and thus the two halves of the bi-infinite path) shown in Theorem \ref{deviation bds} and Lemma \ref{gp independent}.
\end{remark}

\begin{proof}
The claim will be proven by showing that almost surely $(Y_{\infty},Z_{\infty})^{\skrig}_{o}< \infty$. For this we first shift the Gromov product to a finite time in $Y_n$ to use the bounds we have shown above.
Using that by hyperbolicity
\[\big(y,Z_{\infty}\big)^{\skrig}_{o} \geq \min\bigg\{(Y_{\infty},Z_{\infty})^{\skrig}_{o}, \big(Y_n,Y_{\infty}\big)^{\skrig}_{o}\bigg\}-\delta.\]
as in the proof of the previous Lemma \ref{gp independent}, we again get that for $n$ big with high probability
\[(Y_{\infty},Z_{\infty})^{\skrig}_{o}\leq \big(Y_n,Z_{\infty}\big)_o^{\skrig}+\delta.\]
Hence, again
\begin{equation*}
\P_o\Big[\big(Y_{\infty},Z_{\infty}\big)^{\skrig}_{o}\geq D\Big] \leq \P_o\Big[\big(Y_n,Z_{\infty}\big)^{\skrig}_{o}\geq D-\delta \Big] + \P_o\Big[\big(Y_n,Y_{\infty}\big)^{\skrig}_{o}< D\Big]\lesssim \e^{-D}
\end{equation*}
by the bounds from Lemma \ref{gp independent} and Theorem \ref{deviation bds}.

Note, that we saw in the proofs of these statements, that technically these bounds depend on $\P_o[d_{\skrig}(o,Y_n)\geq K]$ for some constants $K>0$. But since we may choose $n\geq 0$ arbitrarily big, we can assume that these are all arbitrarily close to $1$.

Finally, due to the exponential decay of the probabilities, we get that
\[\sum_{n\in\N} \P_o\Big[\big(X_{\infty},X_{-\infty}\big)^{\skrig}_{o}\geq D_n\Big]< \infty\]
when choosing $D_{n}= O(\log(n))$. So by the Lemma of Borel and Cantelli, we get that
\[\P_o\Big[\big(X_{\infty},X_{-\infty}\big)^{\skrig}_{o}=\infty\Big]=0.\]

\end{proof}

\begin{cor}\label{almost geodesic}
For the bi-infinite random walk $(X_{z})_{z\in\Z}$ with hitting points $X_{-\infty}$ and $X_{\infty}$ on $\partial G$, we have that 
\[\limsup_{\abs{z}\to \infty} \frac{d_{\skrig}\big(X_{z},(X_{-\infty},X_{\infty})\big)}{\log(\abs{z})}< \infty\]
almost surely, where $(X_{\infty},X_{-\infty})$ is a quasi-geodesic ray between the two hitting points.
\end{cor}

\begin{proof}
In order to bound the deviation of $(X_{z})_{z\in\Z}$ from the quasi-geodesic $(X_{-\infty}, X_{\infty})$, we will show that it suffices to regard the deviation from by $[o,X_{\infty})$ and $[o,X_{-\infty})$, the quasi-geodesics from $o$ to $X_{\infty}$ and $X_{-\infty}$ respectively.
We consider the triangle $\Delta(X_{-\infty},o,X_{\infty})$. Since $(G,d_{\skrig})$ is $\delta$-hyperbolic, there exist points $m_1 \in [0,X_{-\infty})$ and $m_2 \in [o,X_{\infty})$ such that $d(x,(X_{-\infty},X_{\infty}))< 2\delta$ for $x \in[m_1,X_{-\infty})\cup [m_2,X_{\infty}]$. 
So we can approximate for $n\geq 0$
\[d_{\skrig}\big(X_{n},(X_{-\infty},X_{\infty})\big)\leq d-G\big(X_n,[o,X_{\infty})\big)+ d_{\skrig}(o,m_2)+\delta\]
and the same with $X_{-\infty}$ and $m_1$ for $n< 0$.
Now by the deviation bounds shown in Theorem \ref{deviation bds}, we can apply the Lemma of Borel and Cantelli to $\{d_{\skrig}(X_{n},[0,X_{-\infty}) > D_n)\}$ for $D_n = O(\log(n))$ to get that almost surely
\[\limsup_{n\to\infty} \frac{d\big(X_{n},[0,X_{-\infty})\big)}{\log(n)}< \infty\]
and the same for $n\to -\infty$. Hence we also have that almost surely
\[\limsup_{\abs{n}\to \infty} \frac{d\big(X_{n},(X_{-\infty},X_{\infty})\big)}{\log(\abs{n})}< \infty.\]
\end{proof}

\subsection{Deviation from the Geodesic: a Disintegrated Version}\label{deviation from geodesic disintegrated}

 Since $\nu_o$ is the push-forward of $\P_o$ under $\bd: G^{\Z}\to \partial G$, for every $\xi \in \partial G$ there exists a measure $\P_o(\cdot\vert \, \xi\, )$ such that 
\[
\P_o\Big(G^{\Z}\backslash \bd^{-1}(\xi) \,\, \big \vert \,\,  \xi\Big)= 0,
\]
and for all $f: G^{\Z}\to \R$ measurable
\[\int_{G^{\Z}}f(x)\d\P_o(x)= \int_{\partial G}  \int_{\bd^{-1}(\xi)}f(x)\d\P_o(x\vert \xi) \d \nu_o(\xi)\]
as in Theorem \ref{disintegration}.
Since the geometric arguments in the above proof of Theorem \ref{deviation bds} do not really depend where the limit point $X_{\infty}$ lies on the boundary $\partial G$, it is possible to transfer the above statement to the conditioned measures $\P_o(\cdot \vert \xi)$ as well.

\begin{theorem}\label{disintegrated deviation bds}
Let $(Y_n)_{n\in\N}$ be a random walk started at $o$ with step distribution $\mu$. For any $D\geq 0$ and $n\geq 0$ sufficiently large, there exists $C_D>1$ such that for all $\xi\in \partial G$
\[ C_D^{-1} \e^{-D} \leq \P_{o}\big[d_{\skrig}\big(Y_n, [o,\xi)\big)\geq D\big \vert \xi \big]\leq C_D \e^{-D}.\]
The constants $C_D$ are bounded in $D$ and do not depend on $\xi$. 
\end{theorem}

\begin{proof}
We will prove the claim by showing that for any set $B$ in a generating set of the $\sigma$-field on $ \partial G$ we have 
\[ \P_{o}\big[d_{\skrig}\big(Y_n, [o,Y_{\infty})\big)\geq D, Y_{\infty}\in B]\asymp  \P_{o}\big[d_{\skrig}\big(Y_n, [o,\xi)\big)\geq D]\P_o[Y_{\infty} \in B].\]
This implies the statement, since if there existed a set $B_D\subset \partial G$ of positive measure on which the deviation bound did not hold, then applying the disintegration equality for $f= \mathbbm{1}_{B_D}$ would lead to a contradiction.

As in the above proof of Theorem \ref{deviation bds}, we may rewrite
\begin{equation}\label{ind distance and hitting point}
\P_{o}\big[d_{\skrig}\big(Y_n, [o,Y_{\infty})\big)< D, Y_{\infty}\in B\big] = \sum_{x \in G}\P_{x}\big[d_{\skrig}\big(x, [o,Y_{\infty})\big)< D, Y_{\infty}\in B\big]\P_{o}[Y_n = x]
\end{equation}
Note that for technical reasons it will be nicer, to regard the case "$<D$" rather that $\geq D$ in the probability above. Further, we have shown that there exist $p_l,p_u\in G$ and $R_{l},R_{u}>0$ such that 
\[\mho_{x}(p_l,R_l)\subset \{d_{\skrig}\big(x, [o,Y_{\infty})\big)\geq D, Y_{\infty}\} \subset \mho_x(p_u,R_u)\]
So to get an upper bound on \eqref{ind distance and hitting point}, we need to  consider $x \in G$ for which 
\begin{equation}\label{upper bd intersection shadows}
B  \cap \mho_{x}(p_l,R_l)^c \ne \emptyset
\end{equation}
and for a lower bound those $x \in G$ for which
\begin{equation}\label{lower bd intersection shadows}
\mho_{x}(p_u,R_u)^c \subset B .
\end{equation}

We will first study what happens in the hyperbolic space $\H^d$. We then use the rough isometry $\psi: G \to \H^d$ from the Bonk-Schramm Theorem \ref{thm-bs} to transfer between $\H^d$ and $G$. Note that there exist $\lambda, c$, such that for all $x,y \in G$
\[\lambda d_{\skrig}(x,y)\in [d_H(\psi(x),\psi(y))-c, d_H(\psi(x),\psi(y))+c].\]
In particular, this also means that the Gromov product $\lambda (x,y)_z^{\skrig}$ is bounded above and below by $(\psi(x),\psi(y))_{\psi(z)}\pm 3c$, i.e.\ the image of a shadow $\mho_x(p,R)$ is contained in the shadow $\mho_{\psi(x)}(\psi(p), \tilde R)$ for $\tilde R= \lambda^{-1}(R +4c)$ and contains this shadow for $\tilde R =\lambda^{-1} (R-4c)$. Hence we look at the analogous statements to equations \eqref{upper bd intersection shadows} and \eqref{lower bd intersection shadows} in $\H^d$.

So, let $\mho_{x}^{\H}(p,R)$ and $\mho_{0}^{\H}(a,R')$ be two shadows in $\H^d$, where $a\in G$ is arbitrary and $p$ is assumed to be on the geodesic $[0,x]$ and $d_{\H}(p,x)> R$. 
Then, using the fact that in $\H^d$ the Gromov product is given as
\[(\xi,\eta)_p = \log(\sin(\frac{\theta}{2})^{-1})\]
where $\theta$ is the angle between the geodesics $[p,\xi)$ and $[p,\eta)$, we get that 
\[\mho_{0}^{\H}(a,R')= \{\xi \in \partial G:\, \angle([0,a),[0,\xi))\leq \alpha\}\]
for some $\alpha>0$ dependent on $R'$ and $d_{\H}(0,a)$. In the same way
\[\mho_{x}^{\H}(p,R)^c= \{\xi \in \partial G:\, \angle([0,a),[0,\xi))\geq \gamma\}\]
for some $\beta >0$ dependent on $R$ and $d_{\H}(x,p)$. 

Looking at the geometry of $\H^d$ in the ball model then shows us, that for any $\alpha_l< \alpha$ there exists $d_{\alpha_l}>0$ such that for $x$ with $d_{\H}(x,a)> d_{\alpha_l}$ and such that the angle between $[0,x]$ and $[0,a]$ is less than $\alpha_l$, $\mho_{x}^{\H}(p,R)^c\subset \mho_{0}^{\H}(a,R')$. And similarly, 
for any $\alpha_u> \alpha$ there exists $d_{\alpha_u}>0$ such that for $x$ with $d_{\H}(x,a)< d_{\alpha_l}$ and such that the angle between $[0,x]$ and $[0,a]$ is less than $\alpha_u$, $\mho_{x}^{\H}(p,R)^c\cap \mho_{0}^{\H}(a,R')\ne \emptyset$.

Now, to transfer these observations to $G$, let $\psi:G \to \H^d$ be the $(\lambda,c)$-rough isometry from the Bonk-Schramm embedding. Since $\psi$ is injective and the shadows $\mho_{0}(a,R')$ generate the $\sigma$-field on $\partial \H^d$, sets of the form $\psi^{-1}(\mho_o(a,R'))$ form a generating set of the $\sigma$-algebra on $\partial G$. So, we set $B = \psi^{-1}(\mho_{0}(a,R'))$. Further, for $0<\delta< \alpha$, set 
\[H_{\delta}=\psi^{-1}(\{x \in \H^{d}: \angle([0,a),[0,x))< \alpha-\delta\})\]
and for $\delta>0$
\[H^{\delta}=\psi^{-1}(\{x \in \H^{d}: \angle([0,a),[0,x))< \alpha+\delta\}).\]
Then, using that $G \cup \partial G$ equipped with the visual metric is a topological metric space, we have that
\begin{align*}
\nu_o(B) &= \P_o[\lim_{n\to \infty}Y_n \in B ]= \lim_{\varepsilon \to 0}\lim_{m\to \infty}\P_o[d(X_m,B)\leq \varepsilon]\\
&= \lim_{\varepsilon \to 0}\lim_{m\to \infty}\P_o\big[\exists \xi \in B: (X_m,\xi)_o^{\skrig}\geq \varepsilon^{-1}\big] = \lim_{\delta \to 0}\lim_{m\to \infty}\P_o\big[\exists \eta  \in\mho_0(a,R'): (\psi(X_m),\eta)_0^{\H}\geq \delta^{-1}\big]
\end{align*}
Note that when changing from the Gromov product in $G$ to the one in $\H$, we get upper and lower bounds for the probability with $\delta^{-1}= \lambda^{-1}\varepsilon^{-1}\pm3c$. However, since the limit does not depend on the way  $\varepsilon \to 0$, the specific choice of $\delta$ does not matter either.
Now, by the above argument, for fixed $\delta >0$ we have
\[\{x_m:\,\exists \eta  \in\mho_0(a,R'): (\psi(X_m),\eta)_0^{\H}\geq \delta^{-1}, d_{\H}(x_m,0)> d_{\delta_u}\}\subset \{x \in \H^{d}: \angle([0,a),[0,x))< \alpha+\delta_u\}\]
where $\delta_u$ is the choice of angle corresponding to the Gromov product being $\delta^{-1}$. And in the same way
\[\{x \in \H^{d}: \angle([0,a),[0,x))< \alpha-\delta_l\}\subset \{x_m:\exists \eta  \in\mho_0(a,R'): (\psi(X_m),\eta)_0^{\H} \geq \delta^{-1}\}\]
where we choose $\delta_l<0<\alpha$ such that $\mho_{\psi(x)}^{\H}(\psi(p),\tilde R)^c\subset \mho^{\H}_0(a,R')$ for $x\in G$ such that $d_{\H}(\psi(x),0)\geq \delta^{-1}$.

Then, for fixed $\delta>0$, using also that $\P_o[d_G(X_m,o)\leq d_{\delta}]\to 0$ as $m \to \infty$, we have
\[\lim_{m\to \infty} \P_o\big[\exists \eta  \in\mho_0(a,R'): (\psi(X_m),\eta)_0^{\H}\geq \delta^{-1}\big] \leq\lim_{m\to \infty}\P_o\big[ X_m \in H^{\delta}\big] \]
and
\[\lim_{m\to \infty} \P_o\big[\exists \eta  \in\mho_0(a,R'): (\psi(X_m),\eta)_0^{\H}\geq \delta^{-1}\big] \geq\lim_{m\to \infty}\P_o\big[ X_m \in H_{\delta}\big].\]
Further, as $\delta \to 0$, 
\[\P_o\big[ X_m \in H^{\delta}\big] -\P_o\big[ X_m \in H_{\delta}\big] \to 0\]
so the inequalities are actually equalities.
Finally, 
\[\{x \in H_{\delta}: d_{\H}(\psi(x),0)\geq d_{\delta}\} \subset \{x: \mho_x(p,R)^c \subset B, d_{\H}(\psi(x),0)\geq d_{\delta}\}\]
and
\[\{x: \mho_x(p,R)^c \cap B \ne \emptyset\} \subset H^{\delta}.\]
and hence
\[\sum_{x \in G}\P_{x}\big[d_{\skrig}\big(x, [o,Y_{\infty})\big)< D, Y_{\infty}\in B]\P_{o}[Y_n = x] \leq (1-c_l \e^{-D}) \nu_0(B) \]
for some $c_l> 0$ coming from the Shadow Lemma \ref{shadow lemma}. In the same way we get the lower bound by exchanging $c_l$ for $c_u>0$ using the corresponding upper bound in the Shadow Lemma.
This in particular also implies that 
\[\P_{o}\big[d_{\skrig}\big(Y_n, [o,\xi)\big)\geq D\vert \xi]\asymp \e^{-D}.\]
\end{proof}

\section{Dynamics of the Hitting Measures}\label{sec Theta constructions}

The goal of this section is to provide the first two constructions of the measure $\Theta$ appearing in Theorem \ref{theorem 1}: in Section \ref{sec Theta bd} and Section \ref{sec Theta disintegration} we will show that the measures $\Theta_1$ and $\Theta_2$ come from the {\it dynamics} of the hitting measures along deterministic and random paths respectively.

\subsection{Construction of $\Theta_1$: Dynamics Along Deterministic Paths}\label{sec Theta bd}

\noindent The construction of $\Theta_1$ will depend on a choice of a family of bi-infinite paths connecting each pair of points on the boundary. These paths are constructed as follows:

Recall that 
$$
\partial G \simeq \bigg\{[(x_n)_{n\in\N}]: (x_n)_{n\in\N} \text{ geodesic ray started at }o\bigg\}. 
$$
Hence, for $[\xi],[\eta]\in \partial G$ we may choose such rays $(\xi_n)_{n\in\N}$ and $(\eta_n)_{n\in\N}$. Then $(p_z)_{z\in\Z}:=\big((\xi_n)_{n\in\N}\oplus (\eta_n)_{n\in\N}\big)$ is a bi-infinite ray connecting $\xi$ and $\eta$.

So $\partial^2 G\simeq \{[(\xi \oplus \eta)]: \xi\ne\eta \in \partial G\}$. Using these rays, we now define the measure $\Theta_1$ (recall \eqref{bd meas}) by setting 
\begin{equation}\label{definition Theta1}
\d\Theta_1(\xi \oplus \eta):=\lim_{n\to\infty} \sum_{z=-n}^n \frac{1}{2n+1} \d(\nu_{p_z}\otimes \nu_{p_z})(\xi,\eta)
\end{equation}
for $\xi\ne \eta$. This way, $\Theta_1$ captures how the hitting measures evolve as we move the origin along the paths $(p_z)_{z\in\Z}$.

The chosen paths $(p_z)_{z\in\Z}$ come with a parametrization, which always sets the origin $o$ at time $0$. $\Theta_1$ however is independent of the choice of parametrization. To show this, we consider the action of $\Z$ on the bi-infinite sequences $x=(x_n)_{n\in \Z}$ given by $ z.(x)(n)=x_{n-z}$. Since $x_n$ and $x_{n-z}$ always have at most distance $z$, this action does not change the limits of the sequence as $n\to \infty$ and $n\to -\infty$.

\begin{lemma}
The measure $\Theta_1$ defined in \eqref{definition Theta1} is independent of the parametrization of the rays $(\xi \oplus \eta)$ we chose.
\end{lemma}

\begin{proof}
For $z\in \Z$ we set
\[\d\Theta_1^{z}(\xi \oplus \eta):= \lim_{n\to\infty}\sum_{i=-n}^n \frac{1}{2n+1} \;\d(\nu_{p_{i+z}}\otimes \nu_{p_{z+i}})(\xi_{z+i},\eta_{z+i}).\]
Then 
\begin{align*}
\d\Theta_1^{z}(\xi\oplus\eta) &= \lim_{n\to\infty} \sum_{i=-n}^n \frac{1}{2n+1} \;\d (\nu_{p_{z-i}}\otimes\nu_{p_{z-i}})(\xi, \eta)\\
&= \lim_{n\to\infty}\sum_{i=-n-z}^{n-z}\frac{1}{2n+1} \;\d (\nu_{p_{i}}\otimes\nu_{p_{i}})(\xi_{i}, \eta_{i})\\
&= \lim_{n\to\infty}\sum_{i=-n}^n\frac{1}{2n+1} \;\d (\nu_{p_{i}}\otimes\nu_{p_{i}})(\xi_{i}, \eta_{i})\\
&= \d \Theta_1(\xi\oplus\eta).
\end{align*}
since the $\limsup$ does not change when changing boundedly many summands in each step.
\end{proof}

We now turn to the required estimates on the density of $\Theta_1$ with respect to $(\nu_o\otimes\nu_o)$ (shown in Lemma \ref{density Theta1}) as well as proving the quasi-invariance of $\Theta_1$ under the $G$-action on $\partial^2 G$ (shown in Lemma \ref{qi Theta1}).

\begin{lemma}[Density of $\Theta_1$]\label{density Theta1}
For all $\xi,\eta\in \partial G$, we have 
\[
\e^{-(2 M + 8\delta)}\, \e^{2(\xi,\eta)^{\skrig}_{o}}\leq 
\frac{\d \Theta_1}{\d(\nu_0\otimes\nu_0)}(\xi,\eta) \leq \e^{2\delta}\, \e^{2(\xi,\eta)^{\skrig}_{o}}
\]
where $M$ is the constant such that $(G,d_{\skrig})$ is $M$ quasi-ruled in the sense of Remark~\ref{rmk-qr}. In particular this also means that the limit in the definition of $\Theta_1$ exists, i.e. that $\Theta_1$ is well defined.

\end{lemma}

\begin{proof}

We know from Theorem \ref{density hitting measures}, that the density of the hitting measures for the random walk started at $x$ and at $y$ is given by 
\[\frac{\d \nu_x}{\d \nu_o} (\xi) = \exp{\limsup_{\xi_n \to \xi} d_{\skrig}(o,\xi_n)-d_{\skrig}(x,\xi_n)}.\]
And since for the product measures the density is 
\[\frac{\d (\nu_x\otimes\nu_x)}{\d (\nu_o\otimes \nu_o)}(\xi,\eta)=\frac{\d \nu_x}{\d \nu_y}(\xi)\frac{\d \nu_x}{\d \nu_y}(\eta),\]
we have that
\[
\frac{\d(\nu_{p_z}\otimes \nu_{p_z})}{ \d(\nu_o\otimes \nu_o)}(\xi,\eta) = \exp\bigg\{\limsup\limits_{\substack{\xi_i \to\xi \\ \eta_i\to\eta}}d_{\skrig}(o,\xi_i)+d_{\skrig}(o,\eta_i)-d_{\skrig}(p_z,\xi_i)-d_{\skrig}(p_z,\eta_i)\bigg\}
\]
Now, if we can show that for big $i\in\N$ and $z\in\Z$, the value of $d_{\skrig}(p_z,\xi_i)+d_{\skrig}(p_z,\eta_i)$ is within bounded distance of $d_{\skrig}(\xi_i,\eta_i)$ for a bound independent of $\xi$ and $\eta$, the claim will follow.

\noindent Let $\abs{z} \geq (\xi,\eta)^{\skrig}_{o}+ \varepsilon$. Then in particular also $\abs{z}\geq (\xi_i,\eta_i)^{\skrig}_{o}$ for all $i$ big enough. 
Consider the triangle $\Delta(0,\xi_i,\eta_i)$. By hyperbolicity of $G$ there exists a point $\tilde p_z\in[\xi_i,\eta_i]$ such that $d_{\skrig}(p_z,\tilde p_z)\leq 4\delta$.

And since $(G,d_{\skrig})$ is $M$ quasi-ruled, we have 
\[
d_{\skrig}(\xi_i,\eta_i)\leq d_{\skrig}(p_z,\xi_i)+d_{\skrig}(p_z,\eta_i) \leq d_{\skrig}(\tilde p_z,\xi_i)+d_{\skrig}(\tilde p_z,\eta_i)+ 8\delta \leq d_{\skrig}(\xi_i,\eta_i)+ 2M + 8\delta.
\]
So, all together for $c = 2M + 8\delta$ and $m=(\xi,\eta)^{\skrig}_{o}+ \varepsilon$
\begin{align*}
\frac{\d\Theta_1}{ \d(\nu_o\otimes \nu_o)}(\xi, \eta) 
 &= \lim_{n \to \infty}\sum_{m\leq\abs{z}\leq n} \frac{1}{2n+1}\,\exp\bigg\{\limsup_{\substack{\xi_i \to\xi \\ \eta_i\to\eta}}d_{\skrig}(o,\xi_i)+d_{\skrig}(o,\eta_i)-d_{\skrig}(p_z,\xi_i)-d_{\skrig}(p_z,\eta_i)\bigg\}\\
& \geq \lim_{n \to \infty} \sum_{m\leq\abs{z}\leq n} \frac{1}{2n+1}\, \exp\bigg\{\limsup_{\substack{\xi_i \to\xi \\ \eta_i\to\eta}}d_{\skrig}(o,\xi_i)+d_{\skrig}(o,\eta_i)-d_{\skrig}(\xi_i,\eta_i) - c\bigg\}\\
 &\geq \e^{-c} \, \e^{2(\xi,\eta)^{\skrig}_{o}}
\end{align*}
where we used that ignoring the summands for $\abs{z}<m$ does not change the limit.

\noindent For the lower bound we use the triangle inequality to get rid of $p_z$. This leaves us with
\[\limsup_{\substack{\xi_i \to\xi \\ \eta_i\to\eta}}d_{\skrig}(o,\xi_i)+d_{\skrig}(o,\eta_i)-d_{\skrig}(\xi_i,\eta_i)\leq (\xi,\eta)^{\skrig}_{o} + 2\delta,\]
where we can switch from the $\limsup$ to the $\liminf$ in the Gromov product by hyperbolicity. All together this means
\begin{align*}
\frac{\d\Theta_1}{ \d(\nu_o\otimes \nu_o)}(\xi, \eta) 
 &= \lim_{n \to \infty}\sum_{z=-n}^n \frac{1}{2n+1}\,\exp\bigg\{\limsup_{\substack{\xi_i \to\xi \\ \eta_i\to\eta}}d_{\skrig}(o,\xi_i)+d_{\skrig}(o,\eta_i)-d_{\skrig}(p_z,\xi_i)-d_{\skrig}(p_z,\eta_i)\bigg\}\\
 &\leq \lim_{n \to \infty} \sum_{z=-n}^n \frac{1}{2n+1}\, \exp\bigg\{\limsup_{\substack{\xi_i \to\xi \\ \eta_i\to\eta}}d_{\skrig}(o,\xi_i)+d_{\skrig}(o,\eta_i)-d_{\skrig}(\xi_i,\eta_i) \bigg\}\\
 &\leq \e^{2\delta} \, \e^{2(\xi,\eta)^{\skrig}_{o}}.
\end{align*}
Note that the upper and lower bounds above, show that the Césaro limit is taken over terms which are bounded above and below, and hence the limit exists.

\end{proof}

\begin{lemma}[Quasi-invariance of $\Theta_1$]\label{qi Theta1}
For any $g\in G$, denoting $g_\star\Theta_1$ the push-forward of $\Theta_1$, we have 
\[
\e^{-16\delta -2M}\leq \frac{\d (g_{\ast}\Theta_1)}{\d \Theta_1}(\xi, \eta)\leq \e^{16\delta +2M}.
\]
where, $M$ is the constant for which $(G,d_{\skrig})$ is $M$ quasi-ruled. In particular, the measure $\Theta_1$ is quasi-invariant under the action of $G$ on $\partial^2 G$. 
\end{lemma}

\begin{proof}
Denote by $(p_z)_{z\in\Z}$ the chosen bi-infinite ray connecting $\xi$ and $\eta$, and by $(p^{g}_z)_{z\in\Z}$ the ray between $g\xi$ and $g\eta$. By choice of these rays, $(g p_z)_{z\leq 0}$ and $(g p_z)_{z\geq 0}$ are quasi-geodesics converging to $g\xi$ and $g \eta$ respectively, so 
\[\Delta(0,g,g \xi) = (p^{g}_z)_{z\leq 0} \,\, \cup\,\, [0,g] \,\,\cup \,\, (g p_z)_{z\leq 0} \]
\[\Delta(0,g,g \eta) = (p^{g}_z)_{z\geq 0} \,\, \cup\,\, [0,g]\,\, \cup\,\, (g p_z)_{z\geq 0}\]
are quasi-geodesic triangles. In these triangles, the two sides going to $\xi$ and $\eta$ are within distance $4\delta$ of each other for $z$ such that $\d_{\skrig}(p_z,o)\geq (\xi,\eta)_o^{\skrig}$. Hence there exists $m>0$ such that for any point $p\in(g p_z)_{z\geq m}$ there exists $\tilde p \in(p^{g}_z)_{z\geq m}$ for which $d_{\skrig}(p,\tilde p)\leq 4 \delta$.

So, using that $G,d_{\skrig}$ is $M$ quasi-ruled as in the proof of Lemma \ref{density Theta1}, we have
\[\Big|d_{\skrig}(g p_z,\xi_i)-d_{\skrig}(g p_z,\eta_i) - d_{\skrig}(\tilde p^{g}_z,\xi_i)-d_{\skrig}(\tilde p^{g}_z,\eta_i)\Big|\leq 2M + 8\delta.\]
Hence, for $\abs{z}\geq m$ and $c= 16\delta +2M$
\[d_{\skrig}(g p_z,\xi_i)-d_{\skrig}(g p_z,\eta_i)- c \leq d_{\skrig}( p^{g}_z,\xi_i)-d_{\skrig}(p^{g}_z,\eta_i)\leq d_{\skrig}(g p_z,\xi_i)-d_{\skrig}(g p_z,\eta_i) + c.\]

With this we get 
\begin{align*}
\frac{\d \Theta_1}{\d(\nu_o\otimes \nu_o)}(g \xi,g \eta) 
&= \lim_{n\to\infty} \frac{1}{2n+1} \sum_{z=-n}^n \frac{\d(\nu_{p^{g}_z}\otimes \nu_{p^{g}_z})}{\d(\nu_o\otimes \nu_o)}(g \xi,g \eta) \\
&=\lim_{n\to\infty} \frac{1}{2n+1} \sum_{\abs{g}\leq \abs{z}\leq n} \exp\bigg\{\limsup_{\substack{\xi_i \to g\xi\\ \eta_i\to g\eta}}d_{\skrig}(o,\xi_i)+d_{\skrig}(o,\eta_i)-d_{\skrig}(p^{g}_z,\xi_i)-d_{\skrig}(p^{g}_z,\eta_i)\bigg\}\\
& \leq \lim_{n\to\infty} \frac{\e^{c}}{2n+1} \sum_{\abs{g }\leq \abs{z}\leq n} \exp\bigg\{\limsup_{\substack{\xi_i \to g\xi\\ \eta_i\to g\eta}} d_{\skrig}(o,\xi_i)+d_{\skrig}(o,\eta_i)-d_{\skrig}(g p_z,\xi_i)-d_{\skrig}(g {p}_z,\eta_i)\bigg\}\\
&=\e^{ c} \lim_{n\to\infty} \frac{1}{2n+1} \sum_{\abs{z}\leq n-\abs{g}} \frac{\d(\nu_{g p_z}\otimes \nu_{g p_z})}{\d(\nu_o\otimes \nu_o)}(g \xi,g \eta) \\
&= \e^{c} \lim_{n\to\infty} \frac{1}{2n+1} \sum_{\abs{z}\leq n}\frac{ \d(\nu_{p_z}\otimes \nu_{p_z})}{\d(\nu_o\otimes \nu_o)}(\xi, \eta) = \e^{c}\, \frac{\d \Theta_1}{\d(\nu_o\otimes \nu_o)}.
\end{align*}
using that $\d\nu_{gx}(g\xi)=\d\nu_{x}(\xi)$. Using the lower bounds on $d_{\skrig}( p^{g}_z,\xi_i)-d_{\skrig}(p^{g}_z,\eta_i)$ instead of the upper bounds, we get in the same way that the density is bounded below by $\e^{-c}$.
\end{proof}

\subsection{Construction of $\Theta_2$: Dynamics Along Random Paths}\label{sec Theta disintegration}

\noindent The measure $\Theta_1$ constructed Section \ref{sec Theta bd} depends on a specific choice of bi-infinite rays. In order to remove this in some ways arbitrary choice, we will now construct the measure $\Theta_2$ mentioned in Theorem \ref{theorem 1} in a similar way, where instead of relying on this specific choice of geodesic ray, we integrate over all bi-infinite rays passing through $o$ and converging to the boundary points $\xi, \eta$ to determine the density $\d\Theta_2(\xi,\eta)$. Thus, $\Theta_2$ captures the dynamics of the hitting measures when moved along all random walk paths with origin $X_0=o$. In this way it can be seen as a disintegrated version of $\Theta_1$.

\noindent For notational convenience, we define the function $f: G^{\Z}\to \R$ to be 
\begin{equation}\label{def function f in Theta2}
f\Big((x_n)_{n\in\Z}\Big)= \lim_{n\to\infty} \sum_{z=-n}^n \frac{1}{2n+1} \frac{\d(\nu_{x_z}\otimes \nu_{x_z})}{\d(\nu_{x_0}\otimes \nu_{x_0})}(\xi,\eta)
\end{equation}
where $\xi=\bd\big((x_n)_{n\in\N})\big)$ and $\eta=\bd\big((x_{-n})_{n\in\N})\big)$. We also recall from Definition \ref{def bi-inf RW} that $\bar\P_g$ stands for the distribution of bi-infinite random walk trajectory $(X_n)_{n\in\Z}$. Then for any $A \subset \partial^{2}G$, we set (recall \eqref{Theta disintegration}),
\begin{equation}\label{definition Theta2}
\Theta_2(A) := \int_{\bd^{-1}(A)}f\big((x_n)_{n\in\Z}\big)\;\d\bar\P_o\Big((x_n)_{n\in\Z}\Big).\end{equation}

We will now show the required estimates on the density of $\Theta_2$ with respect to $\nu_o\otimes\nu_o$:

\begin{lemma}[Density of $\Theta_2$]\label{density Theta2}
For every $\xi,\eta \in \partial G$, we have 
\[\frac{\d \Theta_2 }{\d (\nu_o\otimes \nu_o)}(\xi, \eta)\asymp \e^{2(\xi,\eta)_o^{\skrig}}.\]
In particular, this also means the limit in the definition of $f$ in \eqref{def function f in Theta2} exists, and $\Theta_2$ is well defined.
\end{lemma}

\begin{proof}
Since $(\nu_o\otimes\nu_o)$ is the push-forward of $\bar\P_{o}$ under the boundary map, by disintegration of measure, there exists a family $\bar\P(\cdot\vert(\xi,\eta))$ such that for every measurable function $g:G^{\Z}\to \R$ we have
\[\int_{G_{0}^{\Z}}g(x_n)_{n}\d\bar\P_{o}(x_n)_n = \int_{\partial^2G}\int_{\bd^{-1}(\xi)\cross \bd^{-1}(\eta)}g(x_n)_{n}\d\bar\P_o\big((x_n)_{n\in\Z}\vert \xi,\eta\big)\d(\nu_o\otimes\nu_o)(\xi,\eta).\]
In particular this holds for $g(\cdot)= \mathbbm{1}_{\bd^{-1}(A)}(\cdot)f(\cdot)$ with $f$ defined in \eqref{def function f in Theta2} and any $A\in \partial^2 G$. To get the bound on the density, we hence have to estimate 
\[\int_{\bd^{-1}(\xi)\cross \bd^{-1}(\eta)}f\big((x_n)_{n\in \Z}\big)\;\d\bar\P_o\big((x_n)_{n\in\Z}\vert (\xi,\eta)\big)\]
By the triangle inequality 
\[d_{\skrig}(\xi_i,\eta_i)\leq d_{\skrig}(x_z,\xi_i)-d_{\skrig}(x_z,\eta_i)\]
and hence, remembering that we only consider paths with $x_0=o$, we get that 
$$
f((x_n)_{n\in\Z}) \leq \e^{2(\xi,\eta)_o^{\skrig}}, 
$$
implying the desired upper bound of the density.

To get the lower bound, we first notice that by Fatou's lemma, we can exchange the limit and the integral to get a lower bound. Afterwards we may exchange the integral and sum, so we get
\begin{align*}
&\int_{\bd^{-1}(\xi)\cross \bd^{-1}(\eta)}f(x_n)_{n}\d\bar\P_o\big((x_n)_{n\in\Z}\vert (\xi,\eta)\big) \\
&\geq \lim_{n\to\infty} \sum \int \frac{1}{2n+1} \exp\Big\{\limsup_{\substack{\xi_i \to\xi \\ \eta_i\to\eta}}d_{\skrig}(x_0,\xi_i)+d_{\skrig}(x_0,\eta_i)-d_{\skrig}(x_z,\xi_i)-d_{\skrig}(x_z,\eta_i)\Big\} \d\bar\P_o\big((x_n)_{n\in\Z}\vert \xi,\eta\big)
\end{align*}
where the above sum is taken over $z \in \{-n,...,n\}$ and the integral over $(\bd^{-1}(\xi))\times (\bd^{-1}(\eta))$.

Now we will use what we proved in Theorem \ref{disintegrated deviation bds} about the deviation of the random walk paths $(x_n)_{n\in\Z}$ from the geodesic for the disintegrated measures. Note that, since $\bar\P_o$ is equivalent to the product $\P_o\otimes\P_o$ of two independent measures, under splitting the paths at time $0$ and reversing time on the negative time half, the same also holds for the disintegrated measures $\bar\P_o(\cdot\vert (\xi,\eta))$.

So, let $D>0$ be fixed and $x_z$ such that $d_{\skrig}(x_z,(\xi,\eta))\leq D$. Further let $x\in [\xi_i,\eta_i]$ be such that $d_{\skrig}(x_z,x)\leq D$. Then
\[d_{\skrig}(x_z,\xi_i)+d_{\skrig}(x_z,\xi_i) \leq d_{\skrig}(x,\xi_i)+d_{\skrig}(x,\xi_i)+ 2D \leq d_{\skrig}(\xi_i,\eta_i) + 2D+M\]
using that $G$ is $M$ quasi-ruled for $d_{\skrig}$. Hence
\begin{align*}
&\int_{\bd^{-1}(\xi)\cross \bd^{-1}(\eta)} \frac{\d(\nu_{x_z}\otimes \nu_{x_z})}{\d(\nu_{x_0}\otimes \nu_{x_0})}(\xi,\eta) \d\bar\P_o\big((x_n)_{n\in\Z}\vert (\xi,\eta)\big)\\
&\geq \int_{\bd^{-1}(\xi)\cross \bd^{-1}(\eta)} \exp\bigg\{\limsup_{\substack{\xi_i \to\xi \\ \eta_i\to\eta}}d_{\skrig}(x_0,\xi_i)+d_{\skrig}(x_0,\eta_i)-d_{\skrig}(x_z,\xi_i)-d_{\skrig}(x_z,\eta_i)\bigg\}\d\bar\P_o\big((x_n)_{n\in\Z}\vert (\xi,\eta)\big)\\
&\geq \int_{B} \exp\bigg\{\limsup_{\substack{\xi_i \to\xi \\ \eta_i\to\eta}}d_{\skrig}(x_0,\xi_i)+d_{\skrig}(x_0,\eta_i)-d_{\skrig}(\xi_i\eta_i)- 2D-M\bigg\} \d\bar\P_o\big((x_n)_{n\in\Z}\vert (\xi,\eta)\big)\\
& \geq (1-\e^{-D}) \,\e^{-2D-M}\, \e^{2(\xi,\eta)_0}
\end{align*}
where we have restricted the integral to 
\[B= \bd^{-1}(\xi)\cross \bd^{-1}(\eta)\cap \big\{d_{\skrig}\big(x_z,(\xi,\eta)\big)\leq D\big\}.\]
Optimizing over $D$, we get that the maximal lower bound is achieved at $D=\log(\frac{3}{2})$. Here the value is $e^{-M}\,\frac{4}{27}$.
By the upper and lower bound again the Césaro limit is taken over bounded terms and hence exists. Note that, since the claim is that the limit exists for $f$, in the lower bound we also need to look at the terms where $x_n$ is further than $D$ from the geodesic $(\xi,\eta)$. But since this distance turns up in $f$ as a term of order $\e^{-d_{\skrig}(x_n,(\xi,\eta))}$, we still have that all terms are bounded below by $0$.  
\end{proof}

\begin{lemma}\label{qi Theta2}
$\Theta_2$ is quasi-invariant for the $G$-action.
\end{lemma}

\begin{proof}
Let $A\subset \partial^2 G$ and $g\in G$. Then $\bd^{-1}(g A)=g \, \bd^{-1}(A)$ and hence

\begin{align}
\Theta_2(gA) &= \int_{g \bd^{-1}(A)}f\big((x_n)_{n\in\Z}\big)\;\d\bar\P_o\big((x_n)_{n\in\Z}\big)\nonumber\\
&=\int_{\bd^{-1}(A)}f\big((g^{-1}x_n)_{n\in\Z}\big)\; \d\bar\P_o\big((x_n)_{n\in\Z}\big)\label{Theta2QI}\\
&=\int_{\bd^{-1}(A)}f\big((x_n)_{n}\big)\;\d\bar\P_g\big((x_n)_{n\in\Z}\big)\nonumber\\
&=\int_{\bd^{-1}(A)}\int_{\bd^{-1}(\xi)\cross \bd^{-1}(\eta)}f\big((x_n)_{n\in\Z}\big)\;\d\bar\P_g\big((x_n)_{n\in\Z}\vert(\xi,\eta)\big)\;\d(\nu_g\otimes\nu_g)(\xi,\eta)\nonumber\\
&\asymp\int_{\bd^{-1}(A)} \frac{\d(\nu_g\otimes\nu_g)}{\d(\nu_o\otimes\nu_o)}(\xi,\eta)\; \e^{2(\xi,\eta)^{\skrig}_g} \;\d(\nu_o\otimes\nu_o)(\xi,\eta)\nonumber
\end{align}
\noindent
where in \eqref{Theta2QI} we used, that $f$ is invariant under the $G$-shift, and that we can do the same disintegration for $\bar\P_g$ as well. 

Finally, since 
\[\frac{\d(\nu_g\otimes\nu_g)}{\d(\nu_o\otimes\nu_o)}(\xi,\eta) = \exp\bigg\{\limsup_{\substack{\xi_i\to\xi\\ \eta_i \to \eta}}d_{\skrig}(\xi_i,o)+d_{\skrig}(\eta_i,o)-d_{\skrig}(\xi_i,g)-d_{\skrig}(\eta_i,g)\bigg\}\]
it is left to show that there exist $c_1,c_2>0$ such that
\[ 2(\xi,\eta)_g^{\skrig} + \limsup_{\substack{\xi_i\to\xi\\ \eta_i \to \eta}}d_{\skrig}(\xi_i,o)+d_{\skrig}(\eta_i,o)-d_{\skrig}(\xi_i,g)-d_{\skrig}(\eta_i,g) \begin{cases}\geq(\xi,\eta)_o^{\skrig}-c_1\\ \leq (\xi,\eta)_o^{\skrig}+c_2\end{cases}\]
for all $\xi,\eta\in \partial G$ to prove the quasi-invariance. We have
\begin{align*}
2(\xi,\eta)^{\skrig}_g&= \limsup_{\substack{\xi_i \to\xi \\ \eta_i\to\eta}}d_{\skrig}(g,\xi_i)+d_{\skrig}(g,\eta_i)-d_{\skrig}(\xi_i\eta_i)\\
&=\limsup_{\substack{\xi_i \to\xi \\ \eta_i\to\eta}}d_{\skrig}(o,\xi_i)+d_{\skrig}(o,\eta_i)-d_{\skrig}(\xi_i\eta_i) + d_{\skrig}(o,\xi_i)+d_{\skrig}(o,\eta_i)- d_{\skrig}(g,\xi_i)+d_{\skrig}(g,\eta_i).
\end{align*}
And since for any fixed sequences $\xi_i \to\xi$ and $\eta_i \to\eta$ we have
\[\abs{\limsup_{i\to\infty} d_{\skrig}(o,\xi_i)+d_{\skrig}(o,\eta_i)-d_{\skrig}(\xi_i,\eta_i) -(\xi,\eta)_0 } \leq 2 \delta\]
we at most accumulate an error of $4 \delta$ when splitting the above $\limsup$ into two. Switching between the $\limsup$ and $\liminf$ also just changes the value by at most $2 \delta$. Hence
\[2(\xi,\eta)^{\skrig}_g + \limsup_{\substack{\xi_i \to\xi \\ \eta_i\to\eta}}d_{\skrig}(o,\xi_i)+d_{\skrig}(o,\eta_i)- d_{\skrig}(g,\xi_i)-d_{\skrig}(g,\eta_i)\begin{cases}\geq(\xi,\eta)_o^{\skrig}-6\delta\\ \leq (\xi,\eta)_o^{\skrig}+6\delta\end{cases}\]
and we have proved the claim.

\end{proof}

\section{Randomized Geodesic Flow}\label{sec Theta push-forward}

\noindent Previously in Section 
\ref{sec Theta bd} and Section \ref{sec Theta disintegration} we constructed the two measures $\Theta_!$ and $\Theta_2$ which highlight what we expect to see as the limit points of a bi-infinite random walk as we move its origin along the random walk paths.

We now move on to a completely different construction of the measure $\Theta_3$ mentioned in Theorem \ref{theorem 1}, Part (iii), which will descend from a measure on the space of all bi-infinite random walk paths. This version $\Theta_3$ highlights the connection between the $G$-action and the flow along the random walk paths. This connection is also seen in the deterministic setup and is what in the end will lead to the statements of ergodicity.

\subsection{The Measure $\Q$ on $G^\Z$}\label{def Q}
Recall from Definition \ref{def bi-inf RW} that we denote by
\[G^{\Z}:= \bigg\{ \text{all bi-infinite random walk paths } (X_{z})_{z\in\Z} \text{ in } G\bigg\} \]
and write
\[G_{g}^{\Z}:= \bigg \{ \text{all bi-infinite random walk paths } (X_{z})_{z\in\Z} \text{ in } G \text{ with } X_{0}=g \bigg \}\]
for $g \in G$, and that $G^{\Z} = \bigcup_{g\in G}G^{\Z}_g$ is a disjoint union.

\noindent We now consider the following actions on $G^{\Z}$

\begin{align}\label{actions}
G \curvearrowright G^{\Z} &\text{ via } g.(x_z)_{z} := (g x_z)_{z}\\
 G^{\Z}\curvearrowleft \Z &\text{ via } k.(x_z)_{z} := (x_{z-k})_{z}.
\end{align}

These actions commute, since
\[g.(k.(x_z)_{z}) = g.(x_{z-k}) = (gx_{z-k})_{z} = k.(gx_z)_z = k.(g.(x_z)_z).\]

\begin{lemma}\label{lemma measure Q}
Recall the measure $\bar\P_g$ from Definition \ref{def bi-inf RW}. Then 
\begin{equation}\label{Q def}
\Q (A) := \sum_{g\in G} \bar\P_g(A \cap G_g^{\Z})
\end{equation}
is a well-defined measure on $G^{\Z}$, which is invariant under both the $\Z$-action and the $G$-action defined above.
\end{lemma}

\begin{proof}
Note that $\Q$ is well defined on $G^\Z$ since $G^{\Z} = \bigcup_{g\in G}G^{\Z}_g$ is a disjoint union.  Since for each $g\in G$, the measure $\bar\P_g$ is $\sigma$-finite, so is $\Q$.

We will now show $G$-invariance of $\Q$. Since for $A\subset G^{\Z}_x$ and $g \in G$ we have $gA\subset G^{\Z}_{gx}$, it suffices to show the invariance for a set $A \subset G^{\Z}_x$ for some $x\in G$. Then, for $g \in G$
\[\Q(gA) = \bar\P_{gx}(gA) = \bar\P_x(A) = \Q(A)\]
and hence $\Q$ is $G$-invariant.

To show the $\Z$-action, we further restrict ourselves to studying cylinder sets of the form 
\[A= G^{\N}\cross\{g_{-m}\}\cross\{g_{-m+1}\}\cross...\cross\{g_n\}\cross G^{\N},\]
since these generate the $\sigma$-algebra on $G^{\Z}_{g_{0}}$. Since for $g\in G$ the measure $\bar\P_{g}$ is supported on the space of sequences, for which the increments $g_{i-1}^{-1}g_{i}\in S$ are in our generating set, the set $A$ can be equivalently described by 
\[S^{\N}\cross\{s_{-m}\}\cross...\cross\{s_{-1}\}\cross\{g_{0}\}\cross \{s_1\}\cross...\cross\{s_n\}\cross S^{\N}.\]
Then the shifted set $1.A$ is equivalently described by
\[S^{\N}\cross\{s_{-m}\}\cross...\cross\{s_{-1}g_{0}\}\cross\{s_{-1}^{-1}\}\cross \{s_1\}\cross...\cross\{s_n\}\cross S^{\N}.\]
By construction of the measures $\bar\P_x$, this means
\[\Q(A) = \bar\P_{g_0}(A)= \mu(s_{-m})... \mu(s_{-1})\mu(s_{1}) ... \mu(s_n)\]
and
\[\Q(1.A) = \bar\P_{s_{-1}g_0}(1.A) = \mu(s_{-m})... \mu(s_{-2})\mu(s_{-1}^{-1})\mu(s_{1}) ... \mu(s_n).\]
And since we assume $\mu$ to be symmetric , this means $\Q(A)= \Q(1.A)$, i.e.\  $\Q$ is invariant under the $\Z$-action.
\end{proof}

\subsection{The Measures $\widehat\Q$ and $\Theta_3$}\label{sec hatQ Theta3} In order to define the measures $\widehat\Q$ and $\Theta_3$ (recall \eqref{def QZ int} and Theorem \ref{theorem 1}, part (iii)) we now move to quotients of $G^{\Z}$ by the respective actions on $\Z$ and $G$ defined in the previous section. We will identify the quotients with the respective fundamental domains, i.e. sets which contain one point of each orbit of the respective action. 

The fundamental domain for the $G$-action on $G^\Z$ can be chosen to be $G^{\Z}_o$, since $G^{\Z}=\bigcup_{g\in G}G^{\Z}_g$ is a disjoint union. This fundamental domain is already equipped with the measure $\bar\P_{o}$ (which is also the restriction of $\Q$ to $G^Z_o$). 

The $\Z$-action on $G^{\Z}$ then induces a $\Z$-action on $G^{\Z}_{o}$, which we denote by $(\tau_z)_{z\in\Z}$. More precisely, we set 
\[\tau_z(x_{n})_{n\in\Z} : = x_{z}^{-1}(x_{n-z})_{n\in\Z}, \qquad \forall z\in \Z.\]
We note that, heuristically, this $\Z$-action on $G^{\Z}/G$ can be thought of the {\it flow} along the random walk path, i.e. the analogue of the geodesic flow in the deterministic case.

We now define the fundamental domain for the $\Z$-action. This is slightly more tricky and will only be identified up to a set of measure $0$. Indeed, since the random walk paths almost surely diverge (recall Theorem \ref{divergence RW}), we may restrict $G^{\Z}$ to paths which, for any $n$, stay within a ball $B_n(o)$ of radius $n$ only a finite amount of time. A fundamental domain for the $\Z$-action on this restriction of $G^{\Z}$ is then

\begin{equation}\label{def calD}
\begin{aligned}
\skrid=\bigcup_{g \in G} \bigg \{(x_z)_{z} : x_0 = g , & d_{\skrig}(x_{z},o)> d_{\skrig}(x_0,o) \text{ for }z<0, \\
&\mbox{and}\,\, d_{\skrig}(x_{z},o)\geq d_{\skrig}(x_0,o) \text{ for }z>0 \bigg \}. 
\end{aligned}
\end{equation}
In other words, we choose the parametrization of $(x_z)_{z\in\Z}$ according to where it gets closest to $o$ the first time. We again identify the quotient space with this fundamental domain. The measure $\widehat \Q$ on the quotient space is the restriction of $\Q$ to the fundamental domain, i.e.
\begin{equation}
\widehat \Q (B) = \Q(B \cap \skrid) \quad \forall B \subset G^{\Z}.
\end{equation}
Since $G^{\Z}$ is the space of all random walk paths with all parametrizations, the quotient by $\Z$ can be thought of the space of all sequences in $G$ without a specific choice of parametrization.

\begin{theorem}\label{behavior Theta 3}\label{density Theta}
Let $\Theta_3:= \bd_{\ast}\widehat \Q$ be the push-forward under the boundary map. Then $\Theta_3$ is invariant for the $G$-action on $\partial^2 G$ and satisfies
\[\frac{\d \Theta_3}{\d (\nu_o\otimes \nu_o)}(\xi,\eta)\asymp \e^{2(\xi,\eta)_o^{\skrig}}.\]
\end{theorem}
Theorem \ref{behavior Theta 3} is the main result of this section. Following the intuition outlined in Section \ref{example density tree} and building on the technical lemmas proved in Section \ref{bounds on Green combinatorics}, Theorem \ref{behavior Theta 3} will be proved in Section \ref{sec-pfofbehavior Theta 3}.

\subsection{A Guiding Example: the Free Group}\label{example density tree}
Due to the highly disconnected nature of free groups, the proof that $\Theta_3$ has the desired density is much easier in the case that $G$ is a free group. Although proving Theorem \ref{density Theta} is significantly harder in the case of general hyperbolic groups, we will outline the proof for this case here as a guiding philosophy.

\noindent For simplicity we will study the case of the simple random walk on $\mathbb{F}_q$. In other words, we set $G$ to be the free group with $q$ generators and $S$  a minimal, symmetric generating set on $G$. Let $\mu$ be the uniform distribution on $S$. The Cayley graph of $G$ with respect to $S$ then is a $2q$-regular tree $\Ga$.

\noindent For $g\in G$, we refer to $\skrit_g$ as the set of branches at $g$, meaning the connected components in $\Gamma \setminus \{g\}$ which do not contain $o$. One element $T \in \skrit_g$ is a branch at $g$. Let $\skrit = \bigcup_g \skrit_g$. The $\sigma$-algebra on $\partial G$ is then generated by the sets $\{(x_n)_n \text{ fully lies in } T \text{ eventually}\}$ for all $T\in \skrit$.

\noindent Let $B_1,B_2$ be two connected, disjoint subsets of $\partial G$ and $T_1, T_2\in \skrit$ the two branches in $G$, such that $\bd^{-1}(B_i)=\{(x_n)_n \text{ lies in }T_i\text{ eventually}\}$ for $i=1,2$.
Here it suffices to look at disjoint sets, since in $\partial^2 G$ the diagonal is removed. Furthermore, the pre-images under the boundary map are indeed of this form, since we chose $B_1$ and $B_2$ to be connected.

\noindent We now consider two cases, either $T_1$ and $T_2$ are subsets of different branches at $o$, or they lie in the same one. In the first case 
\[\bar\P_g[(X_n)_{n\in\N} \in T_2 \text{ eventually}, X_{i}\not\in B_{d_{\skrig}(g,o)-\delta},(X_{-n})_{n\in\N} \in T_1 \text{ eventually}, X_{-i}\not\in B_{d_{\skrig}(g,o)}]=0\]
for all $g \ne o$, since any geodesic from $T_1$ to $T_2$ has to pass through $o$. So
\[\Theta(B_1\cross B_2)= \P_{o}[(X_n) \in T_1 ] \P_{o}[(X_{-n}) \in T_2 ] = \nu_o(B_1)\nu_o(B_2).\]
And since the Green metric in the tree is actually additive along geodesics, we have $(\xi,\eta)^{\skrig}_{o}=0$ for all $\xi\in T_1, \eta \in T_2$. This means the density is correct for all points in different branches off of $o$.

\noindent The second case is that $T_1$ and $T_2$ lie in the same branch at $o$. Since they are disjoint, however, there does exist a point $x \in G$ such that $T_1$ and $T_2$ lie in different branches off of $x$. Then
\[\bar\P_g[(X_n)_{n\in\N} \in T_2 \text{ eventually}, X_{i}\not\in B_{d_{\skrig}(g,o)-\delta},(X_{-n})_{n\in\N} \in T_1 \text{ eventually}, X_{-i}\not\in B_{d_{\skrig}(g,o)}]\ne0\]
if and only if $g \in [o,x]$. Now, since on the tree the restriction to $\{X_{i}\not\in B_{d_{\skrig}(g,o)-\delta}\}$ just means that we might not step towards $o$ from $g$, and similarly $\{X_{i}\not\in B_{d_{\skrig}(g,o)}\}$ means we might not step towards $o$ from $x_1$, the ratios
$$
\begin{aligned}
&c_1 = \frac{\P_g[(X_n)_{n\in\N} \in T_1 \text{ eventually}, X_{i}\not\in B_{d_{\skrig}(g,o)}]}{\P_g[(X_n)_{n\in\N} \in T_1 \text{ eventually}]} \qquad \text{ and } \\
&c_2 = \frac{\P_g[(X_n)_{n\in\N} \in T_2 \text{ eventually}, X_{i}\not\in B_{d_{\skrig}(g,o)-\delta}]}{\P_g[(X_n)_{n\in\N} \in T_2 \text{ eventually}]}, 
\end{aligned}
$$
do not depend on the base point $g$. Together with the fact that 
$$
\P_g[(X_n)_{n\in\N} \in T_i \text{ eventually}] = \nu_g(T_i) \qquad\mbox{for}\qquad i=1,2,
$$
we get
\begin{align*}
&\sum_{g \in [o,x] } \bar\P_g\Big [(X_n)_{n\in\N} \in T_2 \text{ eventually}, X_{i}\not\in B_{d_{\skrig}(g,o)-1},(X_{-n})_{n\in\N} \in T_1 \text{ eventually}, X_{-i}\not\in B_{d_{\skrig}(g,o)}\Big]\\
&= c_1c_2 \sum_{g \in [o,x]} \nu_{g}(B_1)\nu_g(B_2)\\
&=c_1c_2  \sum_{g \in [o,x]} \int_{B_1\cross B_2}\exp\bigg\{\limsup_{\xi_i \to \xi, \eta_i \to \eta} d_{\skrig}(o,\xi_i)+ d_{\skrig}(o,\eta_i)-(d_{\skrig}(g,\xi_i)+ d_{\skrig}(g,\eta_i))\bigg\} \d(\nu_o \otimes \nu_o)\\
&=c_1c_2 \int_{B_1\cross B_2}\e^{2 (\xi,\eta)^{\skrig}_{o}} \sum_{g \in [o,x]} \e^{- 2d_{\skrig}(g,x)} \d(\nu_o \otimes \nu_o)\\
\end{align*}
Since $\sum_{g \in [o,x]} \e^{- 2d_{\skrig}(g,x)}$ converges as $x\to\infty$, $\Theta$ thus has the correct density on all products of connected, disjoint sets.

\begin{remark}({\bf Relation between $\Theta_1$, $\Theta_2$ and $\Theta_3$.})\label{comparison versions Theta}
The purpose of this remark is to provide the reader with a heuristic picture about about the relationships between the measures $\Theta_1$, $\Theta_2$ and $\Theta_3$. Note that in the same setting as in the above Example \ref{example density tree}, we can also explicitly calculate $\Theta_1$ and $\Theta_2$. Indeed, for the $\Theta_1$ we may choose the rays from $o$ to $\xi\in \partial G$ to be the unique geodesics in the graph metric. In this case, these are also the geodesics in the Green metric. Then

\begin{align*}
\frac{\d \Theta_1}{\d (\nu_o\otimes \nu_o)}(\xi,\eta) 
&= \lim_{n\to \infty}\sum_{z=-n}^n \frac{1}{2n+1} \exp \bigg\{\limsup_{\substack{\xi_n \to \xi,\\\eta_n \to \eta}}d_{\skrig}(o, \xi_n)+d_{\skrig}(o, \eta_n) -d_{\skrig}(p_z, \xi_n)-d_{\skrig}(p_z, \eta_n)\bigg\}\\
&= \lim_{n\to \infty} \frac{1}{2n+1} \bigg( \e^{2(\xi,\eta)^{\skrig}_{o}}\sum_{p_z \in [o,x], \abs{z}\leq n} \e^{-d_{\skrig}(p_z,x)} + \e^{2(\xi,\eta)^{\skrig}_{o}} \sum_{p_z \not\in [o,x], \abs{z}\leq n} \e^{0} \bigg)\\
&= \e^{2(\xi,\eta)^{\skrig}_{o}} 
\end{align*}
where, as in the previous example, $[o,x]$ is the part where $[0,\xi)$ and $[0,\eta)$ overlap and we use the fact, that the Green metric is additive in this case. Similarly one gets that the density of $\Theta_2$ originates from the same behavior of the densities $\frac{\d (\nu_{x_z}\otimes\nu_{x_z})}{\d (\nu_{o}\otimes\nu_{o})}$ for points $x_z$ close to the geodesic $(\xi,\eta)$ by the same argument as in the proof of Lemma \ref{density Theta1}. So, in the end, the fact that the density of $\Theta_1$, $\Theta_2$ and $\Theta_3$ with respect to $\nu_o\otimes \nu_o$ is in $[C^{-1} \e^{2(\xi,\eta)_o^{\skrig}},C \e^{2(\xi,\eta)_o^{\skrig}}]$ for some $C >0$ comes from the same behavior of the densities $\frac{\d (\nu_{x_z}\otimes\nu_{x_z})}{\d (\nu_{o}\otimes\nu_{o})}$. 

\noindent However, whereas $\Theta_3$ comes from the push forward of a measure on the $\Z$-quotient of $G^{\Z}$, $\Theta_1$ and $\Theta_2$ align more closely with the $G$-quotient of $G^{\Z}$. 
Indeed, $\Theta_2$ averages over the hitting measures $(\nu_{x_z}\otimes\nu_{x_z})$ along the path $(x_n)_{n\in\Z}$, where $x_z$ is the position at time $0$ of the path $z.(x_{n})_{n\in\Z}$. 
In order to obtain the density of $\Theta_2$ with respect to $(\nu_o\otimes \nu_o)$, the measure $(\nu_{x_z}\otimes\nu_{x_z})$ then gets ``shifted back" by the position $x_z$ at time $z$ via the density $\frac{\d (\nu_{x_z}\otimes\nu_{x_z})}{\d (\nu_{o}\otimes\nu_{o})}$. This is very similar in spirit to the induced $\Z$-action on the fundamental domain for the $G$-quotient of $G^{\Z}$. This action also takes a path $(x_n)_{n\in\Z}$, shifts the time by $z$ and then shifts the whole path back to $x_z^{-1}(x_n)_{n\in\Z}$ by $x_z$. This shifting back is needed to again get a path that still is at $o$ at time $0$. The difference is however, that the $\Z$-shift works ``on the inside" of the space $G^{\Z}$, by shifting the paths directly, whereas the ``shift" in $\Theta_2$ works from the point of view of the measure. 

This means, in a way, that $\Theta_2$ lives on the $G$-quotient of $G^{\Z}$ and this is where its $G$ (quasi)-invariance comes from. The Cesàro-limit over the shifts in the definition of $\Theta_2$ comes from the fact, that we want $\Theta_2$ to only depend on the limits on the boundary of the paths. For $\Theta_1$ the behavior is similar, but as we only regard one path $(p_z)_{z\in\Z}$ for each pair of points $\xi,\eta$, the connection to $G^{\Z}/G$ is not as immediate.
\end{remark}

\subsection{Bounds on the Greens Function and Combinatorial Arguments}\label{bounds on Green combinatorics}

\noindent As remarked earlier, in general hyperbolic groups we cannot assume that by removing a ball of a certain size we disconnect points from each other. Due to the geometric nature of hyperbolic groups it does however significantly increase the distance between points on opposite sides of the ball. So a random walk started on one side of a ball is much less likely to hit a point on the opposite side of the ball when forced to walk around the ball. 

\noindent We will need to quantify how much forcing the random walk to walk around a ball actually changes the hitting probabilities. As discussed in Section \ref{sec proof ideas},  
we will provide a refined quantitative analysis of how the probability of walking from a point $x$ to a point $y$ at distance more than $2n$ away changes, if the random walk is forced to walk around a ball of radius $n$ centered on the geodesic $[x,y]$. Note that we will consider cases where the center $o$ of the ball that gets removed does not lie on the geodesic between $x$ and $y$. Furthermore, our bounds are sensitive to the distance of $y$ to the ball's center $o$, effectively yielding estimates for the hitting distribution on a geodesic $[o,\xi)$ of the random walk started at $x$ when forced to walk around a ball.

\noindent In the following we let $d\geq 1$ be as in Theorem \ref{thm-bs} by Bonk and Schramm, such that there exists a rough similarity $\psi:G\to \H^d$, i.e. a map for which there exist $\lambda,c>0$ such that 
\begin{equation}\label{rough similarity}
\abs{\lambda d_{\skrig}(g,h)- d_{\H^d}(\psi(g),\psi(h))}\leq c \quad\forall g,h\in G.
\end{equation}
Using this map, we may make use of the explicit knowledge we have about the space $\H^{d}$. We will throughout this section always look at $\H^{d}$ as given by the ball model. 

\noindent For $x,y \in \H^d$, we denote by $\angle([0,x),[0,y))$ the Euclidean angle between the rays $[0,x)$ and $[0,y)$ as seen in the ball model for $\H^d$ with center $0$.

\noindent We further define the {\emph restricted Green's function} as
\begin{equation}\label{restricted green}
\skrig(a,b;A) = \sum_{n\in\N}\P_a[X_n = b, X_2,...,X_{n-1}\notin A]
\end{equation}
for $a,b \in G, A\subset G$. 

\begin{lemma}\label{decay greens function}
For $n\in \N$ big enough, there exist $\varepsilon,r >0$ such that for all $a,b\in G$ with $d_{\skrig}(a,o),d_{\skrig}(b,o)\geq n$ we have
\[\skrig(a,b;B_{n}^{\skrig}(o))\leq \exp\big[- rn \alpha \e^{\varepsilon n}\big] \skrig(b,b^n)\]
where $b^n$ is a point on $[o,b)$ with $\lfloor d_{\skrig}(b^n,0)\rfloor =n$ and $\alpha= \angle([0,\psi(a)[0,\psi(b))$ is the angle between $[0,\psi(a))$ and $[0,\psi(b))$ in $\H^d$.
\end{lemma}
The idea of the proof is to construct a sequence of barriers $H_1,...H_M$ which the random walk has to pass through, and show that the probability of passing from one barrier to the next is sufficiently small. The concrete statement is the following.

\begin{lemma}\label{lemma def barriers}
Let $a\in G$ be such that $d_{\skrig}(a,o)\geq n$ and fix a point $\xi\in \partial \H^d$ and denote by $\alpha \in[0,\pi]$ the angle between the rays $[0,\xi)$ and $[0,\psi(a)]$. 
For fixed $n\in\N,\varepsilon>0$, set $2M = \lfloor\alpha \e^{\varepsilon n}\rfloor$ and let $I_{i} = [\frac{2i-1}{\e^{\varepsilon n}},\frac{2i}{\e^{\varepsilon n}}]$ for $i=1,..,M$.
Then there exist barriers $H_1,..H_M\subset G$ and $C_1>0$ such that
\begin{equation}\label{bd green function fixed end}
\sum_{a_{i} \in H_{i}}\skrig(a_i,a_{i+1})^2 \leq C_1 \e^{-rn} \skrig(a_{i+1}, a_{i+1}^{n})^2
\end{equation}
for all $a_{i+1}\in H_{i+1}$ and where $a_{i+1}^{n}$ is a point on $[a_{i+1},0]$ such that $\lfloor d_{\skrig}(a^n_{i+1},o)\rfloor=n$.
In particular this also means that there exists $C_2>0$ such that
\begin{equation}\label{bd green function double sum}
\sum_{a_{i} \in H_{i},a_{i+1} \in H_{i+1}}\skrig(a_i,a_{i+1})^2\leq C_2 \e^{-rn}.
\end{equation}
These barriers can be defined as 
\begin{equation}\label{barriers}
H_i = H(\alpha_i) = \bigcup_{\eta} \bigg\{x \in G: d_{\skrig}(x,o)\geq n \text{ and } d_{\H}(\psi(x),[0,\eta))<C_0\bigg\}
\end{equation}
where the union is over all $\eta\in \partial\H^d$ such that $\angle ([0,\eta),[0,\xi))=\alpha_i$,
for some choice $\alpha_i\in I_i$ for $i=1,..,M$ and for some $C_0>0$.
\end{lemma}

\begin{figure}[h!]
\centering

\begin{tikzpicture}[scale=0.3] 

\coordinate[label=left :$\psi(a)$]  (a) at (4,11.5);
\coordinate[label=below:$0$]  (o) at (7,9.5);

\fill (a) circle (3pt);
\fill (o) circle (3pt);

\draw  (7,9.5) circle (2cm);
\draw (7,9.5) circle (9.5cm);

\draw [line width=1.5pt] (7,9.5) -- (16.5,9);
\draw [line width=1.5pt] (7,9.5) -- (16.5,10);
\fill[cyan] (7,9.5) -- (16.5,9) arc [start angle=-5, end angle=1.4, radius=8.9cm] -- cycle;
\coordinate[label=right:$B$]  (B) at (16.5,9.5);

\draw [line width=1.5pt] (8.75,10.5) -- (15,14.6)node[right]{$H_4$};
\draw [line width=1.5pt] (8,11.25) -- (11.75,17.75)node[above]{$H_3$};
\draw [line width=1.5pt] (7,11.5) -- (7,19)node[above]{$H_2$};
\draw [line width=1.5pt] (6,11.25) -- (2.25,17.75)node[above]{$H_1$};

\draw [line width=1.5pt] (8.75,8.5) -- (15.4,5)node[right]{$H_4$};
\draw [line width=1.5pt] (8,7.75) -- (11.75,1.25)node[below]{$H_3$};
\draw [line width=1.5pt] (7,7.5) -- (7,0)node[below]{$H_2$};
\draw [line width=1.5pt] (6,7.75) -- (2.75,1)node[below]{$H_1$};

\draw[red, decorate, decoration={snake, amplitude=0.5mm, segment length=5mm}] (4,11.5) .. controls (6,13) .. (16.5,9.5); 

\draw[red, decorate, decoration={snake, amplitude=0.5mm, segment length=5mm}] (4,11.5) .. controls (5,2) and (13,11) ..(16.5,9.5); 

\end{tikzpicture}
\label{fig:barriers}
\caption{This is an illustration of the embedding of a random walk started in $a$ into $\mathbb{H}^2$ under the Bonk-Schramm embedding. In the image, all paths converging to a point in $B$ and staying outside $B_{n}(0)$ have to hit each of the barriers $H_1,...,H_4$ at some time.}
\end{figure}
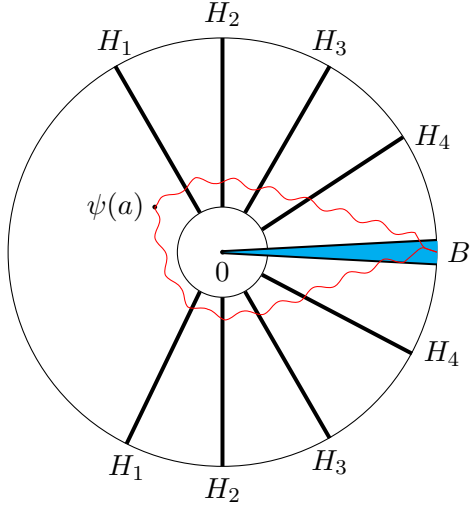

\begin{remark}
The definition of the barriers is motivated by the fact that we want to apply this lemma to estimate the probability to walk from $\psi(a)$ to a small neighborhood of $\xi$. They could also be chosen in a more general manner, e.g by only looking at the angle between $[0,\psi(a)]$ and the rays in the barrier.
\end{remark}

Before we can however turn to proving these statements, we first need to show some estimates on the number of points of $\psi(G)$ we may find in certain regions of $\H^d$.

\begin{lemma}\label{number hitting points}
Let $\xi \in \partial \H^d$, and $\beta \in [0,\pi]$. Let $B\subset \partial \H^d$ be the set of all $\zeta$ such that $\angle([0,\xi),[0,\zeta))<\beta$ and set
\[A:= \bigcup_{\zeta \in B}\big\{x \in [0,\zeta)\big\}.\]
Further, let $\alpha \in (\beta,\pi]$ and $H=H(\alpha)$ be a barrier as defined in \eqref{barriers}, i.e.
\[ H(\alpha) = \bigcup_{\eta} \bigg\{x \in G: d_{\skrig}(x,o)\geq n \text{ and } d_{\H}(\psi(x),[0,\eta))<C_0\bigg\}.\]
Let $\lambda>0$ be the multiplicative constant for the rough similarity of the Bonk-Schramm embedding \eqref{rough similarity}.
Then there exist $C_1,C_2>0$ such that
\begin{enumerate}
\item $n(m):= \abs{\{g\in H : \lfloor d_{\skrig}(g,o)\rfloor =m\}} \leq C_1$.
\item$N(m):= \abs{\{g\in G : \lfloor d_{\skrig}(g,o)\rfloor =m, \psi(g)\in A\}}\geq C_2 \e^{m}$, as long as $\psi^{-1}(B)\ne \emptyset$. The constant $C_2$ might depend on $\alpha.$
\end{enumerate}
\end{lemma}

\begin{proof}
$(i)$: Let $H_m :=\{g\in H : \lfloor d_{\skrig}(g,o)\rfloor =m\}$. We will look at which region in $\H^d$ the image of $H_m$ under the Bonk-Schramm embedding $\psi$ can lie in. 
We estimate the size of this region, and the number of balls $B_1$ of radius one needed, to cover this region. This then leads to an upper bound on $n(m)$, since under the Bonk-Schramm embedding each of these balls can contain at most a bounded number of points in $\psi(G)$.

\noindent Since for $g\in H_m$ we have $\lfloor d_{\skrig}(o,g)\rfloor=m$, in $\H^d$, $\d_{\H}(\psi(g),0)\in[\lambda m -c,\lambda (m+1) +c]$.
Denoting by $B^{\circ}_{r}$ the interior of the ball of size $r$ around $0$ in $\H^{d}$, we thus have
\[n(m)= \bigg\vert\bigcup_{\eta} \bigg\{g\in G : \lfloor d_{\skrig}(o,g)\rfloor=m,\, d_{\H}(\psi(g),[0,\eta))<c\bigg\}\bigg\vert \leq \abs{\psi(H) \cap\big(B^{\circ}_{\lambda (m+1) +c}\backslash B^{\circ}_{\lambda m -c}\big)}.\]

By definition, $\psi(g)$ is at distance at most $c$ from some ray $[0,\eta)$. Denote by $\eta_g$ such a point at distance less than $c$ to $\psi(g)$. 
Then $\angle([0,\psi(g)],[0,\eta))< c_1e^{-\lambda m}$ for some $c_1>0$. This can be seen as follows. For $\xi, \eta \in \partial \H^d$, we have 
\[(\xi,\eta)_0^{\H^d} = \log(\sin\Big({\frac{\vartheta}{2}} \Big)^{-1})\]
where $\vartheta = \angle ([0,\xi),[0,\eta))$. So as the angle between the rays goes to $0$, their Gromov product goes to infinity. And, for small values of $\vartheta$, 
\[(\xi,\eta)_0^{\H^d} \approx \log \frac{2}{\vartheta}.\]
More generally, $\log(\nicefrac{2}{\vartheta})$ is an upper bound on the Gromov product. So let $x,y\in \H^d$ with $d_{\H}(x,0)=k$ and $d_{\H}(x,y)= c$, where $k$ is big compared to $c$, and assume that $r= d_{\H}(x,0) - d_{\H}(y,0)>0$. Then
\[(x,y)_0^{\H} = (0,y)_x^{\H} - d_{\H}(x,0) = \frac{1}{2} \big( d_{\H}(x,y) + r \big)-d_{\H}(x,0)\]
and hence
\[\vartheta \approx \exp\Big(-k + \frac{1}{2}(c + r)\Big) \leq \exp\Big(-k+\frac{c}{2}\Big).\]
Since $sin(x)\approx x$ is a good approximation for small values of $x$, we have to make sure that $k$ is big enough with respect to $c$. Note that we also used that for $x,y$ far enough from $0$, we have $(x,y)^{\H}_0= (\xi,\eta)^{\H}_0$ where $\xi, \eta$ are the hitting points on $\partial \H^d$, when we extend the rays $[0,x)$ and $[0,y)$ past $x$ and $y$ to the boundary. This holds, since both Gromov products denote the maximum distance along  $[0,\eta)$ from $0$, at which $[0,\xi)$ and $[0,\eta)$ are $2\delta$ apart. Here $\delta>0$ is the
minimal constant for which  $\H^d$ is $\delta$-hyperbolic.

\noindent Applying this to our setting, there exists $c_1>0$ depending on $c$ but not on $k$, such that the angle between $[0,\psi(g)]$ and $[0,\eta)$ is at most $c_1 \e^{-\lambda k}$. Hence, for any $g\in H_m$,
\[\angle([0,\psi(g)],[0,\xi))\in[\beta - c_1 \e^{-\lambda m},\beta - c_1 \e^{-\lambda m}].\]
By looking at $\H^d$ in the ball model in polar coordinates, the proportion of rays with angle in $[\beta - c_1 \e^{-\lambda m},\beta - c_1 \e^{-\lambda m}]$ to $[0,\xi)$ is 
\[(2\pi)^{-(d-1)} \big((\beta + c_1 \e^{-\lambda m})^{d-1}-(\beta - c_1 \e^{-\lambda m})^{d-1}\big)\leq c_2 \e^{-\lambda m}\]
for some $c_2>0$.

\noindent Now, for any point $g \in H$, the corresponding closest point $\eta_g$ on one such ray, is at distance $d_{\H}(x,0)\in[\lambda m-2c,\lambda m +2c]$ by the triangle inequality.
Hence the images of points in $H_m$ under $\psi$ all lie in a space of size
\[ c_2 \e^{-\lambda m} \abs{B^{\circ}_{\lambda m +2c}\backslash B^{\circ}_{\lambda m -2c}} = c_2 \e^{-\lambda m}\Big(\e^{\lambda m +2c} - \e^{\lambda m +2c}\Big) \leq C_1 \]
for an appropriate choice of $C_1>0$.

\noindent So we can cover the region in $\H^{d}$, where $\psi(H_m)$ lies, by $C_1$ balls of diameter $1$. And
since in a ball of size $1$ in $\H^d$, there are at most as many points of $\psi(G)$ as fit into $B^{\skrig}_{\lambda^{-1} +c}(o)$ in the group, this means that $n(m)\leq C_1$.

$(ii)$: By assumption there exists a point $\hat \xi \in \psi^{-1}(B)$. Let $(x_n)_{n\in\N}$ be a sequence of neighboring points in $G$ which converges to $\hat \xi$. Then at some time $N>0$, $\psi(x_N) \in A$. And since the image of $G$ under $\psi$ is convex, this cannot happen for the first time at a distance $d_{\skrig}(x_N,o)$ where the angle between the images of two neighboring points is much less than $\alpha$. By the above study of the connection between the Gromov product and the angle in $\H^d$, this means it must be $\e^{-\lambda d_{\skrig}(x_N,o)}\gtrsim \alpha$. Now, walking away from $x_N$ towards the boundary can only lead to an angle of at most $c_1\e^{-d_{\skrig}(x_N,o)}$. Hence eventually all group geodesics starting at a point inside $A$ will converge to a point in $B$. So the number of points inside $A_1$ grows exponentially with the same growth factor as the volume growth of $G$. Since $\nu_G =1$, we have $N(m)\geq C_2 \e^{m}$ where the constant $C_2$ absorbs the fact that the exponential growth only starts after $X_N$. 
\end{proof}

\begin{proof}[{\bf Proof of Lemma \ref{lemma def barriers}}]
We will show that if we start with a fixed barrier $H_M$, we may choose an $\alpha_{M-1}$ such that for $H_{M-1}= H(\alpha_{M-1})$ the bound in \eqref{bd green function fixed end} holds. The claim then follows inductively.

\noindent Let $H_M:=H(\alpha_M)$ for any $\alpha_M\in I_M$ and $a_M\in H_M$, and let $\alpha_{M-1}\in I_{M-1}$ be chosen uniformly at random, i.e. with distribution being the rescaled Lebesgue measure $\e^{\varepsilon n}\Leb$. Denote by $I(a)$ the interval in $I_{M-1}$ of angles $\alpha$ for which $a \in H_{\alpha}$ and by $H(I_{M-1})=\bigcup_{\alpha \in I_{M-1}}H(\alpha)$. Then 
\begin{align}
\E\Big[\sum_{a_{M-1} \in H_{M-1}}\skrig(a_{M-1},a_{M})^2\Big]&=\sum_{\substack{a \in G,\\d_{\skrig}(a,0)\geq n}}\skrig(a,a_M)^2 \e^{\varepsilon n}\Leb(I(\alpha)) \nonumber\\
&{\leq} \sum_{a \in H(I_{M-1})}\skrig(a,a_M)^2 \e^{\varepsilon n}\e^{- \lambda d_{\skrig}(a,0)}\label{ast1}\\
&{\leq} C_A^2 \e^{\varepsilon n} \sum_{a \in H(I_{M-1})}\skrig(a,x_a)^2\skrig(x_a,a_M)^2 \e^{- \lambda d_{\skrig}(a,0)}\label{ast2}\\
&{\leq} C_A^2 \e^{(\varepsilon-\lambda) n} \sup_{a \in H(I_{M-1})}\skrig(x_a,a_M)^2 \sum_{a \in H(I_{M-1})}\skrig(a,x_a)^2\nonumber\\
&{\leq} C \e^{(\varepsilon-\lambda) n} \skrig(a_M,a_M^n)^2\nonumber.
\end{align}
In \eqref{ast1} we used that moving away by $c_0$ from a point $x \in [0,\eta)$ at distance $d(0,x)=n$ produces an angle of at most $2^{-1}\e^{\frac{c_0}{2}}\e^{-n}=c_1\e^{-n}$ to bound $\Leb(I(a))$. This follows again from the fact that for $\xi, \eta \in \partial \H^d$, we have 
\[(\xi,\eta)_0^{\H^d} = \log(\Big(\sin{\frac{\theta}{2}} \Big)^{-1})\]
where $\theta= \angle([0,\xi),[0,\eta))$, as in the previous proof of Lemma \ref{number hitting points}.

Along the same lines, we also get that since $\angle([0,\psi(a)),[0,\psi(a_M)))\geq\e^{-\nicefrac{\varepsilon n}{2}}$, the rays are more than $2\delta$ apart outside $B_{\varepsilon n}(0)$.
Therefore, there exists a point $x_a \in [a,b]$ such that $d_{\skrig}(x_a,0)\leq \varepsilon n$, which was used in \eqref{ast2}. We then used the Ancona inequality from Lemma \ref{ancona}, which says there exists $C_A>0$ such that for all geodesics $[a,a_M]$ and $x_a \in [a,a_M]$
\[\skrig(a,b)\leq C_A \skrig(a,x_a)\skrig(a_M,x_a).\] 
In the last line we used that $\sum_{a \in G}\skrig(a,0)^2<\infty$. And finally, we may bound the supremum  by $\skrig(a_M,a_{M}^n)$ where $a_{M}^n\in [o,a_M]$ is a point with $\lfloor d_{\skrig}(a^n_{M},o)\rfloor=n$, since $d_{\skrig}(x_a,o)\leq \varepsilon n$ and hence $d_{\skrig}(x_a,a_M)>d_{\skrig}(a_M,a_M^n)$, meaning the Green functions have the opposite relationship. 

So by choosing $\varepsilon < \lambda$, we have shown that the expectation $\E[\sum_{a_{M-1} \in H_{M-1}}\skrig(a_{M-1},a_{M})^2]$ on the LHS of \eqref{ast1}-\eqref{ast2} 
decays exponentially in $n$. Hence it is possible to choose $\alpha_{M-1}$ such that for $H_{M-1}:= H(\alpha_{M-1})$ we have 
\[\sum_{a_{M-1} \in H_{M-1}}\skrig(a_{M-1},a_{M})^2 \leq C \e^{(\varepsilon-\lambda) n} \skrig(a_M,a_M^n)^2.\]
This completes the proof of the first claim in Lemma \ref{lemma def barriers}.

The second claim then follows from this bound. Namely, if we denote by $a_M^n$ a point $B_{\varepsilon n}\cap [a_M,a_{M+1}]$
\begin{align*}
\sum_{\substack{a_{M-1}\in H_{M-1},\\a_{M}\in H_{M}}} \skrig(a_{M-1},a_M)^2 
&\leq \sum_{a_{M}\in H_{M}} C \e^{(\varepsilon-\lambda) n} \skrig(a_M,a_M^{n})^2\\
&=C \e^{(\varepsilon-\lambda) n} \sum_{m\geq n}\sum_{\substack{a_{M}\in H_{M}, \\\lfloor d_{\skrig}(o,a_M)\rfloor =m}} \skrig(a_M,a_M^{n})^2\\
&\leq C \e^{(\varepsilon-\lambda) n} \sum_{m\geq n} C_1 \e^{-2(m-\varepsilon n)}\\
\end{align*}
Here we used that by the definition of the Green metric,  $\skrig(x,y)^2 = \e^{-2d_{\skrig}(x,y)} \skrig(0,0)^2$.\\
Further $n(m):= \{a_{M}\in H_{M}, \lfloor d_{\skrig}(o,a_M)\rfloor =m\} \leq C_1 $, as shown in Lemma \ref{number hitting points}.
In total we get something that decays exponentially in $n$ for $\varepsilon < 1 \leq \lambda$, hence we have the second claim of Lemma \ref{lemma def barriers} and thus the proof of this lemma is also complete. 
\end{proof}

\begin{proof}[{\bf Proof of Lemma \ref{decay greens function}}]
Fix $\varepsilon >0$ and let $2M = \lfloor \e^{\varepsilon n}\rfloor$. We note that, if we define the barriers $H_1,...,H_M$ as in Lemma \ref{lemma def barriers}, the random walk restricted to staying outside of $B_n(0)$ has to pass each of those barriers at some point to get from $a$ to $b$. Let $H_0=\{a\}$ and restrict $H_M$ to $\{b\}$, then we can estimate
\[\skrig(a,b;B_n^c(0))\leq \sum_{a_0 \in H_0}... \sum_{a_M\in H_M} \prod_{i=0}^M \skrig(a_i,a_{i+1}) .\]
Denote by $L_i: \ell^2(H_i)\to \ell^2(H_{i+1})$ the operator defined as follows
\[(L_i f)(a)= \sum_{b\in H_{i}}\skrig(a,b) f(b).\]
Then the above sum can be written as $(L_0...L_M\delta_b)(a)$ and hence
\[\skrig(a,b,B_n^c(0))\leq \prod_{i=1}^M\norm{L_i}_{\ell^2}.\]
By the Cauchy Schwartz inequality we have
\begin{equation*}
\begin{aligned}
&\norm{L_i}_{\ell^2}^2 = \sum_{a \in H_i}\Big\vert\sum_{b\in H_{i+1}}\skrig(a,b)f(b)\Big\vert^2 \\
\leq& \sum_{a \in H_i}\Big(\sum_{b\in H_{i+1}}\skrig(a,b)^2\Big)\Big(\sum_{b\in H_{i+1}} f(b)^2\Big) =\sum_{a \in H_i, b\in H_{i+1}}\skrig(a,b)^2 \norm{f}_{\ell^2}^2
\end{aligned}
\end{equation*}
and hence by the bounds on the Green function shown in Lemma \ref{lemma def barriers}
\[
\skrig(a,b;B_n^c(0)) \leq (\e^{-rn})^M \skrig(b,b^n) = \exp\big\{-rn \alpha \,\e^{\varepsilon n} \big\}\skrig(b,b^n).
\]
\end{proof}

\subsection{Proof of Theorem \ref{behavior Theta 3}}\label{sec-pfofbehavior Theta 3}

We will now show that 

\[\Theta_3(B_1 \times B_2)\asymp \int_{B_1\times B_2}\e^{2(\xi,\eta)_0}\d(\nu_o\otimes\nu_o)(\xi,\eta)\]
for all disjoint sets $B_1,B_2$ in a generating set of the $\sigma$-algebra on $\partial^2 G$. 

Throughout the whole section, we let again $\psi$ be the rough similarity which embeds $G$ into $\H^d$ for $d\geq 1$ big enough as in Theorem \cite{BS}. And denote by $\lambda$ the multiplicative constant and the additive constant by $c$.

We first show that the lower bound in the equation above holds. This is done similarly as in the case of the free group in Section \ref{example density tree}, by just summing over origins which roughly lie on the shortest geodesic between $o$ and the geodesic $(\xi,\eta)$.

\begin{prop}[Lower bound on the density]\label{lower bound density}
There exists $C>0$ such that for all sets $B_1\times B_2\subset \partial2 G$ we have
\[\Theta_3(B_1 \times B_2)\geq C \int_{B_1\times B_2}\e^{2(\xi,\eta)_0}\d(\nu_o\otimes\nu_o)(\xi,\eta).\]
\end{prop}

\begin{proof}
Let $\xi_1,\xi_2\in \partial G$, $\beta_1,\beta_2 \in [0,\delta]$ and $B_1,B_2\subset \partial G$ be the sets 
\[B_{i}:=\{\eta\in \partial G: \,(\xi_i,\eta)_o^{\skrig}\geq \beta_i^{-1}\}.\]
The sets $B_1\times B_2$ then generate the $\sigma$-algebra on $\partial^2 G$.

Let $g \in G$ be the closest point to $o$ on a geodesic $(\xi_1,\xi_2)$. 
Then $(\xi_1,g)_o^{\skrig},(\xi_2,g)_o^{\skrig}\in [(\xi_1,\xi_2)_o^{\skrig}-r,(\xi_1,\xi_2)_o^{\skrig}+r]$ where $r$ is the universal constant needed for this to be possible in general, coming from the fact that $G$ is a coarse geometric space.
Now, since $B_1,B_2$ only contain points in a small neighborhood of $\xi_1$ and $\xi_2$, there exists $R>0$ such that for any $\eta_1,\eta_2$ in $B_1$ and $B_2$ respectively, if $\tilde g$ is the closest point to $o$ on the geodesic $(\eta_1,\eta_2)$, then $\tilde g \in B^{\skrig}_R(g)$.

Let $H:=\{x \in [0,\tilde g], \tilde g \in B_R(g)\}$. Note that, since two geodesics between the same two points are always at most at bounded distance from each other, the size of $H$ grows linearly in $d_{\skrig}(o,g)$. In particular the number of $x\in H$ with $\lfloor d_{\skrig}(o,x)\rfloor=m$ is bounded for all $m$.

Further there exists $c_0>0$ such that
\[\frac{\P_x\big(\bd(X_n)_{n\in\N}\in B_1 ,\,d_{\skrig}(X_n,o)\geq d_{\skrig}(o,x)\big)}{\P_x(\bd(X_n)_{n\in\N}\in B_1)} 
\geq c_0\]
independently of $x$, since we can bound
\[\P_x\big(\bd(X_n)_{n\in\N}\in B_1 ,\,d_{\skrig}(X_n,o)\geq d_{\skrig}(o,x)\big) \geq \P_x\big(\bd(X_n)_{n\in\N}\in B_1 ,(X_n)_{n\in\N} \in H_x \big)\]
where $H_x$ is a half space in $G$ which does not contain $o$ and has $x$ on its boundary.

The same bound then also holds, if we replace $\geq$ with $>$ in the probability in the numerator. Hence

\begin{align}
\Theta_3(B_1\times B_2)
&=\sum_{x\in G}\bar\P_x\Big(\bd\big(X_z\big)_{z\in\Z}\in B_1\times B_2,\, d_{\skrig}(X_{-n},o)> d_{\skrig}(x,o),\text{and }\nonumber \\
&\hspace{40mm}d_{\skrig}(X_n,o)\geq d_{\skrig}(x,o) \text{ for } n>0\Big)\nonumber\\
&\geq \sum_{x \in H}\bigg[\P_x\big(\bd(X_n)_{n\in\N}\in B_1 ,d_{\skrig}(X_z,o)> d_{\skrig}(x,o) \text{ for } z<0\big)\nonumber \\
&\qquad\qquad\qquad\qquad \times \P_x\big(\bd(X_n)_{n\in\N}\in B_2 ,d_{\skrig}(X_n,o)\geq d_{\skrig}(x,o)\big) \bigg]\nonumber\\
&\geq c_0^2 \sum_{x \in H}\P_x\big(\bd(X_n)_{n\in\N}\in B_1\big)\P_x\big(\bd(X_n)_{n\in\N}\in B_2\big)\nonumber\\
&= c_0^2 \int_{B_1 \times B_2} \sum_{x\in H} \frac{\d (\nu_x \otimes \nu_x)}{\d (\nu_o \otimes \nu_o)}(\xi,\eta)\d (\nu_o \otimes \nu_o)(\xi,\eta)\nonumber\\
& = c_0^2 \int_{B_1 \times B_2} \sum_{x\in H}\exp\Big\{\limsup_{\xi_n \to \xi,\eta_n \to \eta} d_{\skrig}(\xi_n, 0)+d_{\skrig}(\eta_n, 0)-d_{\skrig}(\xi_n, x)-d_{\skrig}(\eta_n, x)\Big\}\d (\nu_o \otimes \nu_o)(\xi,\eta)\nonumber\\
& \geq c_0^2 \int_{B_1 \times B_2} \sum_{x\in H}\e^{-2(\xi,\eta)_x^{\skrig}-2\delta}\, \e^{2(\xi,\eta)_o^{\skrig}}\d (\nu_o \otimes \nu_o)(\xi,\eta)\nonumber\\
&{\geq} c_0^2 \int_{B_1 \times B_2} \sum_{x\in H}\e^{-2(\xi,\eta)_{\tilde x}^{\skrig}-2\delta-d_{\skrig}(\tilde x,x)}\,\e^{2(\xi,\eta)_o^{\skrig}}\d (\nu_o \otimes \nu_o)(\xi,\eta)\label{ast3}\\
& \geq c_0^2 \int_{B_1 \times B_2} \sum_{x\in H}\e^{-2(\xi,\eta)_{\tilde x}^{\skrig}-2\delta-R}\,\e^{2(\xi,\eta)_o^{\skrig}}\d (\nu_o \otimes \nu_o)(\xi,\eta)\nonumber\\
& \geq C \int_{B_1 \times B_2} \e^{2(\xi,\eta)_o^{\skrig}}\d (\nu_o \otimes \nu_o)(\xi,\eta)\nonumber
\end{align}
In \eqref{ast3} we chose $\tilde x$ dependent on $x,\xi,\eta$ such that $\tilde x$ is on the shortest geodesic between $o$ and $(\xi,\eta)]$ and $\lfloor d_{\skrig}(o,\tilde x)\rfloor= \lfloor d_{\skrig}(o, x)\rfloor$.
The constant $C$ comes from $c_0^2$ and the fact, that the sum converges as $g \to \infty$. Indeed, since $d_{\skrig}$ is almost additive on geodesics, $(\xi,\eta)_{\tilde x}^{\skrig}$ is of order $(\xi,\eta)_o^{\skrig}-d_{\skrig}(o,\tilde x)$, and for any $m$ there are at most linearly many $\tilde x$ in $H$ with $\lfloor d_{\skrig}(o,\tilde x)\rfloor=m$. This completes the proof of the desired lower bound. 
\end{proof}

As remarked earlier, the proof of the upper bound is much more involved.
We will split the sum in $\widehat \Q$ roughly into three terms, namely

\begin{enumerate}
\item terms for $g\in B^{\skrig}_R(o)$ for some small constant $R>0$,
\item terms for $g$ close to the geodesic between $[o,x]$, where $x$ is the closest point to $o$ on the geodesic $(\xi,\eta)$,
\item and terms for $g$ far away from $o$ and $[0,x]$.
\end{enumerate}
The reason for this splitting is as follows. 

\begin{enumerate}
\item for origins $g$ inside a small ball $B^{\skrig}_R(o)$ the restriction to staying outside $B_{d_{\skrig}(o,g)}(o)$ will not change the probabilities much. A random walk started in $g$ and converging to a part of the boundary on the opposite side of the ball is not forced to walk far away from the geodesic. 

\item The terms for $g$ close to $[o,x]$ will be the ones that lead to the exponential factor in the density. For these starting points, the geodesics from $g$ into $B_1$ and $B_2$ do not cross $B_{d_{\skrig}(o,g)}(o)$ (for very long). So restricting the probabilities of hitting $B_1$ and $B_2$ to staying outside $B_{d_{\skrig}(o,g)}(o)$, will not change their value a lot. 
\item For the remaining points, the restriction of staying outside the ball significantly reduces the probabilities. Here we will use the bound from Lemma \ref{decay greens function} to show that this part of the sum does not contribute a significant amount to $\Theta_3$. 
\end{enumerate}

\begin{prop}[Upper bound on the density]\label{upper bound density}
There exists $C>0$ such that for all sets $B_1\times B_2\subset \partial^2 G$ we have
\[\Theta_3(B_1 \times B_2)\leq C \int_{B_1\times B_2}\e^{2(\xi,\eta)_0}\d(\nu_o\otimes\nu_o)(\xi,\eta).\]
\end{prop}

\begin{proof}
We show the upper bound for a slightly different set of generating sets of the $\sigma$-algebra on $\partial^2 G$. We again denote by $\psi: G\to \H^d$ the Bonk-Schramm embedding. Note that this embedding can be extended to an injective embedding $\partial \psi: \partial G\to\partial \H^d$. For $i=1,2$, let $\xi_i \in \partial \H^d$ and $\beta_i \in [0,\delta]$ for some $\delta$ small enough. Then set $\tilde B_i$ to be the set of all $\eta\in \partial \H^d$ such that $\angle([0,\xi_i),[0,\eta))\leq \beta_i$, $B_i = \partial\psi^{-1} \tilde B_i$ and
\begin{equation}
A_i= A(\xi_i, \beta_i):=\bigcup_{\eta \in \tilde B_i}\{x \in [0,\eta)\} \quad\text{ for } i =1,2
\end{equation}
The sets $\tilde B_1\times\tilde B_2$ generate the $\sigma$-algebra on $\partial^{2}\H^d$. And since 
$\partial \psi: \partial G\to\partial \H^d$ is injective, the sets $B_1\times B_2$ generate the $\sigma$-algebra on $\partial^2 G$.

\noindent We will first consider the case that the sets $B_i$ lie roughly on opposite sides of $0$, which reduces the cases we have to consider to points $(i)$ and $(iii)$ and makes it easier to see the main arguments. The general case then follows along the same lines, one just has to be more careful about into which category the origin $g$ falls. 

\noindent So let $A_1,A_2$ be sets as defined above for $\xi_1,\xi_2$ for which $\angle([0,\xi_1),[0,\xi_2))\approx \pi$. Let $R>0$ be big enough, that we can apply Lemma \ref{decay greens function} to the origins which lie outside $B^{\skrig}_{R}(o)$. By choice of $\xi_1,\xi_2$ and the upper bound on the angles $\alpha_1,\alpha_2$ we may assume that all geodesics between points $\eta_1 \in B_1$ and $\eta_2 \in B_2$ go through the ball $B_{\lambda R}(0)$, i.e.\ for all such points $(\eta_1,\eta_2)^{\H}_0\leq \lambda R$.
By definition of $\psi$,
\[\abs{\lambda (\gamma_1,\gamma_2)_o^{\skrig} - (\psi(\gamma_1),\psi(\gamma_2))_0}\leq 3c\]
 and for $\gamma_i \in B_i$, so
\[(\gamma_1,\gamma_2)_o^{\skrig} \leq R+3c.\]
We want to show that in this case
\[\Theta_3(B_1 \times B_2) \lesssim \nu_o(B_1)\nu_o(B_1).\]
For origins $g\in B_R^{\skrig}(o)$, restricting to $\skrid$ in the definition of $\widehat \Q$ means we remove paths which lead through balls contained in $B_R^{\skrig}(o)$. Hence there exists $c_1>0$ such that for all $g\in B_R^{\skrig}(o)$
\[\frac{\P_g(\bd(X_n)_n\in B_1, (X_n)_n \notin B_{d_{\skrig}(o,g)}(o))}{\P_g(\bd(X_n)_n\in B_1)}\leq c_1.\]
The same also holds for the restriction to the inside of $B_{d_{\skrig}(o,g)}^{c}(o)$.
Furthermore, since we only consider $g$ such that $d_{\skrig}(o,g)\leq R$, there also exists $c_2>0$ such that
\[\frac{\P_g(\bd(X_n)_n\in B_1)}{\P_o(\bd(X_n)_n\in B_1)}\leq c_2.\]
Putting these two bounds together, we get that
\begin{equation}\label{inside BR}
\begin{aligned}
&\sum_{g\in B_R^{\skrig}(o)}\bar\P_g\Big(\bd(X_n)_{n\in\Z}\in B_1 \times B_2 ,d_{\skrig}(o,X_{-n})> d_{\skrig}(o,g),d_{\skrig}(o,X_n)\geq d_{\skrig}(o,g) \text{ for }n>0\Big)\\
&= \sum_{g\in B_R(0)}\P_g\Big(\bd(X_n)_{n\in\N}\in B_1 ,d_{\skrig}(o,X_n)> d_{\skrig}(o,g)\text{ for } n>0\Big)\\
&\qquad\qquad\qquad\cdot \P_g\Big(\bd(X_n)_{n\in\N}\in B_2 ,d_{\skrig}(o,X_n)\geq d_{\skrig}(o,g) \text{ for } n>0 \Big)\\
&\leq c_1^2 c_2^2 \abs{B_R(0)} \nu_o(B_1)\nu_o(B_2) 
\end{aligned}
\end{equation}
For $g \notin B_R^{\skrig}(o)$ we will estimate
\[\frac{\P_g(\bd(X_n)_n\in B_1, (X_n)_n \in B_{d_{\skrig}(o,g)}^{c}(o))}{\P_g(\bd(X_n)_n\in B_1)}\]
by using the bound shown in Lemma \ref{decay greens function}. So let $H_1,...,H_M$ be the boundaries as constructed in Lemma \ref{decay greens function}. In the proof of that lemma we showed that the last boundary $H_M$ may be chosen freely. So let 
\[H_M = \bigcup_{\eta} \{x \in G: d_{\skrig}(x,o)\geq n \text{ and } d_{\H}(\psi(x),[0,\eta))<C_0\}\]
where $\psi$ is the Bonk-Schramm embedding, 
and the union is over all rays $\eta$ such that $\angle([0,\eta),[0, \xi_1))=\alpha_1+\delta$ for some very small $\delta>0$. The random walk started at $g$ and converging to a point in $B_1$ and restricted to staying outside $B_{d_{\skrig}(o,g)}^{\skrig}(o)$, has to hit $H_M$ at some point.
This gives us the following bound:
\begin{equation}\label{change from restriction}
\begin{aligned}
&\frac{\P_g\big(\bd(X_n)_n\in B_1, (X_n)_n \notin B^{\skrig}_{d_{\skrig}(o,g)}(0)\big)}{\P_g\big(\bd(X_n)_n\in B_1\big)}& \\
&\leq \sum_{g'\in H_M}\frac{\P_{g'}\big(\bd(X_n)_n\in B_1, (X_n)_n \in B_{d(0,g)}^{c}(0)\big)\,\P_g\big(g' \text{ is hitting point of }(X_n) \text{ in }H_M\big)}{\sum_{\tilde g \in H_{M}(g)}\P_{\tilde g}\big(\bd(X_n)_n \in B_1, (X_n)_n \in B_{d(0,\tilde g)}^c(0)\big)\,\P_g\big(X_n \text{ hits }\tilde g\big)}\\
&\leq \sum_{g'\in H_M}\frac{\P_{g'}\big(\bd(X_n)_n\in B_1, (X_n)_n \in B_{d_{\skrig}(o,g)}^{c}(0)\big) \,\skrig(g,g';B_{d_{\skrig}(o,g)}^c(0))}{\sum_{\tilde g \in H_{M}(g)}\P_{\tilde g}\big(\bd(X_n)_n \in B_1, (X_n)_n \in B_{d(0,\tilde g)}^c(0)\big)\, F(g,\tilde g)}\\
&\leq \sum_{g'\in H_M}\frac{\exp\big\{- rn k \, \e^{\varepsilon n}\big\}\, \skrig(g',g'_n)}{\sum_{\tilde g \in H_{M}(g)} \skrig(g,\tilde g_n)\,\skrig(\tilde g, \tilde g_n)\,\skrig(e,e)}\\
&\leq \exp\big\{- rn k \,\e^{\varepsilon n}\big\} \, \skrig(e,e)^{-1}\sum_{m \geq d_{\skrig}(o,g)} \frac{n(m)}{N(m) + n(m)} \frac{1}{\inf_{\tilde g}\skrig(g,\tilde g_n)}\\
&\leq c_3 \exp\big\{- rn k\, \e^{\varepsilon n}\big\}
\end{aligned}
\end{equation}
for some appropriate choice of $c_3>0$, since $n(m)\leq C_1$ and $N(m)\geq C_2 \e^{vm}$ for some constants $C_1,C_2>0$ as shown in Lemma \ref{number hitting points}. The same arguments lead to the corresponding bounds for $B_2$.
With this we get that
\begin{equation}\label{outside BR}
\begin{aligned}
& \sum_{g\notin B_R^{\skrig}(0)}\P_g\big(\bd(X_n)_{n\in\N}\in B_1 ,(X_n)_{n\in\N}\in B_{d_{\skrig}(o,g)-1}^c(0)\big)\P_g\big(\bd(X_n)_{n\in\N}\in B_2 ,(X_n)_{n\in\N}\in B_{d_{\skrig}(o,g)}^c(0)\big)\\
& \leq \sum_{m\geq R} \sum_{g: d_{\skrig}(o,g)=m} C \exp\big\{- rn \beta \e^{\varepsilon n}\big\}\P_{g}\big(\bd(X_n)_n \in B_1\big)\P_{g}\big(\bd(X_n)_n \in B_2\big)\\
& \leq \sum_{m\geq R} \sum_{g: d(0,g)=m} C \exp\big\{- rn \beta\,  \e^{\varepsilon n}\big\} \int_{B_1\times B_2} \e^{\limsup_{\xi_n \to \xi, \eta_n \to \eta} d_{\skrig}(o,\xi_n)+ d_{\skrig}(o,\eta_n)-d_{\skrig}(g,\xi_n)- d_{\skrig}(g,\eta_n)}\d (\nu_o \otimes \nu_o)\\
& \leq \sum_{m\geq R} \abs{\S_m(0)} \, c_3\, \exp\big\{- rm \beta \, \e^{\varepsilon m} + 2m \big\} \nu_o(B_1)\nu_o(B_2)\\
&\leq c_4 \nu_o(B_1)\nu_o(B_2)
\end{aligned}
\end{equation}
where we used that the $\limsup$ is bounded above by $2 d_{\skrig}(o,g)$ and the sum over $m>R$ converges due to the double exponential decay and $c_4>0$ absorbs all those constants.

\noindent Putting the bounds \eqref{inside BR} and \eqref{outside BR} together we thus get
\[
\widehat \Q(\bd^{-1}(B_1 \times B_2))\leq (c_1^2 c_2^2 \abs{B_R(0)}+c_4 )\nu_o(B_1)\nu_o(B_2) = C \nu_o(B_1)\nu_o(B_2)
\]
meaning we have shown the claim for $B_1$ and $B_2$ lying roughly opposite in $\H^d$. 

Now, on to the general case. We still consider $A_i,B_i$ as above for $i=1,2$, i.e. we let $\xi_i \in \partial \H^d$ and $\beta_i \in [0,\delta]$ for some $\delta$ small enough, and set $\tilde B_i$ to be the set of all $\eta\in \partial \H^d$ such that $\angle([0,\xi_i),[0,\eta))\leq \beta_i$, $B_i = \partial\psi^{-1} \tilde B_i$, and $A_i:= \bigcup_{\eta \in \tilde B_i}\{x \in [0,\eta)\}$. But now we allow the angle $\beta= \angle([0,\xi_1),[0,\xi_2))$ to be arbitrary. Further, we may assume that there exist $\xi_1^G,\xi_2^G$ such that $\xi_i = \partial \psi(\xi^G_i)$ for $i=1,2$.

Let $[0,\eta)$ be the geodesic which contains the point $x_m \in (\xi_1,\xi_2)$ closest to $0$. By choice of $\tilde B_1,\tilde B_2$, there then exists $R_1>0$ such that for any pair of points $\eta_i \in \tilde B_i$ for $i=1,2$ the point closest to $0$ on $(\eta_1,\eta_2)$ is in $B_{R_1}(x_m)$.

Further, by the properties of the embedding, there are at most $c_R$ many points in the image of $G$ inside $B_{R_1}(x_m)$.

\noindent We now split the sum over the origins $g$ in $\widehat \Q$ into the following parts.
\begin{itemize}
\item $D_1 := \big\{g \in G: d(\psi(g),0)\leq (\xi_1,\xi_2)_0^{\H^d}+R_1+R_2,\, \angle\big([0,\psi(g)],[0,\eta)\big) \leq \e^{-(\xi^G_1,\xi^G_2)_o^{\skrig} - R_1}\big\}$
\item $D_2 :=\big \{g \in G: d(\psi(g),0)\geq (\xi_1,\xi_2)_0^{\H^d}+R_1+R_2, \,\angle\big([0,\psi(g)],[0,\eta)\big) \leq \e^{-(\xi^G_1,\xi^G_2)_o^{\skrig}-R_1}\big\}$
\item $D_3 := B^{\skrig}_{R_2}(0)\backslash D_1$
\item $D_4 := \big\{g \in G: d(\psi(g),0)\leq (\xi_1,\xi_2)_0^{\H^d}+R_1+R_2, \,\angle\big([0,\psi(g)],[0,\eta)\big) \geq \max\{\gamma, \e^{-(\xi^G_1,\xi^G_2)_o^{\skrig}-R_1}\}\big\}$
\item $\begin{aligned}[t] D_5 := \big\{g \in G:\,& d(\psi(g),0)\geq (\xi_1,\xi_2)_0^{\H^d}+R_1+R_2, \\ 
&\e^{-(\xi^G_1,\xi^G_2)_o^{\skrig}-R_1} \leq \angle\big([0,\psi(g)],[0,\eta)\big) \leq \max\{\gamma, \e^{-(\xi^G_1,\xi^G_2)_o^{\skrig}-R_1}\}\big\}\end{aligned}$
\end{itemize}
Then $D_1$ contains the points close to $[0,x_m]$, $D_2$ the region between $A_1$ and $A_2$ and far away from $0$. $D_3$ consists of the points close to $0$, which are not also in $D_1$. And finally, $D_4$ and $D_5$ consists of  the points which do not lie between $A_1$ and $A_2$. Thus  for these the angle to $[0,\eta)$ is big. Note that splitting this last region into the sets $D_4$ and $D_5$ is necessary, since it grows linearly in $(\xi^G_1,\xi^G_2)_o^{\skrig}$. We hence split it into $D_4$, a set whose size is uniformly bounded, and $D_5$, the problematic region. Showing that the bound for the sum over origins in $D_5$ is small enough is the most involved and requires us to carefully consider how big the angle $\angle([0,\psi(g), [0,\eta))$ actually is.

\noindent $D_1$ is the area, where we see the exponential behavior of the density. 
We first note that 
\begin{align*}
D_1 &= \{g \in G: d(\psi(g),0)\leq (\xi_1,\xi_2)_0^{\H^d}+R_1+R_2, \angle\big([0,\psi(g)],[0,\eta)\big) \leq \e^{-(\xi^G_1,\xi^G_2)_o^{\skrig}-R_1}\}\\
&\subseteq \{g \in G: d_{\skrig}(g,0)\leq (\xi^G_1,\xi^G_2)_o^{\skrig}+R_1+R_2+3\lambda^{-1}c, \angle\big([0,\psi(g)],[0,\eta)\big) \leq \e^{-(\xi^G_1,\xi^G_2)_o^{\skrig}-R_1}\}.
\end{align*}
Finally, using the same counting argument as in Lemma \ref{number hitting points}, we have
\[\abs{D_1}\leq \tilde h_1 \e^{-(d-1)((\xi^G_1,\xi^G_2)_o^{\skrig}+R_1)} \e^{(\xi^G_1,\xi^G_2)_o^{\skrig}+R_1+R_2+3\lambda^{-1}c} \leq h'_1\]
Recall that the volume growth with respect to $d_{\skrig}$ is $1$. The upper bound by a constant $h'_1>0$ independent of $\xi, \eta$ comes from the fact that for non-elementary hyperbolic groups the dimension of $\H^d$ in the embedding is $d \geq 2$. Hence
\begin{equation}\label{H1}
\begin{aligned}
&\sum_{g\in D_1}\P_g\big(\bd(X_n)_{n\in\N}\in B_1 ,d_{\skrig}(X_n,o)> d_{\skrig}(o,g)\text{ for }n\geq 1\big)\\
&\quad \quad \cdot \P_g\big(\bd(X_n)_{n\in\N}\in B_2 ,d_{\skrig}(X_n,o)\geq d_{\skrig}(o,g)\text{ for }n\geq 1\big)\\
\leq& \int_{B_1 \times B_2} \sum_{g\in D_1} \exp\big\{2(\xi,\eta)_o^{\skrig}-2(\xi,\eta)^{\skrig}_g+4\delta\big\}\d (\nu_o \otimes \nu_o)(\xi,\eta)\\
\leq& \int_{B_1 \times B_2}h_1\, \e^{2(\xi,\eta)_o^{\skrig}} \d (\nu_o \otimes \nu_o)(\xi,\eta)\\
\end{aligned}
\end{equation}
where the $4\delta$ comes from the fact that for any fixed sequences $\xi_n \to \xi, \eta_n\to \eta$ we have 
\[\limsup_{n\to \infty} (\xi_n,\eta_n)_g^{\skrig} \geq (\xi,\eta)_g^{\skrig} - 4\delta.\]
In the last step we combined the bound on the number of summands and and the fact that, after pulling out the factor $\exp\{2 (\xi,\eta)_o^{\skrig}\}$, each summand is bounded by $\e^{4\delta}$ into the constant $h_1$.

\noindent To get a bound for the points in $D_2$, we first notice that for $\xi_1\in B_1,\xi_2\in B_2$,  $(\xi_1,\xi_2)_{x_m}^{\skrig}\leq R_1$. So, by replacing the restriction of staying outside of $B_{d_{\skrig}(o,g)}(o)$, by the images $\psi(X_n)$ staying outside $B_{d_{\H}(x_m,\psi(g))}(x_m)$, we get an upper bound for the probability. We can then immediately apply \eqref{outside BR} from the case of a bounded Gromov product between $\xi_1$ and $\xi_2$, i.e. there exists $h_2>0$ such that
\begin{equation}\label{H2}
\begin{aligned}
&\sum_{g\in D_2}\P_g\big(\bd(X_n)_{n\in\N}\in B_1 ,d_{\skrig}(X_n,o)> d_{\skrig}(o,g)\text{ for }n\geq 1\big)\\
&\quad\quad \P_g\big(\bd(X_n)_{n\in\N}\in B_2 ,d_{\skrig}(X_n,o)\geq d_{\skrig}(o,g)\text{ for }n\geq 1\big)\\
\leq& \,h_2 \,\nu_o(B_1)\nu_o(B_2).
\end{aligned}
\end{equation}

For $D_3$, we get by the same argument as used for \eqref{inside BR} in the above case that there exists $h_3>0$ such that
\begin{equation}\label{H3}
\begin{aligned}
&\sum_{g\in D_3}\P_g\big(\bd(X_n)_{n\in\N}\in B_1 ,d_{\skrig}(X_n,o)> d_{\skrig}(o,g)\text{ for }n\geq 1\big)\\
&\quad\quad\P_g\big(\bd(X_n)_{n\in\N}\in B_2 ,d_{\skrig}(X_n,o)\geq d_{\skrig}(o,g)\text{ for }n\geq 1\big)\\
\leq& \,h_3\, \nu_o(B_1)\nu_o(B_2).
\end{aligned}
\end{equation}

Now for the remaining points we will again apply the bounds for the probability of walking a certain angle around the ball proven in \eqref{decay greens function}. If we however immediately apply this bound to all remaining points, by saying that the random walk has to at least cross the angle between the sets $A_1$ and $A_2$, this will lead to a multiplicative constant of order $(\xi^G_1,\xi^G_2)_o^{\skrig}$, i.e. a constant which is not uniformly bounded over $\partial^2 G$. We hence choose some angle $\gamma \in [0, \pi]$ such that $\sin(\gamma)$ is still sufficiently close to $\gamma$. If $(\xi^G_1,\xi^G_2)_o^{\skrig} \leq \log(\nicefrac{2}{\gamma})$, we get by the same arguments as in \eqref{outside BR}, that there exists $h'_4>0$ such that

\begin{equation}\label{H4}
\begin{aligned}
&\sum_{g\in D_4}\P_g\big(\bd(X_n)_{n\in\N}\in B_1 ,d_{\skrig}(X_n,o)> d_{\skrig}(o,g)\text{ for }n\geq 1\big)\\
&\quad\quad\P_g\big(\bd(X_n)_{n\in\N}\in B_2 ,d_{\skrig}(X_n,o)\geq d_{\skrig}(o,g)\text{ for }n\geq 1\big)\\
\leq& h'_4 \nu_o(B_1)\nu_o(B_2).
\end{aligned}
\end{equation}
for $D_4 = \{g \in G: \angle\big([0,\psi(g)],[0,\eta)\big) \geq \gamma\}$.
The constant $h'_4$ however comes from summing over $\exp\{- r d_{\skrig}(o,g) \gamma \e{\varepsilon d_{\skrig}(o,g)}\}$, and in particular goes to $\infty$ as $(\xi^G_1,\xi^G_2)_o^{\skrig}$ does. Hence we have to bound the Gromov product in $D_4$ and consider $D_5$ separately. Again using \eqref{change from restriction} with angle $\e^{(\xi^G_1,\xi^G_2)_o^{\skrig}}$, we get that

\begin{equation}\label{H5}
\begin{aligned}
&\sum_{g\in D_5}\P_g\big(\bd(X_n)_{n\in\N}\in B_1 ,d_{\skrig}(X_n,o)> d_{\skrig}(o,g)\text{ for }n\geq 1\big)\\
&\quad\quad \P_g\big(\bd(X_n)_{n\in\N}\in B_2 ,d_{\skrig}(X_n,o)\geq d_{\skrig}(o,g)\text{ for }n\geq 1\big)\\
\leq& \e^{4\delta} \sum_{g\in D_5} \exp\Big\{-r d_{\skrig}(o,g)\, \e^{\varepsilon d_{\skrig}(o,g)-\lambda(\xi^G_1,\xi^G_2)_o^{\skrig}+R_1} +2(\xi^G_1,\xi^G_2)_o^{\skrig}-2(\xi^G_1,\xi^G_2)_g^{\skrig} \Big\}
\nu_o(B_1)\nu_o(B_2)
\end{aligned}
\end{equation}
Now, for $d_{\skrig}(o,g)\geq\frac{\lambda}{\varepsilon}(\xi^G_1,\xi^G_2)_o^{\skrig}$, this decays exponentially. And as $(\xi^G_1,\xi^G_2)_o^{\skrig}$ grows, the amount of this sum decreases. We hence get a uniform upper bound $c_5$ for this sum and only really have to consider those $g$ which lie inside $B_{\frac{\lambda}{\varepsilon}(\xi^G_1,\xi^G_2)_o^{\skrig}}(0)$. 

We first take care of the $\e^{-2(\xi_1,\xi_2)_g^{\skrig}}$ part of the sum. Let $\alpha = 2\e^{-(\xi^G_1,\xi^G_2)_o^{\skrig}}, 2k=(\gamma-\alpha)\e^{(\xi^G_1,\xi^G_2)_o^{\skrig}}$ and $A(m):= \{g\in H_5: \angle ([0,\psi(g)),[0,\eta))\in [m\alpha,(m+1)\alpha]\}$ for $m=1,...,k$

\begin{align}
\sum_{g\in D_5: \lfloor d_{\skrig}(o,g)\rfloor =n} \e^{-2(\xi_1,\xi_2)_g^{\skrig}} &= \sum_{m=1}^{k} \sum_{\substack{g\in A(m),\\ \lfloor d_{\skrig}(o,g)\rfloor =n}} \e^{-2(\xi_1,\xi_2)_g^{\skrig}} \nonumber\\
&{\leq }\sum_{m=1}^{k} \sum_{\substack{g\in A(m),\\ \lfloor d_{\skrig}(o,g)\rfloor =n}} \e^{-4 \log(2m\alpha)}\e^{-4n}\label{ast4}\\ 
&{\leq } \e^{-4n} \sum_{m=1}^{k}\Big((m+1)^{(d-1)} -m^{(d-1)}\Big)\, \alpha^{d-1}\, (2m\alpha)^{-4}\,\e^{n}\label{ast5}\\ 
&\leq 2^{-4}\e^{-3n} \alpha^{d-5} \sum_{m=1}^{k}(d-1)(m+1)^{(d-2)}\,m^{-4}\nonumber\\ 
&\leq 2^{-4}\, \e^{-3n}\, \alpha^{d-5}\, k\,(k+1)^{(d-2)}\,k^{-4}\nonumber\\ 
&\leq \gamma^{- d+5} \,2^{-4}\, \e^{-3n}\, \alpha^{d-5}\, \alpha^{-1}\, (\alpha^{-1}+1)^{(d-2)}\,\alpha^{4}\nonumber\\ 
&\leq c_6\, \e^{-3n}\nonumber
\end{align}
for some $c_6>0$ large enough. In \eqref{ast5} we used the estimate from Lemma \ref{number hitting points}. 
And in \eqref{ast4} we use that $(\xi,\eta)_g^{\skrig}$ is approximately the distance of $g$ to $(\xi,\eta)$. We then estimate this distance by the angle between $[0,\psi(g))$ and $A_1$ or $A_2$ respectively, which is given by $m\alpha$. From this angle we get the distance by remembering, that by walking away $c$ from a point at distance $n$, we at most change the angle by $e^{c/2}e^{-n}$. Plugging this in, the distance is bigger than $2\log(2\alpha + n)$.

And hence
\begin{equation}
\begin{aligned}
& \e^{4\delta} \sum_{g\in D_5, d_{\skrig}(o,g)\leq \frac{\lambda}{\varepsilon}(\xi^G_1,\xi^G_2)_o^{\skrig}} \exp\Big\{-rd_{\skrig}(o,g)\,\e^{\varepsilon d_{\skrig}(o,g)-\lambda(\xi^G_1,\xi^G_2)_o^{\skrig}+R_1}+ 2(\xi^G_1,\xi^G_2)_o^{\skrig}-2(\xi_1,\xi_2)_g^{\skrig}\Big\}\\
& =\e^{2(\xi^G_1,\xi^G_2)_o^{\skrig}} \sum_{n=R_2}^{\frac{\lambda}{\varepsilon}(\xi^G_1,\xi^G_2)_o^{\skrig}} \exp\Big\{-rn\,\e^{\varepsilon n-\lambda(\xi^G_1,\xi^G_2)_o^{\skrig}+R_1}\Big\} \sum_{g\in D_5, \lfloor d_{\skrig}(o,g)\rfloor =n}\e^{-2(\xi_1,\xi_2)_g^{\skrig}}\\
&\leq c_6 \e^{2(\xi^G_1,\xi^G_2)_o^{\skrig}} \sum_{n=R_2}^{\frac{\lambda}{\varepsilon}(\xi^G_1,\xi^G_2)_o^{\skrig}} \exp\Big\{-rn\e^{\varepsilon n-\lambda(\xi^G_1,\xi^G_2)_o^{\skrig}+R_1}\Big\} \, \e^{-3n}\\
&\leq c_7 \e^{2(\xi^G_1,\xi^G_2)_o^{\skrig}}
\end{aligned}
\end{equation}
where we bounded 
$$
\exp\bigg\{-rn\e^{\varepsilon n-\lambda(\xi^G_1,\xi^G_2)_o^{\skrig}+R_1}\bigg\}\leq 1
$$
and used that the sum over $\e^{-3n}$ converges. Setting $h_5 =c_5+c_7$, we thus get

\[\Theta(B_1 \times B_2) \leq \int_{B_1 \times B_2}\Big(h_1 +h_5 +\frac{h_2+h_3+h_4}{\e^{(\xi,\eta)_o^{\skrig}}}\Big)\,  \e^{2(\xi,\eta)_o^{\skrig}}\, \d (\nu_0 \otimes \nu_o)(\xi,\eta)\]
which is the desired bound. 

\end{proof}
Combining these arguments we can now conclude the proof of Theorem \ref{behavior Theta 3}.
\begin{proof}[{\bf Proof of Theorem \ref{behavior Theta 3}}]
In Proposition \ref{upper bound density} and Proposition \ref{lower bound density} we have shown that almost surely
\[\frac{\d \Theta_3}{\d (\nu_o\otimes \nu_o)}(\xi,\eta)\asymp \e^{2(\xi,\eta)_o^{\skrig}}.\]

The invariance of $\Theta_3$ under the $G$-action on $\partial^2 G$ follows directly from the invariance of $\Q$ under the $G$-action on $G^{\Z}$. Indeed, since $\Q$ is invariant for both the $G$ and $\Z$-action, and since these two actions commute, $\widehat \Q$ is invariant under the induced $G$-action on the fundamental domain $\skrid$. And since the left multiplication by group elements commutes with the boundary map $\bd: G^{\Z}\to \partial^2 G$, we have for any $g\in G$ and $B_1\times B_2\subset \partial^2 G$
\begin{align*}
\Theta_3\big(g B_1\times g B_2\big) &= \widehat \Q\big( \bd^{-1}(g B_1\times g B_2)\big) =\widehat \Q\big(g \bd^{-1}(B_1\times B_2)\big)\\
&= \widehat \Q\big(\bd^{-1}(B_1\times B_2)\big)=\Theta_3\big(B_1\times B_2\big) 
\end{align*}
Thus $\Theta_3$ is invariant for the $G$-action on $\partial^2 G$.
\end{proof}

\section{Double-Ergodicity, Exponential Mixing and CLT}\label{sec ergodicity}
\subsection{Double-ergodicity.}\label{sec proof ergodicity} 
To conclude the proof of Theorem \ref{theorem 1}, we now show that $\Theta_3$, and hence $\Theta$, is ergodic for the $G$-action on $\partial^2 G$. In order to achieve this, we first prove Theorem \ref{big ergodicity intro}, namely that the measures $\bar\P_{o}$ and $\widehat \Q$ are ergodic for the random walk flow and the $G$-shift on the paths respectively. The ergodicity of $\Theta_3$, and hence $\Theta$, then follows as a corollary. As discussed towards the end of Section \ref{sec ergodicity} and Remark \ref{comparison versions Theta}, the construction and properties of the measure $\Theta_3$ developed in Section \ref{sec Theta push-forward} will play a very important role for the purposes of the current section. 
 
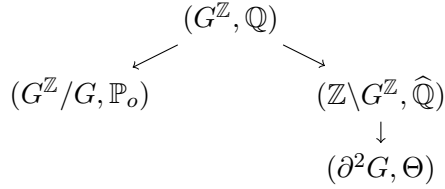
\begin{figure}[h!]
  \centering

    \begin{tikzpicture}
 \node (A) at (0,1) {$(G^{\Z},\Q)$};
 \node (B) at (-2,0) {$(G^{\Z}/G, \P_o)$};
 \node (C) at (2,0) {$(\Z\backslash G^{\Z}, \widehat \Q)$};
 \node (D) at (2,-1) {$(\partial^2 G, \Theta)$};
 
 \draw[->] (A) -- (B);
 \draw[->] (A) -- (C);
 \draw[->] (C) -- (D);
\end{tikzpicture}
\caption{Showing ergodicity in $\partial^2 G$ requiring a detour via the $G$ and $\Z$-quotients of $G^{\Z}$.}
\end{figure}

\begin{theorem}\label{big ergodicity}
The $G$-action on $\Z \backslash G^{\Z}$ and $\Z$-action on $G^{\Z}/G$ are ergodic for $\bar\P_{o}$ and $\widehat \Q$. Furthermore, $G^{\Z} = G^{\Z}_{0} \cross G$ is ergodic for $\Q$ under the product of the $\Z$-action on $G_{0}^{\Z}$ with the $G$-action on $G$. 
\end{theorem}

Key to proving this Theorem will be the following Birkhoff-type ergodic theorem for the $\Z$-action on $G_{0}^{\Z}$ with respect to the measure $\bar\P_{o}$.

\begin{theorem}
Let $f$ be a continuous function on $G^{\Z}_{0}$ which is integrable with respect to $\bar\P_{o}$. Then
\[\lim_{n\to\infty}\frac{1}{n} \sum_{l=0}^{n} f(\tau_l (x_z)_{z\in\Z})= \int_{G^{\Z}_{0}} f \d \bar\P_{o}\]
almost surely.
\end{theorem}

\begin{proof}
By construction, $\bar\P_{o}$ is invariant under the induced $\Z$-action by $\{\tau_z: z\in \Z\}$. Hence by Birkhoff's ergodic theorem 
\[\lim_{n\to\infty}\frac{1}{n} \sum_{l=0}^{n} f(\tau_l (x_z)_{z\in\Z}) = \E_{0}[f \vert \skrif^{\tau_{1}}](x_z)_{z\in \Z}\]
almost surely, where $\skrif_{\tau_{1}}$ is the $\tau_{1}$-invariant sub-$\sigma$-algebra of the product-$\sigma$-algebra on $G^{\Z}_{0}$. But in the same way also
\[\lim_{n\to -\infty}\frac{1}{-n} \sum_{l=0}^{n} f(\tau_l (x_z)_{z\in\Z}) = \E_{0}[f \vert \skrif^{\tau_{-1}}](x_z)_{z\in\Z}\]
almost surely, where we're conditioning on the $\tau_{-1}$-invariant $\sigma$-algebra. But since $\tau_{1}$ and $\tau_{-1}$ are inverse to each other, $\skrif_{\tau_{1}}=\skrif_{\tau_{-1}}$ and hence the limits agree almost surely as well. We denote this limit by $f_{\infty}$.

What is left to do , is to show that $f_{\infty}$ is indeed constant. We first regard the case that $f$ is a simple function over cylinder sets. Notice that by the above argument, $f_{\infty}(x_z)_{z\in\Z} =f_{\infty}(x_{-z})_{z\in\Z}$, so $f_{\infty}$ is invariant under time-reversal. Furthermore, $f$ only depends on finitely many indices. 
This also means that $f_{\infty}(x_z)_{z\in\Z}$ only depends on the positive indices of $(x_z)_{z\in\Z}$, since at some point $\tau_l(x_z)_{z\in\Z}$ has moved $x_0$ so far back that $f(\tau_l(x_z)_{z\in\Z})$ does not depend on $(x_{-n})_{n\in\N}$ anymore. But by invariance of $f_{\infty}$ under time-reversal, this also means that $f_{\infty}$ does not depend on the positive indices of $(x_z)_{z\in\Z}$. And hence, $f_{\infty}$ is constant almost surely. The identity
\[f_{\infty} = \int_{G^{\Z}_0}f \d \P_{o}\]
then also follows from Birkhoff's ergodic theorem.

For a general continuous function $f$ note that we can approximate $f$ arbitrarily well by a simple function over cylinder sets. The claim then follows by dominated convergence.
\end{proof}

\begin{proof}[{\bf Proof of Theorem \ref{big ergodicity}}]
We will first consider the $\Z$-action. Let $A$ be invariant under the $\Z$-action and assume that $\bar\P_{o}(A)>0$. Then almost surely
\[\lim_{n\to\infty}\frac{1}{n} \sum_{l=0}^{n} \mathbbm{1}_{A}\big(\tau_l (x_z)_{z\in\Z}\big)= \P_{o}(A).\]
But since $\tau_l (x_z)_{z\in\Z} \in A \Leftrightarrow (x_z)_{z\in\Z}\in A$ by invariance of $A$, this means for any choice of sequence $(x_z)_{z\in\Z}$ the summands are either always $0$ or always $1$. And since the series converges almost surely and $\bar\P(A)>0$, this means $\mathbbm{1}_{A}((x_z)_{z\in\Z})=1$ on an almost sure set, meaning in particular that $\bar\P(A)=1$.

Since the $G$-action and the $\Z$-action on $G^{\Z}$ commute, this in particular also means that the induced $G$-action on $\Z \backslash G^{\Z}$ is ergodic. Indeed, assume that there exists a non-trivial set $C$ in our fundamental domain of $\Z \backslash G^{\Z}$ which is $G$-invariant. Then $\Z C\in G^{\Z}$ is both $G$-invariant by assumption and $\Z$ invariant by construction. Further, it is non-trivial in $\Q$, since $\Q(\Z C)= \sum_{z\in\Z}\Q(z.C)$. And since $\Q(C)$ is non trivial, by invariance of $\Q$ under the $\Z$-action, this also means $\Q(\Z C)$ is non-trivial. Further, by construction, $\Q=\sum_{g}\bar\P_g$. This implies in particular that $\bar\P_g(\Z C)$ is non trivial for one, and hence by $G$-invariance for all $g\in G$. In particular, $\bar\P_{o}(\Z C)$ is non trivial. But since this is the measure of the projection of $\Z C$ into the $G$-quotient, this means that we have found a $\Z$ invariant set in the quotient, which is non trivial, contradicting the ergodicity of the $\Z$-action in the quotient. 
\end{proof}

\begin{cor}
$\Theta_3$ and hence $\Theta$ is ergodic for the $G$-action on $\partial^2 G$.
\end{cor}

\begin{proof}
For the boundary map $\bd: G^{\Z}\to\partial^2G$ we have $\bd(g(x_n)_n)= g(\bd(x_n)_n)$, i.e. it commutes with the $G$-actions on the two spaces. Further the boundary map does not depend on the parametrization of the sequence $(x_n)_n$. Hence, the same also holds for the restriction of the boundary map to $\Z \backslash G^{\Z}$ with the respective $G$-action.

 So for any $G$-invariant $A\subset\partial^{2}G$, $\bd^{-1}(A)$ is also $G$-invariant and hence by Theorem \ref{big ergodicity}
\[\Theta_3(A)= \widehat \Q(\bd^{-1}(A))=0\]
or
\[\Theta_3(A^c)= \widehat \Q(\bd^{-1}(A^c))=\widehat \Q(\bd^{-1}(A)^c)= 0.\]
Thus, $\Theta_3$ is ergodic for the $G$-action on $\partial^2 G$. And by Remark \ref{rem ergodicity}, this means $\Theta$ is ergodic for the $G$-action on $\partial^2 G$ as well. 
\end{proof}

\subsection{Exponential Mixing and the CLT}\label{sec mixing}

In this section we will prove Theorem \ref{thm mixing} and Theorem \ref{thm CLT}.

\noindent{\bf Proof of Theorem \ref{thm mixing}.} We will first prove that the $\Z$ action on $(G^\Z_o,\bar\P_o)$ is mixing of all orders (recall Definition \ref{def Mixing}). As in the  proof of ergodicity, we will again leverage the fact that the $\Z$-action is essentially a shift action on the sequence of i.i.d increments. Indeed, let $A_1,…,A_r$ be finite unions of cylinder sets in $G^{\Z}_o$ and let $(t_i^{1})_i,…,(t_i^{r})_i$ be sequences in $\Z$ such that $t_i = \min_{l\ne k} \abs{t_i^{l}-t_{i}^k}$ diverges. Then there exists $t>0$ such that for all $\abs{t_i}>t$ fixed, for $k=1,..,r$ the supports of the cylinders $\tau_{t_i^{k}}A_k$ are disjoint. Hence, 
\begin{equation}\label{cylinder fun}
\int_{G_{o}^{\Z}} \prod_{i=1}^r \big(\mathbbm{1}_{A_i}\circ \tau_{t_i^i}\big) \d \bar\P_o = \prod_{i=1}^r \bigg(\int_{G_{o}^{\Z}} \mathbbm{1}_{A_i} \d \bar\P_o\bigg) 
\end{equation}

by independence of the increments and shift invariance of $\bar\P_o$. So for $t_i>t$, the difference is $0$.

Now suppose $r\geq 2$ and $f_1,\dots, f_r\in L^\infty(\bar\P_o)$. For these functions, the general statement can be shown by approximating these functions by simple functions defined w.r.t. cylinder sets. Indeed, for $i=1,..,r$, let $(f_i^{m})$ be a sequence of such simple functions approximating each $f_i\in L^\infty(\bar\P_o)$ as $m\to\infty$. By the above argument, we have for every fixed $m$, 
$$
\lim_{i\to \infty} \abs{\int_{G_{o}^{\Z}} f_1^m\circ \tau_{t_i^1} \cdot...\cdot f_r^m\circ \tau_{t_i^r} \d \bar\P_o - \int_{G_{o}^{\Z}} f_i^m d \bar\P_o \cdot...\cdot  \int_{G_{o}^{\Z}} f_r^m \d \bar\P_o} =0.
$$
Furthermore, for fixed $i \in \N$, we have 
\begin{align*}
&\lim_{m \to \infty}\abs{\int_{G_{o}^{\Z}} f_1^m\circ \tau_{t_i^1} \cdot...\cdot f_r^m\circ \tau_{t_i^r} d \bar\P_o - \int_{G_{o}^{\Z}} f_i^m d \bar\P_o \cdot...\cdot  \int_{G_{o}^{\Z}} f_r^m d \bar\P_o }  \\
&= \abs{\int_{G_{o}^{\Z}} f_1\circ \tau_{t_i^1} \cdot...\cdot f_r\circ \tau_{t_i^r} d \bar\P_o - \int_{G_{o}^{\Z}} f_1 d \bar\P_o \cdot...\cdot  \int_{G_{o}^{\Z}} f_r d \bar\P_o}  
\end{align*}

So using dominated convergence, 
\begin{align*}
0
= \lim_{i\to \infty} \abs{\int_{G_{o}^{\Z}} f_1\circ \tau_{t_i^1} \cdot...\cdot f_r\circ \tau_{t_i^r} d \bar\P_o - \int_{G_{o}^{\Z}} f_i d \bar\P_o \cdot...\cdot  \int_{G_{o}^{\Z}} f_r d \bar\P_o }.
\end{align*}
This proves the first statement of Theorem \ref{thm mixing} regarding mixing of all orders.

\medskip

 For Hölder continuous functions $f_i$ we want to be a bit more careful regarding how we choose the approximating functions $f_i^m$. 
We recall that we equip $X=G^\Z_o$ with the metric $d(x,y)= \sum_j 2^{-j} \max\{1,d_{\mathcal G}(x,y)\}$ which is equivalent to the following metric on $G^\Z_o$:
\begin{equation}\label{def metric2}
d^\prime(x,y)= \e^{-n(x,y)} \qquad\mbox{where}\quad n(x,y)=\max\{|k|: x_k \ne y_k\}.
\end{equation}
Both metrics induce the product topology. We also recall that 
a function $f: G^\Z_o\to \R$ is H\"older continuous if there exists $C=C_f>0$ and $\alpha=\alpha_f>0$ such that 
\begin{equation}\label{def Holder2}
x_k= y_k \quad\forall |k|\leq n \qquad\mbox{implies} \qquad |f(x) - f(y)| \leq C \e^{-\alpha n}
\end{equation}
This is equivalent to requiring that $|f(x)- f(y)| \leq C d(x,y)^\alpha$ for all $x,y\in G^\Z_o$ with $d(x,y)=\sum_j 2^{-j} \max\{1,d_{\mathcal G}(x,y)\}$. For any $n\in \N$, we define 
$$
\mathcal F_n:=\sigma(x_{-n},\cdots.x_n\big)
$$

It suffices to show the exponential mixing for two functions $f,g$. The rest follows inductively. We approximate $f, g$ by simple functions as follows.

\begin{lemma}\label{lemma step2}
If $f \in L^\infty(\bar\P_o)$ is H\"older continuous with constant $C$ and exponent $\alpha>0$, then for $f_n=\E^{\bar\P_o}[f | \mathcal F_n]$
$$
\|f - f_n\|_\infty \leq C \e^{-\alpha n} 
$$
\end{lemma}
\begin{proof}
Fix $x\in G^\Z_o$. Then the conditional expectation $f_n=\E^{\bar\P_o}[f| \mathcal F_n]$ is constant on the cylinder set $C_x^n:=\{y\in G^\Z_o: y_{-n,\dots, n}=x_{-n,\dots,n}\} \in \mathcal F_n$. For any $y\in C_x^n$, we have $n(x,y)\geq n$ (see \eqref{def metric2} above). Since $f$ is assumed to be H\"older continuous (see \eqref{def Holder2}), we have for $y\in C_x^n$, $|f(y)- f(x)| \leq C \e^{-\alpha n}$. Using these two facts we have 
$$
|f_n(x) - f(x)| = \bigg|\int_{C_{x}^n} (f(y) - f(x)) \bar\P_o(\d y | \mathcal \mathcal{F}_n)\bigg| \leq C \e^{-\alpha n}.
$$
Taking supremum over $x\in G^\Z_o$ shows $\|f - f_n\|_\infty \leq C \e^{-\alpha n}$, proving the lemma.  
\end{proof}

With this we can prove
\begin{prop}\label{prop exp mixing}
Let us assume $r=2$, and let $f, g\in L^\infty(\bar\P_o)$ be H\"older continuous functions with exponent $\alpha = \min\{\alpha_f, \alpha_g\}>0$. Then for all $s,t\in \Z$, 
$$
\bigg|\int (f\circ \tau_t) (g\circ \tau_s) \d\bar\P_o - \int f \d\bar\P_o \int g\d\bar\P_o \bigg|  \leq C(f,g) \e^{- \delta |t-s|}, \qquad\delta= \alpha/2>0.
$$
\end{prop}

\noindent{\bf Proof of Proposition \ref{prop exp mixing}.} Let 
$$
\delta_{f,g}(t-s): = \int (f\circ \tau_t)(g\circ \tau_s) \d\bar\P_o - \int f\d\bar\P_o \int g\d\bar\P_o = \int f \,\, (g\circ \tau_{t-s}) \d\bar\P_o - \int f\d\bar\P_o\int g\d\bar\P_o
$$
Also we set 
$$
f_n:= \E^{\bar\P_o}[f|\mathcal F_n], \qquad g_n:=\E^{\bar\P_o}[g|\mathcal F_n]
$$
so that $f_n$ depends on the coordinates in $[-n,n]$ and $g_n\circ\tau_{t-s}$ depends on the coordinates in $[t-s-n,t-s+n]$ so that if $\ell\geq 1$ is an integer such that $|t-s| \geq 2n + \ell$, then the two windows are separated by at least $\ell\geq 1$. Hence,  as in (\ref{cylinder fun}), 
\begin{equation}\label{decorr n}
\delta_{f_n,g_n}(t-s)=0\qquad\forall |t-s| \geq 2n + \ell.
\end{equation}
Therefore, 
$$
\begin{aligned}
|\delta_{f,g}(t-s)|&= |\delta_{f,g}(t-s) - \delta_{f_n,g_n}(t-s)| \\
&= \bigg| \int f\,\,  g\circ\tau_{t-s} \d\bar\P_o - \int f\d\bar\P_o \int g\bar\P_o - \int f_n \,\, g_n \circ\tau_{t-s}\d\bar\P_o + \int f_n \d\bar\P_o\int g_n\d\bar\P_o \bigg| \\
&\leq \bigg| \int f \,\, g\circ\tau_{t-s}\d\bar\P_o - \int f_n \,\, g_n\circ \tau_{t-s}\d\bar\P_o \bigg| + \bigg| \int f_n \d\bar\P_o\int g_n\d\bar\P_o - \int f\d\bar\P_o \int gd\bar\P_o\bigg| \\
&= (I) + (II)
\end{aligned}
$$
Now 
$$
\begin{aligned}
(I) &= \bigg| \int f \,\, g\circ\tau_{t-s}\d\bar\P_o - \int f_n \,\, g_n\circ \tau_{t-s}\d\bar\P_o \bigg| \\
&= \bigg| \int f \,\, g\circ\tau_{t-s} - \int f_n \d\bar\P_o\,\, g \circ \tau_{t-s}\d\bar\P_o  + \int f_n \,\, g\circ\tau_{t-s}\d\bar\P_o - \int f_n \,\, g_n\circ\tau_{t-s}\d\bar\P_o\bigg| \\
& \leq \|f - f_n\|_\infty \|g\|_\infty + \|f_n\|_\infty \|g-g_n\|_\infty
\end{aligned}
$$
using that $\bar\P_o$ is $\Z$ invariant. Moreover, note that, by definition and by Jensen's inequality, $\|f_n\|_\infty \leq \|f\|_\infty$ (and likewise $\|g_n\|_\infty\leq \|g\|_\infty$). We now apply Lemma \ref{lemma step2} to estimate $\|f-f_n\|_\infty$ and $\|g-g_n\|_\infty$. Likewise,
$$
\begin{aligned}
(II) &= \bigg| \int f_n\d\bar\P_o \int g_n\d\bar\P_o - \int f\d\bar\P_o \int g\d\bar\P_o\bigg| \\
&=\bigg| \int (f_n- f)\d\bar\P_o \int g_n\d\bar\P_o + \int f\d\bar\P_o \int g_n\d\bar\P_o -\int f\d\bar\P_o \int g\d\bar\P_o\bigg| \\
&\leq \|f_n - f\|_\infty \|g_n\|_\infty + \|f\|_\infty \|g_n-g\|_\infty
\end{aligned}
$$
and once more we use Lemma \ref{lemma step2} to estimate $\|f-f_n\|_\infty$ and $\|g-g_n\|_\infty$ and invoke $\|g_n\|_\infty\leq \|g\|_\infty$. So choosing 
$$
2n:= |t-s| - \ell -1, \qquad \delta:= \frac \alpha 2. 
$$
proves Proposition \ref{prop exp mixing}. Therefore, Theorem \ref{thm mixing} is also proved now. \qed

\medskip

\noindent{\bf Proof of Theorem \ref{thm CLT}.} We note that it is equivalent to look at the space of increments $(S^{\Z}, \mu^{\otimes \Z})$ instead of $(G^{\Z}_o,\P_o)$, and Hölder continuous functions $f:S^{\Z}\to \R$ (recall that $S$ is the finite symmetric generating set of $G$ and $\mu$ is the step distribution which is supported on $S$).  

Assume $f$ is of the form $f\big((s_n)_{n\in \Z}\big) = f_m\big((s_n)_{n=-\infty}^m\big)$ for some $m$ fixed. Then by \cite[Theorem 3 (i)]{Wu05}, for every $m$, 

\begin{equation}\label{CLT m}
\frac{\sum_{i=1}^{n}f_m\circ\tau_i ((S_n)_{n}) - n \E^{\bar\P_o}[f_m(S_n)_n]}{\sqrt{n}} \to \mathcal{N}(0,\sigma_m^2)
\end{equation}
in distribution as $n\to\infty$, where 
\[\sigma_m^2= \norm{\sum_{n \in \N }\E^{\bar\P_o}\bigg[f_m\circ\tau_n(S_z)_z\big\vert \sigma(S_i: i\leq 0)\bigg] - \E^{\bar\P_o}\bigg[f_m\circ\tau_n(S_n)_n\vert \sigma(S_i: i< 0)\bigg]}_2^2 < \infty.\]

\noindent Indeed, since $f_m$ is Hölder continuous, for every $m$,
\[w^m(n):=\norm{\E^{\bar\P_o}[f_m\circ\tau_n(S_n)_n\vert \sigma(S_i: i\leq 0)] - \E^{\bar\P_o}[f_m\circ\tau_n(S_n)_n\vert \sigma(S_i: i< 0)]}_2\leq C_{f_m} \e^{-\alpha_{f_m} n}\]
and hence $\sum_{n\in \N}w^m(n)<\infty$ and so is $\sigma_m$. For general Hölder continuous functions $f$, we again approximate $f$ by $f_m := \E[f\vert \mathcal F_m]$ (recall Lemma \ref{lemma step2}). Note that,  
\[\abs{\sum_{i=1}^{n}f_m\circ\tau_i ((S_n)_{n}) - n \E^{\bar\P_o}[f_m(S_n)_n] - \sum_{i=1}^{n}f\circ\tau_i ((S_n)_{n}) - n \E^{\bar\P_o}[f(S_n)_n]} \leq  2C_f n \e^{-\alpha_f m} \]
which decays in $n$ if we choose for example $m=n$. Thus, by \eqref{CLT m}, we conclude that the central limit theorem also holds for $f$ by Slutsky's theorem, using also that by dominated convergence $\sigma_m^2\to \sigma^2$ as $m \to \infty$.
\qed

\appendix

\section{}\label{appendix}

\subsubsection{\bf Relations and Differences with Some Percolation Models in Statistical Mechanics}
A different, systematic way of sampling bi-infinite random walk paths was uncovered in a seminal work of Sznitman \cite{szn-annals} on random interlacements in the context of $\Z^d$, and generalized to transient weighted graphs $\Gamma$ by Teixeira \cite{teix}. 
We will now  underline the differences in the two viewpoints in the discussion below and in Example \ref{example RI} we will discuss a concrete case to demonstrate these differences. 

We restrict the construction in \cite{teix} to the test case that $\Gamma$ is the Cayley graph of a finitely generated group with respect to a generating set $S$ and we are regarding the random walk on $\Gamma$ with increments with symmetric law $\mu$ supported on $S$.

Let $L$ be the Cayley graph of $\Z$ with respect to the generating set $\{-1,1\}$ and for $-\infty\leq n\leq m\leq \infty$, denote by $L(n,m)$ the restriction of $L$ to the vertices $\{n,...,m\}$. Further, we denote by $\mathcal{W}(n,m)$ the set of all graph homomorphism $L(n,m) \to \Gamma$ such that the pre-image of each vertex in $\Gamma$ is finite. Finally 
$\mathcal{W} = \bigcup_{n,m}\mathcal{W}(n,m)$. 
On $\mathcal{W}$ the time shift $s_k: \mathcal{W}(n,m)\to \mathcal{W}(n-k,m-k)$ is then defined by $s_k(v_i)=v_{i-k}$ and $s_k(e(i,i+1))=e(i-k,i+1+k)$ where the vertex $v_i$ is the image of $i$ in the graph homomorphism, and similarly $e(i,i+1)$ the image of the edge between $i$ and $i+1$. The space of trajectories is then given by $\mathcal{W}^{\ast}= \mathcal{W}/\sim$ where $w_1 \sim w_2$ if $w_1=s_k(w_2)$ for some $k\in \Z$. 

In the case that $\Gamma$ is the Cayley graph of a group $G$, $\mathcal W(-\infty,\infty)$ is a systematic way of representing the restriction of the support of $\Q$ on $G^{\Z}$ to paths which escape to $\infty$. It is however in the measures that the different points of view become apparent.

As shown in \cite{teix}, $\mathcal{W}^{\ast}$ is then equipped with a measure $\mathcal{Q}^{\ast}$, which is invariant under graph automorphisms of $G$. This measure is uniquely identified by 
\[\mathcal{Q}^{\ast}(A \cap \mathcal{W}_K)= \mathcal{Q}_K(\pi^{-1}(A))
\]
for $A \subset \mathcal{W}^{\ast}$ and $K$ a finite set of vertices in $G$, where $\pi$ is the projection $\mathcal{W}\to \mathcal{W}^{\ast}$, $\mathcal{W}_K$ are the trajectories which visit $K$, and $\mathcal{Q_K}$ is defined by
\begin{align*}
&\mathcal{Q}_K\big(w \in \mathcal{W}\,:\, w_{\vert(-\infty,0]}\in B_1,\, w_{\vert[0,\infty)}\in B_2 \big)\\
&= \sum_{u\in K} \P_u\big((X_n)_{n\in \N}\in \bar B_1, \tau_K = \infty\big)\P_u\big((X_n)_{n\in \N}\in \bar B_2\big)
\end{align*}
where $\P_u$ is the distribution of the random walk started at $u$ and $\tau_K$ being the first hitting time in $K$ after time $0$. Note that since $K$ is a finite set of vertices, this only gives us information on the behavior of the measure $\mathcal{Q}^{\ast}$ inside the graph $G$. Indeed, $\mathcal{Q}^{\ast}$ is a measure on the one-point compactification of $G$ in the following way. 

Let $V_n$ be an increasing sequence of vertices in $G$ exhausting $V$ and set $G_n$ to be the graph gained from $G$ by identifying all points in $V_n^{c}$ to one point $\partial_n$. Setting for $K\subset V_n$
\begin{align*}
&\mathcal{Q}_K^n\big(w \in \mathcal{W}\,:\, w_{\vert(-\infty,0]}\in B_1,\, w_{\vert[0,\infty)}\in B_2 \big)\\
&= 
\sum_{u\in K} \P_u\big((X_n)_{n\in \N}\in \bar B_1, \tau_K > \tau_{\partial_n}, \tau_{\partial_n}< \infty \big)\P_u\big((X_n)_{n\in \N}\in \bar B_2, \tau_{\partial_n}< \infty\big)
\end{align*}
where $\tau_{\partial_n}$ is the hitting time in $\partial_n$. After normalizing this measure by $\mathcal{Q}_K^n(\mathcal{W})$, this is the distribution of a random walk started at the "point at infinity" $\partial_n$ and killed upon returning to $\partial_n$, conditioned to hitting $K$ for the first time at time $0$. Now, we again get a measure $\mathcal{Q}^{n\ast}$ identified by
\[\mathcal{Q}^{\ast}(A \cap \mathcal{W}_K)= \mathcal{Q}_K(\pi^{-1}(A))
\]
and since $\mathcal{Q}_K^n \to \mathcal{Q}_K$ weakly as $n\to \infty$, also $\mathcal{Q}^{n\ast}\to\mathcal{Q}^{\ast}$.
See also for example \cite[Sec. 3]{hutchcroft interlacements}, for a more detailed write-up of this.

While our quotient $\Z\backslash G^{\Z}$ is, up to a set of measure zero, isomorphic to $\mathcal{W}(-\infty,\infty)$, the measure $\widehat \Q$ differs from $\mathcal{Q}$ due to our construction of the bi-infinite random walks with fixed origin $g$. \emph{Thus, it is the quotient measure spaces constructed that give rise to an essential difference between $\widehat \Q$ and $\mathcal{Q}$.}
Recall that we defined $\Q= \sum_{g\in G}\bar \P_g$ and $\widehat \Q$ as the restriction of $\Q$ to an almost sure fundamental domain $\mathcal D$ of the $\Z$-shift on $G^{\Z}$. 
In particular, for $K\subset G$ finite,
\begin{align*}
\widehat \Q((x_z)_{z\in\Z}\text{ hits }K)=
\sum_{g\in G}\bar\P_g\big((X_z)_{z\in\Z}\cap K \ne \emptyset, d_{\skrig}(o,X_{-n})>d_{\skrig}(o,g), d_{\skrig}(o,X_{n})\geq d_{\skrig}(o,g) \text{for }n>0 \big)
\end{align*}
which reduces the summands to points $g$ such that $K$ is not fully contained in the ball $\{g': d_{\skrig}(g',o)<d_{\skrig}(g,o)\}$. Note that while for $K=o$ this agrees with $\mathcal Q_K$, for any other choice of $K$ it does not, as in our definition the set $K$ can be hit by both the negative time part of $(X_z)_{z\in\Z}$ as well as by the positive time part of the path.

Note that, while both $\mathcal{Q}^{\ast}$ and $\widehat \Q$ are invariant for the $G$-action on the quotient $\Z\backslash G^{Z}$, they are so in subtly different ways. Since $\widehat \Q$ is defined on an almost sure fundamental domain of the $\Z$-action, it is invariant under the induced $G$-action on the domain. This action shifts a path $(x_z)_{z\in\Z}$ and then reparametrizes it. $\mathcal{Q}^{\ast}$ on the other hand is identified via the push-forward of measures $\mathcal{Q}_K$ on $\mathcal{W}$. In particular this means that for any $g\in G$
\[\mathcal{Q}^{\ast}(\text{trajectory hits }K)=\mathcal{Q}^{\ast}(\text{trajectory hits }gK)\]
while in general
\[\widehat\Q(\text{trajectory hits }K)\ne\widehat\Q(\text{trajectory hits }gK).\]

Finally, since  the bi-infinite random walk $(X_n)_{n\in\Z}$ converges to random points $X_{\infty}$ and $X_{-\infty}$ in $\partial G$ as $z\to \infty$ and $z\to -\infty$ respectively and by definition $\widehat \Q$ can measure the behavior of the limits on the boundary. 
In particular, our measure $\widehat \Q$ does not live on the one-point compactification of $G$ but rather in the Gromov compactification $G\cup \partial G$.

\begin{example}\label{example RI}
To further illustrate the difference between the intensity measures of the random interlacements and the measure $\widehat \Q$, we discuss this example which shows the different behavior of the measures for sets inside the group.

Let $T$ be the Cayley graph of $\mathbb{F}_2$ the free group in two generators $a$ and $b$ with respect to the symmetric edge set $S=\{a,a^{-1}, b,b^{-1}\}$. 
We consider our bi-infinite random walk paths $(X_z)_{z\in \Z}$ as described above with driving measure $\mu = \text{Unif}(S)$. Let $K=\{k\}\subset \mathbb{F}_2$ and consider the event $A=\{\text{ the bi-infinite path hits } k\}$. This set can be seen as a subset of either notion of space of bi-infinite random walk paths.

\noindent To calculate $\widehat \Q(A)$, we first notice that the restrictions $d_{\skrig}(X_z,o)> d_{\skrig}(g,o) \text{ for } z<0$ and $ d_{\skrig}(X_z,o)\geq d_{\skrig}(g,o) \text{ for } z>0$ mean that the summands in the definition of $\widehat \Q$ are only positive for origins $g\in [o,k]$, the unique geodesic between $o$ and $k$. Let
\[c^1_g = \frac{\P_g\big((X_n)_{n\in\N}\in A, d_{\skrig}(X_n,o)> d_{\skrig}(g,o) \text{ for all } n\big) }{\P_g\big((X_n)_{n\in\N}\in A\big)}\]
and 
\[c^2_g = \frac{\P_g\big((X_n)_{n\in\N}\in A, d_{\skrig}(X_n,o)\geq d_{\skrig}(g,o) \text{ for all } n\big) }{\P_g\big((X_n)_{n\in\N}\in A\big)}.\]
Note that these values actually just depend on the distance between $g$ and $k$.
Then with $\widehat \Q$ defined in \eqref{def QZ int}, 
\begin{align*}
\widehat \Q(A)&= \sum_{g\in [o,k]}\bar\P_g\big((X_z)_{z\in\Z}\in A, d_{\skrig}(X_z,o)> d_{\skrig}(g,o) \text{ for } z<0, d_{\skrig}(X_z,o)\geq d_{\skrig}(g,o) \text{ for } z>0\big)\\
&=\sum_{g\in [o,k]} (c^1_g+c^2_g) \big(2 \P_g((X_n)_{n\in\N}\in A) - \P_g((X_n)_{n\in\N}\in A)^2\big)\\
&=\sum_{g\in [o,k]} (c^1_g+c^2_g) \big(2 F(g,k) - F(g,k)^2\big).
\end{align*}
It is well known, that for the simple random walk on a $q$-regular tree, we have
\[F(x,y) = (q-1)^{-d(x,y)} \]
where $d$ is the graph metric (cf. \cite{woess}, Lemma 1.24).
In our case this leads to 

\[\widehat \Q(A) = \sum_{g\in [o,k]} (c^1_g+c^2_g) \big(2\cdot 3^{-d(g,k)}- 9^{-d(g,k)}\big) \] 
i.e.\ it is positive and increases in the distance of $k$ to $o$.

For the intensity measure of random interlacements however
\begin{align*}
\mathcal{Q}^{\ast}(\mathcal{W}^{\ast}_K) &= \mathcal{Q}_{K}(\mathcal{W})= \P_k((X_n)_{n\in\N} \text{ does not return to }k)\\
&= 1- \sum_{g\sim k}\mu(gk^{-1}) F(k,g) = \frac{2}{3}.
\end{align*}
Which, as stated above, is constant over all choices of $k$. This demonstrates that the two measures are not equal.
\end{example}

\noindent{\bf Acknowledgements.} We are  grateful to Vadim Kaimanovich for valuable feedback and pointing out a number of references to us -- these and others are included in Section~\ref{sec-lit}. We also thank Konstantin Recke and Eduardo Silva for very helpful  discussions on the first draft of the article. 
The research of the first and the third author is funded by the Deutsche Forschungsgemeinschaft (DFG) under Germany’s Excellence Strategy EXC 2044-390685587, Mathematics M\"unster: Dynamics-Geometry- Structure. The second author is partly supported by a grant from the TNQ foundation under the "Numbers and Shapes" initiative, an  endowment of the Infosys Foundation, a DST JC Bose Fellowship,  and by  the Department of Atomic Energy, Government of India, under project no.12-R\&D-TFR-5.01-0500. A part of this work was done when the second author was visiting the first and the third authors at University of M\"unster, and another part when all three authors were visiting Fields Institute, Toronto, during a special semester on Randomness and Geometry in 2024. We thank these institutions for their support.

\end{document}